\documentclass[hidelinks,onefignum,onetabnum]{siamart251216}
\usepackage{amssymb, amsmath}
\usepackage{mathtools} % for \coloneq
\usepackage{tikz}
\usepackage{xcolor}
\usetikzlibrary{trees}
\usepackage{enumitem}
\usetikzlibrary{arrows.meta, positioning, shapes.geometric}
\usetikzlibrary{decorations.pathmorphing}
\usetikzlibrary{shapes,shadows,fit} % fit is needed for cloud
\usetikzlibrary{calc}
\usepackage[utf8]{inputenc}
\usepackage{graphicx} % Required for \includegraphics
\usepackage{caption}  % Required for \captionof, if you don't use subfigure/subfloat
\usepackage{float}
\usepackage{booktabs}
\usepackage{subcaption} % Recommended for subfigures (more modern than subfigure)
\usepackage{makecell} % for \Xhline
\usepackage{cleveref}
\usepackage{diagbox}
\usepackage{siunitx}   % For number alignment (optional)
\usepackage{caption}
\usepackage{multirow, booktabs}

\usetikzlibrary{arrows.meta, positioning}
\usetikzlibrary{decorations.pathmorphing, decorations.text}
\usetikzlibrary{decorations.pathreplacing, positioning}
\usepackage{lipsum}
\usepackage{amsmath}
\usepackage{amsfonts}
\usepackage{graphicx}
\usepackage{epstopdf}
\usepackage{algorithmic}
\ifpdf
  \DeclareGraphicsExtensions{.eps,.pdf,.png,.jpg}
\else
  \DeclareGraphicsExtensions{.eps}
\fi

\newsiamremark{remark}{Remark}
\newsiamremark{example}{Example}
\newsiamremark{assumption}{Assumption}
\crefname{assumption}{Assumption}{Assumption}
\newsiamthm{claim}{Claim}
\newsiamremark{fact}{Fact}
\crefname{fact}{Fact}{Facts}

\headers{DDIR: denoiser driven iterative regularization}{H. Bajpai, A. K. Giri, T. Jahn and A. Jha}

\title{On the convergence of an adaptive denoiser driven iterative regularization with early stopping \thanks{Submitted to the editors DATE.
\funding{Under project number 57762238 this research received financial support from  DST, India and DAAD, Germany as a part of project-related Indo-German person exchange program.  The first author gratefully acknowledges the financial assistance provided by the Ministry of Education, Government of India, and the Indian Institute of Technology Roorkee through a Ph.D. fellowship supporting this research. The third author gratefully acknowledges funding by the Deutsche Forschungsgemeinschaft (DFG, German Research
Foundation) under Germany´s Excellence Strategy – The Berlin Mathematics
Research Center MATH+ (EXC-2046/1, EXC-2046/2, project ID: 390685689).
The fourth author gratefully acknowledges funding by the Indian Institute of Technology, Gandhinagar, through the grant IP/IP/52016.}}}

\author{
Harshit Bajpai%
\thanks{Department of Mathematics, Indian Institute of Technology Roorkee, Roorkee 247667, India
(\email{harshit\_b@ma.iitr.ac.in, bajpaiharshit87@gmail.com}).}
\and
Ankik Kumar Giri%
\thanks{Department of Mathematics, Indian Institute of Technology Roorkee, Roorkee 247667, India
(\email{ankik.giri@ma.iitr.ac.in}).}
\and 
Tim Jahn
\thanks{Institut für Mathematik, Technische Universität Berlin, Berlin, Germany 
(\email{jahn@tu-berlin.de}).}
\and
Abhinav Jha%
\thanks{Department of Mathematics, Indian Institute of Technology Gandhinagar, Palaj 382055, India
(\email{abhinav.jha@iitgn.ac.in}).}
}
\usepackage{amsopn}

\ifpdf
\hypersetup{
  pdftitle={On the convergence of an adaptive denoiser driven iterative regularization with early stopping},
  pdfauthor={Harshit Bajpai, Ankik Kumar Giri, Tim Jahn and Abhinav Jha }
}
\fi

\def\d{\delta}

\def\d{\delta}

\begin{document}

\maketitle
% REQUIRED
\begin{abstract}
Solving inverse problems requires appropriate regularization techniques to ensure well-posedness and stability. In recent years, denoiser-driven methods have emerged as effective regularization strategies, achieving state-of-the-art performance in various imaging applications. However, their stability and convergence within iterative regularization frameworks remain largely unexplored.
% a rigorous theoretical understanding of their stability and convergence within the framework of iterative regularization remains largely unexplored.
In this work, we extend the framework of Regularization by Denoising (RED) by introducing a novel denoiser-driven iterative regularization scheme, referred to as \texttt{DDIR}, that incorporates a new regularization functional based on averaged denoisers. The proposed approach employs an adaptive step-size strategy together with an \emph{a posteriori} stopping rule to ensure stability while alleviating oscillatory behavior and semi-convergence effects induced by noise. As our main theoretical contribution, we prove that the resulting reconstruction method constitutes a stable and convergent regularization scheme in the classical sense. To the best of our knowledge, this provides the first rigorous justification of \texttt{DDIR} within the framework of regularization theory.
Finally, we demonstrate the performance of the proposed method through numerical experiments on image deblurring and phase retrieval Computed Tomography (CT) using three denoisers, namely median, TNRD, and TV proximal. The results highlight the effectiveness of the method in terms of reconstruction accuracy and computational efficiency.
\end{abstract}

% REQUIRED

% REQUIRED

\section{Introduction}
Image reconstruction is a fundamental component of many scientific and engineering applications, including biomedical imaging, nondestructive testing, and remote sensing. These tasks are typically formulated as inverse problems, which aim to recover an unknown model parameter 
$u^\dagger \in U$ from noisy and indirect observations 
$v^\delta \in V$, modeled as
\begin{equation}
v^\delta = \mathcal G (u^\dagger) + \eta.
\label{eq:forward_model}
\end{equation}
Here, 
\(\mathcal{G} : \operatorname{dom}(\mathcal G) \subseteq U \to V\) with \(\operatorname{dom}(\mathcal G)\) denoting the domain of $\mathcal{G}$, is an operator (possibly nonlinear) between Hilbert spaces $U$ and $V$ modeling the forward problem. The spaces $U$ and $V$ are equipped with the usual inner product
$\langle \cdot, \cdot \rangle$ 
which induce the corresponding norm $\|\cdot\|$.
The term \(\eta\) represents the data perturbation and is assumed to be bounded by 
a known noise level $\delta > 0$, that is,
\(
\|\eta\| \leq \delta.
\)
In the noise-free case \(\delta = 0\), the (exact) data is denoted by $v$, which is given by \(v:= \mathcal G (u^\dagger)\).

The inverse problem in \eqref{eq:forward_model} is said to be ill-posed  if at least one of the following properties fails: injectivity or surjectivity of the forward operator, or stability of the inverse mapping. In particular, when $\mathcal{G}(\cdot)$ is a compact operator with an infinite-dimensional range, then surjectivity typically fails and the inverse mapping (if it exists) is unstable. This situation arises, for example, for the ray transform underlying several medical imaging modalities, including computed tomography (CT) and positron emission tomography (PET) \cite{natterer2001mathematics, natterer2001mathematical}. Consequently, ill-posedness is a fundamental feature of inverse problems and will be assumed throughout this work.
Therefore, to mitigate ill-posedness, regularization techniques are essential to guarantee the stability and reliability of approximate solutions. A large class of such methods falls under the category of variational regularization~\cite{scherzer2009variational}, which is commonly formulated as
\begin{equation}\label{eqn:variational}
u_\lambda^\delta 
\coloneqq 
\operatorname*{arg\,min}_{u \in \operatorname{dom}(\mathcal{G})} 
\left\{ 
\mathcal{D}\big(\mathcal{G}(u), v^\delta\big) 
+ \lambda \Theta(u) 
\right\}.
\end{equation}
Here, \(\mathcal{D}(\cdot, v^\delta)\) enforces fidelity to the observed data, while the second term \(\Theta(\cdot)\) serves as a regularization functional that restores well-posedness to the problem, and $\lambda > 0$ is a regularization parameter that trades between data fidelity and regularization. 

 A central question concerns the appropriate choice of the data fidelity term
$\mathcal{D}(\cdot, v^\delta)$ and the regularization functional $\Theta(\cdot)$ in
\eqref{eqn:variational}. A commonly adopted choice is
\[
\mathcal D(\mathcal{G}(u), v^\delta) = \tfrac{1}{2}\|\mathcal{G}(u) - v^\delta\|_2^2, 
\qquad 
\Theta(u) = \|u\|_1,
\]
where the $\ell_2$-norm corresponds to a least-squares data fidelity term, while the
$\ell_1$-norm promotes sparsity in the reconstructed solution. For further discussion and theoretical background, we refer the reader to \cite{daubechies2004iterative,engl1996regularization, scherzer2009variational}. Other prominent and emerging  choices of $\Theta(\cdot)$ include data-driven regularizers, which leverage priors learned from data to adaptively steer the reconstruction \cite{arridge2019solving, aspri2020data, aspri2020data1, bajpai2025stochastic, zhou2025convergence}, trained neural network based regularizers~\cite{bianchi2023uniformly,li2020nett,lunz2018adversarial, obmann2021augmented},  and graph Laplacian regularizers, which encode geometric or relational structure via graph-based similarity models \cite{bajpai2026graph, bajpai2025convergence, bianchi2025data}. These approaches have demonstrated improved reconstruction performance in high-dimensional ill-posed inverse problems by promoting structure-aware and data-adaptive solutions.

 In \cite{romano2017little}
a \emph{Regularization by Denoising (RED)} algorithm was proposed for linear inverse problems. Building upon
the extensive literature on image denoising (see, e.g., \cite{buades2005review,milanfar2012tour}) and recent
methodological developments (e.g., \cite{chen2016trainable,zhang2017beyond}), they demonstrated how an explicit
regularization functional $\Theta(u)$ can be systematically constructed from an image
denoiser $f :  U \to  U$ via a simple and effective way as
\begin{equation}\label{eqn:RED regularizer}
    \Theta_{\text{red}}(u) = \frac{1}{2} \langle u, u - f(u) \rangle.
\end{equation}
The regularizer is defined through an image-adaptive Laplacian, where the underlying structure is induced by the selected denoiser $f(\cdot).$
Using $\Theta_{\text{red}}(\cdot)$ they proposed several reconstruction algorithms based on steepest descent, the alternating direction method of multipliers, and fixed point iterations, which achieve state of the art performance in super resolution  and image deblurring  tasks. Later in~\cite{reehorst2018regularization},
it is demonstrated that the  minimization of variational form of RED algorithms holds only under restrictive conditions, namely when
the denoiser is locally homogeneous, meaning that
$f((1+\varepsilon)u) = (1+\varepsilon)f(u)$ for all $u$ and sufficiently
small $\varepsilon >0$, and when the Jacobian of $f$ is symmetric.
Since these
assumptions are generally violated by several practical denoisers, the variational interpretation
of RED is not valid in general. To address this limitation, the authors of~\cite{reehorst2018regularization}
introduced a score-matching framework to analyze the convergence behavior of RED algorithms. In both the works, convergence is not proven in the sense of regularization.

To establish denoiser-based algorithms as convergent regularization methods,
Plug-and-Play (PnP) denoising is considered in conjunction with forward--backward
splitting (FBS), leading to the following iterative regularization scheme:
\begin{equation}\label{eqn: PnP-FBS}
u_{k+1,\sigma}^\delta
= f_\sigma\!\left(u_{k, \sigma}^\delta - \mu G^*\bigl(G u_{k,\sigma}^\delta - v^\delta\bigr)\right),
\end{equation}
where $k \in \mathbb{N}$ denotes the iteration index and $u_{k,\sigma}^\delta$ represents the $k$-th iterate, $\mu > 0$ is a fixed step-size, and $f_\sigma$ denotes a denoising operator that removes Gaussian noise with standard deviation $\sigma$ from its input. Furthermore, $G$ denotes the linear version (linearization) of the forward operator $\mathcal{G}$, and $G^*$ represents the adjoint of $G$.
This approach
is inspired by the seminal work of~\cite{venkatakrishnan2013plug}, which introduced
the incorporation of denoising operators within proximal splitting algorithms in a
PnP framework. To the best of our knowledge, the convergence of \eqref{eqn: PnP-FBS} in the sense of
regularization was first established in~\cite{ebner2024plug}. This analysis was
subsequently extended in~\cite{hauptmann2025convergent}, where a more refined control
of the denoiser parameter $\sigma$ was introduced in order to appropriately adjust
the regularization strength across different noise levels. However, the scheme \eqref{eqn: PnP-FBS} suffers from several limitations. The regularizer is implicit through the denoiser, making its theoretical interpretation difficult. Moreover, the method is primarily suited for linear problems and lacks adaptivity due to the use of a fixed step size. Additionally, selecting an appropriate regularization parameter remains challenging and is often impractical in real-world applications.
% However, the scheme \eqref{eqn: PnP-FBS} suffers from several limitations, primarily due to the implicit nature of the regularizer, limitation to solve linear problems, non adaptive nature due to fixed step size, and the challenge of selecting an appropriate regularization parameter, which is often infeasible in practical applications.
These limitations motivate the development of an iterative regularization method that does not rely on \emph{a priori} parameter selection and instead employs the regularizer \eqref{eqn:RED regularizer} in an explicit manner.
Hence, we propose an iterative algorithm, which  is given as 
\begin{equation}\label{eqn:DDR}
    u_{k+1}^\delta = u_k^\delta - \mu_k^\delta \mathcal{G}'(u_k^\delta)^* (\mathcal G(u_k^\delta) - v^\delta) - \lambda_k^\delta (u_k^\delta - D_{h_k}(u_k^\delta)),
\end{equation}
where $D_{h_k}(\cdot) : U \to U$ is a denoiser with the parameter $h_k>0$ that controls the strength of the denoising, $\mathcal{G}'(u_k^\delta)^*$ denotes the adjoint of the Fr\'echet derivative of $\mathcal G$ at $u_k^\delta$,  $\mu_k^\delta$ denotes the step size, and $\lambda_k^\delta$ represents the
weighting parameter, both of which are chosen appropriately in adaptive manner. A detailed discussion of the required assumptions is presented in 
Section~\ref{sec:assumptions}, while the algorithm is described in Subsection~\ref{subsec:the method}.
Owing to the ill-conditioning of the operator $\mathcal{G}(\cdot)$, the iterative scheme cannot be
continued indefinitely and must be terminated in a controlled manner, a strategy
commonly referred to as \emph{early stopping} in the machine
learning  and inverse problem  literature~\cite{engl1996regularization}. Early stopping is now widely recognized as a crucial
mechanism for preventing overfitting, particularly in over-parameterized
models~\cite{barbanoimage,huang2025early,jahn2020discrepancy,jahn2024early, wang2023early}.

With this goal in mind, we adopt the discrepancy principle~\cite{morozov1966solution}, which is based on the
premise that the iteration should be terminated once the residual norm becomes
comparable to the noise level in the data.
% Let $u_k^\delta$ denote the
% iterate at iteration $k$ corresponding to noisy data $v^\delta$, and let
% $\tau > 1$ be a fixed tolerance parameter.
The stopping index
$k_{\mathrm{dp}}$ determined by the discrepancy principle is then defined as
\begin{equation}\label{eqn:dp}
k_{\mathrm{dp}}
:= \min\left\{ k \ge 0 :
\left\| \mathcal{G}(u_k^\delta) - v^\delta \right\|
\le \tau \delta \right\},
\end{equation}
where $\tau > 1$ be a fixed tolerance parameter.
Formally, the applicability of this principle requires an estimate of the noise
level $\delta$, which can often be obtained directly from the observed data. We refer to the iterative scheme \eqref{eqn:DDR}, combined with 
the stopping rule \eqref{eqn:dp}, as 
\emph{Denoiser-Driven Iterative Regularization} (\texttt{DDIR}). 
To the best of our knowledge, this work provides the first rigorous 
convergence analysis of a denoiser-driven regularization method 
formulated within an iterative framework and equipped with an early 
stopping strategy. Consequently, the proposed approach can be regarded 
as a new member of the class of iterative regularization methods in 
modern regularization theory. Extensive numerical experiments on image 
deblurring and phase retrieval CT problems demonstrate that \texttt{DDIR} effectively alleviates several well-known limitations of 
RED-type methods while maintaining high reconstruction quality. We further compare the proposed method with the PnP regularization approach~\cite{ebner2024plug} in order to highlight the advantages of iterative regularization and early stopping over existing methods.

The remainder of this manuscript is organized as follows. Section~\ref{sec:assumptions} introduces the notation, preliminaries, and the assumptions required for the analysis. Section~\ref{sec:Convergence} presents the proposed method and establishes its finite termination property. Section~\ref{Sec:Con_analysis} establishes stability and convergence as 
the noise level tends to zero. Numerical experiments demonstrating the effectiveness of the method 
are provided in Section~\ref{sec:numerical}. Finally, Section~\ref{sec:conclusion} provides a summary of the principal contributions and discusses possible avenues for future research.

% \subsection{Our Contributions}
% \begin{itemize}
%     \item We propose a novel denoiser-driven iterative regularization method equipped with an early stopping strategy, which serves as a principled mechanism to ensure stability and prevent overfitting.

%     \item By incorporating adaptive step sizes and dynamically weighted parameters, the proposed scheme addresses and resolves a challenging open question raised in \cite{romano2017little}.

%     \item We provide a comprehensive comparison with state-of-the-art methods, demonstrating the effectiveness, robustness, and practical advantages of the proposed approach.
% \end{itemize}
%%%%%%%%%%%%%%%%%%%%%%%%%%%%%%%%%%%%%%%%%%%%%%%%%%%%%%%
\section{Preliminaries and assumptions}\label{sec:assumptions} 
In this section, we introduce the notation, state the required assumptions, and present the preliminary material that forms the foundation of our analysis. For any $\wp > 0$ and $u \in U$, we 
denote by $\mathcal{B}(\wp, u)$ the closed ball of radius $\wp$ centered at $u$, that is,
\(
\mathcal{B}(\wp, u)
:=
\left\{
\tilde{u} \in U \;:\; \|u - \tilde{u}\| \le \wp
\right\}.
\)
Throughout the paper, we work under the following assumptions. 
\begin{assumption}\label{assum:existence of a solution}
\begin{enumerate}
    \item [(A1)] There exists a $\wp> 0$ such that $\mathcal{B}(2\wp, u_0) \subset \operatorname{dom}(\mathcal{G})$ and \eqref{eq:forward_model} has a solution $u^\dagger \in \mathcal{B}(\wp, u_0)$ corresponding to $\delta = 0.$
    \item [(A2)]  \(\mathcal{G} : \operatorname{dom}(\mathcal G) \subseteq U \to V\) is continuous.
%     \textcolor{blue}{$\mathcal{G}$ is weakly closed on $\operatorname{dom}(\mathcal G)$, i.e., for any sequence 
% $\{u_k\} \subset \operatorname{dom}(\mathcal G)$ satisfying $u_k \rightharpoonup u \in U$ and 
% $\mathcal{G}(u_k) \rightharpoonup v \in Y$, it holds that 
% $u \in \operatorname{dom}(\mathcal G)$ and $\mathcal{G}(u) = v$. }
\item[(A3)] $\mathcal{G}$ is Fr\'echet differentiable on $\operatorname{dom}(\mathcal{G})$, and the mapping 
$u \mapsto \mathcal{G}'(u)$ is continuous on $\mathcal{B}(2\wp, u_0)$. Moreover,
\(
\|\mathcal{G}'(u)\| \le B_\mathcal{G}, \)
for all \( u  \in \mathcal{B}(2\wp, u_0),
\)
 and for some constant $B_\mathcal{G} > 0$. Furthermore, there exists a constant 
$0 \le \zeta < 1$ such that
\begin{equation}\label{eqn:TCC}
\|\mathcal{G}(u) - \mathcal{G}(\bar{u}) - \mathcal{G}'(\bar{u})(u - \bar{u})\|
\le \zeta \|\mathcal{G}(u) - \mathcal{G}(\bar{u})\|
\end{equation}
for all $u, \bar{u} \in \mathcal{B}(2\wp, u_0)$.
\end{enumerate}
\end{assumption}
Here $u_0 = u_0^\delta$ denotes the initial guess.
The requirements stated in Assumption~\ref{assum:existence of a solution} are standard benchmarks for analyzing the convergence of iterative regularization techniques for nonlinear ill-posed inverse problems 
\cite{hanke1995convergence,kaltenbacher2008iterative}. In particular, (A1) is a standing assumption in which we suppose that the noise-free observation $v\in V $, corresponding to $\delta = 0$ in~(\ref{eq:forward_model}), is generated by the action of the operator $\mathcal{G}(\cdot)$ on a ground-truth element $u^\dagger \in \mathcal{B}(\wp, u_0)$. The inequality \eqref{eqn:TCC}, known as 
the \emph{tangential cone condition} (TCC), has been established for various classes of 
nonlinear inverse problems \cite{hanke1995convergence,jin2012sparsity}. Moreover, if the operators $\mathcal{G}(\cdot)$ is 
linear, then \eqref{eqn:TCC} holds with $\zeta = 0$.
% \begin{assumption}\label{assum:existence of a solution}
%     The operator $M: U \to V$ is bounded linear and there exist $u^\dagger \in U$ such that $Mu^\dagger =v.$
% \end{assumption}

Next, we provide the definition of the proximal operator associated with a (possibly non-smooth) convex functional $g(\cdot)$, a fundamental concept in modern non-smooth optimization and variational analysis.
\begin{definition}
   Let $g : U \to \mathbb{R} \cup \{+\infty\}$ be a proper, lower semicontinuous, and convex function. 
For a given parameter $\omega > 0$, the proximal operator associated with $g$ is defined as
\begin{equation}\label{eq:prox_def}
\operatorname{prox}_{\omega g}(u)
:= \arg\min_{y \in U}
\left\{
\omega g(y) + \frac{1}{2} \| u - y \|^2
\right\},
\end{equation}
for any $y \in U$.
The minimization problem in~\eqref{eq:prox_def} admits a unique solution due to the strong convexity induced by the quadratic term.
\end{definition}
% If, in addition, $g$ is differentiable, the proximal operator admits an interpretation in terms of gradient flows.
% With step size $\omega = \lambda$, the proximal iteration
% \(
% u_{k+1} = \operatorname{prox}_{\lambda g}(u_k)
% \)
% can be equivalently written as the implicit gradient descent scheme
% \[
% u_{k+1} = u_k - \lambda \nabla g(u_{k+1}),
% \]
% which follows directly from the first-order optimality condition associated with~\eqref{eq:prox_def}.
%%%%%%%%%%%%%%%%%%%%%%%%%%%%%%%%%%%%%%%%%%%%%%%%%%%%%%%%%%%%%%%%%%%%%%%%%%%%%%%%%%%%%%%%%%%%%%%%%%%%%%%%%%%%%%%%%%%%%%%%%%%%%%%%%%%%%%%%%%%%%%%%%%%%
\subsection{Image denoising}
    Image denoising can be viewed as a particular instance of~\eqref{eq:forward_model}, corresponding to the identity forward model
$v^\delta = u + \eta$, where $\eta$ denotes additive white Gaussian noise with variance
$\sigma^2$. In this setting, the variational form 
reduces to
\begin{equation}\label{eqn:variational denoising}
u_{\mathrm{denoise}} \in \arg\min_{u \in U}
\left\{
\frac{1}{2\sigma^2}\|v^\delta - u\|^2 + \lambda \Theta(u)
\right\},
\end{equation}
where $\lambda \geq 0$ is the step size and $\Theta (\cdot): U \to \mathbb{R}\cup \{+ \infty\}$ is a specifically chosen prior functional. Over the past decade, numerous highly effective image denoising algorithms have been developed, achieving remarkable performance. In general, an image denoiser can be viewed as a mapping 
\(
D : U \rightarrow U,
\)
which transforms a noisy image $v^\delta$ into an estimate $\hat{u} = D(v^\delta)$ of the underlying clean image $u$. Ideally, $\hat{u}$ provides an accurate approximation of $u$. Denoising methods are commonly derived from frameworks such as maximum a posteriori (MAP) estimation, minimum mean square error (MMSE) estimation, collaborative filtering techniques, supervised learning approaches, and related methodologies.

It is worth noting that \eqref{eq:prox_def} bears a strong resemblance to \eqref{eqn:variational denoising}, suggesting that proximal operators can be interpreted as a particular class of denoisers. Beyond this class, several highly effective denoising methods have been proposed in the literature, including NLM~\cite{buades2005non}, BM3D~\cite{dabov2007image}, TNRD~\cite{chen2016trainable}, and DnCNN~\cite{zhang2017beyond}. See Fig.~\ref{fig:denoising_hierarchy} for structural relation.
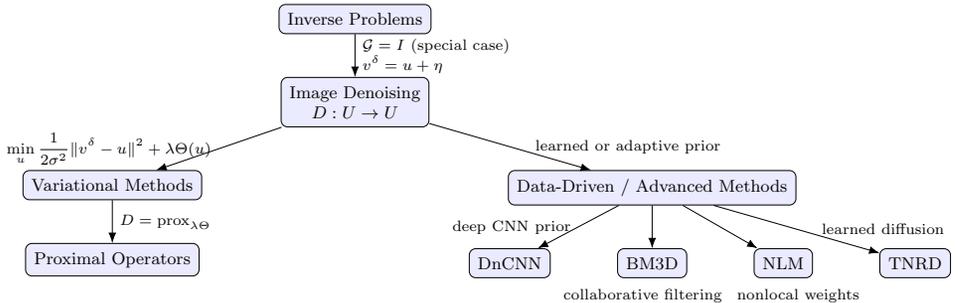
\begin{figure}[ht]
\centering
\resizebox{\columnwidth}{!}{
\begin{tikzpicture}[
    node distance=0.9cm and 1.6cm,
    every node/.style={font=\footnotesize, align=center},
    arrow/.style={-Latex, thin},
    box/.style={
        rectangle,
        rounded corners=3pt,
        draw=black,
        fill=blue!8,
        inner sep=4pt
    }
]

% Top
\node[box] (inv) {Inverse Problems};

% Level 2
\node[box] (den) [below=0.7cm of inv] {Image Denoising \\ $D: U \to U$};

% Level 3
\node[box] (var) [below left=0.7cm and 1.3cm of den] {Variational Methods};
\node[box] (data) [below right=0.7cm and 1.3cm of den] {Data-Driven / Advanced Methods};

% Level 4
\node[box] (prox) [below=0.7cm of var] {Proximal Operators};

\node[box] (bm3d) [below=0.7cm of data] {BM3D};
\node[box] (nlm) [right=1.1cm of bm3d] {NLM};
\node[box] (tnrd) [right=1.1cm of nlm] {TNRD};
\node[box] (dncnn) [left=1.1cm of bm3d] {DnCNN};

% Arrow: Inverse Problems → Image Denoising
\draw[arrow] (inv) -- 
node[midway, right, font=\scriptsize, align=left]
{$\mathcal{G}=I$ (special case) \\ $v^\delta=u+\eta$} 
(den);

% Arrow: Image Denoising → Variational Methods
\draw[arrow] (den) -- 
node[midway, left, font=\scriptsize, align=left]
{\hspace{1mm} $\displaystyle 
\min_u \frac{1}{2\sigma^2}\|v^\delta-u\|^2 + \lambda \Theta(u)
$} 
(var);

% Arrow: Image Denoising → Data-Driven Methods
\draw[arrow] (den) -- 
node[midway, right, font=\scriptsize, align=left]
{\hspace{2mm} learned or adaptive prior} 
(data);

% Arrow: Variational Methods → Proximal Operators
\draw[arrow] (var) -- 
node[midway, right, font=\scriptsize]
{$D=\operatorname{prox}_{\lambda \Theta}$} 
(prox);

% Arrow: Data-Driven → BM3D
\draw[arrow] (data) -- 
% node[midway, above, font=\scriptsize]
% {collaborative filtering} 
(bm3d);

% Arrow: Data-Driven → NLM
\draw[arrow] (data) -- 
% node[midway, below, font=\scriptsize]
% {nonlocal weights} 
(nlm);

% Arrow: Data-Driven → TNRD
\draw[arrow] (data) -- 
node[midway, right, font=\scriptsize]
{\hspace{2mm} learned diffusion} 
(tnrd);

% Arrow: Data-Driven → DnCNN
\draw[arrow] (data) -- 
node[midway, left, font=\scriptsize]
{\hspace{-2mm} deep CNN prior} 
(dncnn);

% Small descriptions below boxes (no border)
\node[below =0.05cm of bm3d, font=\scriptsize] 
{\hspace{-4mm} collaborative filtering};

\node[below =0.05cm of nlm, font=\scriptsize] 
{\hspace{4mm} nonlocal weights};

\end{tikzpicture}
}
\caption{Image denoising methods within inverse problems.}
\label{fig:denoising_hierarchy}
\end{figure}
The remarkable effectiveness of image denoising in noise suppression has motivated the use of denoisers in solving more general inverse problems of the form \eqref{eq:forward_model}. In this work, we likewise employ a denoiser as an explicit regularization functional, defined in \eqref{eqn:RED regularizer}, in the context of iterative regularization method.
\subsection{Assumptions and properties of denoiser} We impose the following assumptions on the denoiser $D(\cdot)$, which are assumed to hold throughout this paper.
\begin{assumption}\label{local homogeneity} 
 The denoiser $D(\cdot)$ is assumed to be differentiable. 
Moreover, for sufficiently small $\varepsilon > 0$ and for all $u \in U$, it satisfies
\(
D\big((1+\varepsilon)u\big) = (1+\varepsilon)D(u),
\)
that is, the denoiser is locally homogeneous.  
\end{assumption}
\begin{proposition}[\cite{romano2017little}]
   Under Assumption~\ref{local homogeneity}, $D(\cdot)$ satisfies
\(
[\mathbf{J}D(u)]\,u = D(u),
\)
where $[\mathbf{J}D(u)]$ denotes the Jacobian of $D$ evaluated at $u$.
\end{proposition}
\begin{assumption}\label{ass:symmetric jacobian}
    The Jacobian $\mathbf{J}$ of $D$ is symmetric, that is $[\mathbf{J}D(u)] = [\mathbf{J}D(u)]^T.$
\end{assumption}
\begin{proposition}[\cite{reehorst2018regularization}]
 Let Assumption~\ref{local homogeneity} hold. For $\Theta_{\mathrm{red}}(\cdot)$ defined in \eqref{eqn:RED regularizer},
    \begin{equation*}
        \nabla_{u} \Theta_{\mathrm{red}}(u) = u - \frac{1}{2}D(u) - \frac{1}{2}[\mathbf{J}D(u)]^T u.
    \end{equation*}
    If, in addition, Assumption~\ref{ass:symmetric jacobian} also hold, then 
      \(
        \nabla_{u} \Theta_{\mathrm{red}}(u) = u - D(u).
  \)
\end{proposition}
\begin{remark}
    In~\cite{reehorst2018regularization}, it is demonstrated that a number of popular denoisers fail to satisfy the Assumption~\ref{ass:symmetric jacobian}, including MF~\cite{huang1979fast},  NLM~\cite{buades2005non}, BM3D~\cite{dabov2007image}, TNRD~\cite{chen2016trainable}, and DnCNN~\cite{zhang2017beyond}. 
    % As a result, the corresponding \texttt{DDIR} algorithms cannot be characterized as iterative methods for the regularizer $\Theta_{\mathrm{red}}(\cdot)$ in (\ref{eqn:RED regularizer}). 
    Natural approaches to overcome the limitations imposed by Assumption~\ref{local homogeneity} and~\ref{ass:symmetric jacobian} include adopting the \emph{score-matching-by-denoising} (SMD) framework~\cite{reehorst2018regularization}, or alternatively, invoking Rockafellar function~\cite{rockafellar2009variational} to reformulate the regularization functional $\Theta_{\mathrm{red}}(\cdot)$, as detailed in Appendix~B of~\cite{cohen2021regularization}~\footnote{To relax the differentiability requirement on the denoiser, it is necessary that the residual operator \( D_{\mathrm{res}} = I - D \) be maximally cyclically monotone. In contrast, requiring the mapping \( D(\cdot) \) to be cyclically firmly nonexpansive constitutes a sufficient condition. This property is satisfied, for instance, by proximal denoisers~\eqref{eqn:variational denoising}, see \cite[Theorem B.6]{cohen2021regularization}.}.
However, since this direction is not central to the present work and the main results would remain essentially unchanged within these settings, we do not pursue this extension here.
\end{remark}
% \begin{assumption}\label{ass:linear_symmetric_denoiser}
% The denoiser $D_\sigma : U \to U$ is assumed to be
% a bounded, linear, self-adjoint, and positive definite operator satisfying
% \begin{equation}
% 0 < c \le \langle D_\sigma u, u \rangle \le C \|u\|^2,
% \qquad \forall\, u \in U,
% \end{equation}
% for some constants $c, C > 0$. 
%     \marginpar{\footnotesize \textcolor{blue}{HB: Need to check whether this assumptions is required in our case or not.}}
% \end{assumption}
It is immediate to see that, under these two assumptions the gradient descent step to minimize the  objective function 
$\mathcal{J}(u):= \frac{1}{2} \|\mathcal{G}(u) - v^\delta\|^2 + \lambda \Theta_{\mathrm{red}}(u)$ 
is given by 
\begin{align}\label{eqn:fixed_denoiser_gradient_step}
   \nonumber u_{k+1}^\delta &= u_k^\delta - \mu \nabla_{u_k^\delta} \mathcal{J}(u_k^\delta)\\ \nonumber
    & =  u_k^\delta - \mu \left(\frac{1}{2}\nabla_{u_k^\delta} (\|\mathcal{G}(u_k^\delta) - v^\delta \|^2) + \lambda \nabla_{u_k^\delta} (\Theta_\mathrm{red}(u_k^\delta)) \right) \\
    & =  u_k^\delta - \mu \left (\mathcal{G}'(u_k^\delta)^*(\mathcal{G}(u_k^\delta) - v^\delta) + \lambda (u_k^\delta - D(u_k^\delta)) \right).
\end{align}
This expression is equivalent to \eqref{eqn:DDR} provided that the step-size $\mu$ and weighted parameter $\lambda$
 is chosen appropriately in adaptive way.
% \paragraph{Examples}
% Typical examples satisfying Assumption include
% graph Laplacian denoisers of the form $D_\sigma = (I + \sigma L)^{-1}$, where $L$ is
% a symmetric graph Laplacian, classical Tikhonov-type smoothing operators, Gaussian
% and diffusion-based filters, as well as spectral filters associated with symmetric
% operators.

% Under Assumption~\ref{ass:linear_symmetric_denoiser}, there exists a convex quadratic
% functional
% \begin{equation}
% \mathcal{R}_\sigma(u)
% = \tfrac{1}{2} \langle u, (D_\sigma^{-1} - I) u \rangle,
% \end{equation}
% such that
% \[
% D_\sigma(u) = \operatorname{prox}_{\mathcal{R}_\sigma}(u) = \arg \min_{w} \frac{1}{2}\|u - w\|^2 + \sigma\mathcal{R}_\sigma(u).
% \]
% Moreover, the Jacobian of $D_\sigma$ is constant and symmetric.
% \begin{assumption}\label{assum:existence of a solution}
%     The operator $M: U \to V$ is bounded linear and there exist $u^\dagger \in U$ such that $Mu^\dagger =v.$
% \end{assumption}
We also require the following assumption on the denoiser to prove convergence and stability of our proposed method.
\begin{assumption}
\label{assum:denoiser}
Let $D : U \to U$ be a denoising operator. 
We assume that $D$ is $q$-contractive for some 
$q \in (0,1)$, i.e.,
\[
\|D(u) - D(\hat u)\| 
\le q \|u - \hat u\|,
\quad \text{for all } u, \hat u \in U.
\]
\end{assumption}
We call $u \in U$ a fixed point of $D$ iff $D(u) =u$ and we write $\operatorname{Fix}(D) = \{u \in U \; \mid \; D(u) = u\}.$ Throughout this paper we assume that  
$\operatorname{Fix}(D)$ is nonempty~\footnote{
A commonly adopted assumption is that $D(0)=0$, which is
satisfied, for instance, by denoisers of the form $D(u)=W(u)u$, as
discussed in~\cite{cohen2021regularization}. Moreover, under
Assumption~\ref{assum:existence of a solution} (A1), we have $u^\dagger \in \operatorname{Fix}(D)$,
which implies that the  set $\operatorname{Fix}(D)$ is
nonempty. Finally, all denoisers employed in our numerical
experiments satisfy this property.
}. In addition, we assume that 
every solution $u$ of the equation $\mathcal{G}(u) = v$ belongs to 
$\operatorname{Fix}(D)$.
Instead of employing a fixed denoiser within the gradient descent step 
as in~\eqref{eqn:fixed_denoiser_gradient_step}, it is desirable to 
adaptively control the denoising strength during the iteration process. 
In particular, the denoising effect should gradually diminish as the 
iterates approach the solution. To this end, we introduce a parametric 
family of denoisers $(D_h)_{h>0}$, where the parameter $h$ regulates 
the denoising strength. This family is constructed using the concept 
of averaged operators~\cite{cohen2021regularization}.
\begin{definition}
    Given a denoiser $D$ and $h \in (0, 1)$, we define the averaged denoiser $D_h = hD + (1-h)I,$ where $I$ is the identity operator. Notice it holds that $\operatorname{Fix}(D_h) = \operatorname{Fix}(D).$ 
\end{definition}
\begin{proposition}\label{Prop:2.10}
Let Assumption~\ref{assum:denoiser} hold.
For $h \in (0,1)$, let $D_h$ is the averaged denoiser
then it is contractive with Lipschitz constant
\(
L_h = 1 - h(1-q) < 1.
\) Moreover $$\|D_h(u) - u \| \to 0 \quad \text{as } h \to 0.$$
\end{proposition}
\begin{proof}
For arbitrary $u, \hat{u} \in U$, we have
\begin{align*}
\|D_h(u)-D_h(\hat{u})\|
% &= \|h\big(D(u)-D(\)\big)+(1-h)(x_1-x_2)\|\\
&\le
h\|D(u)-D(\hat u)\| + (1-h)\|u -\hat{u}\| \\
&\le
\big(hq + 1-h\big)\|u-\hat{u}\| =
\big(1 - h(1-q)\big)\|u -\hat{u}\|.
\end{align*}
Since $0<h<1$ and $0< q<1$, it follows that $L_h=1-h(1-q)<1$. Also $\|D_h(u) - u \| = h\|D(u)-u\| \to 0$ as $h \to 0.$
\end{proof}
The above result offers a systematic way to modulate the denoising strength through an external parameter by performing appropriate averaging.
We refer to the collection of operators 
$(D_h)_{0<h<1}$, with $D_h : U \to U$ 
for each $h \in (0,1)$, 
as an admissible denoiser family. The condition $\|D_h(u) - u \| \to 0$ is natural, as it requires the denoising effect to diminish asymptotically in the vanishing noise regime.
%%%%%%%%%%%%%%%%%%%%%%%%%%%%%%%%%%%%%%%%%%%%%%%%%%%%%%%%%%%%%%%%%%%%%% 
\section{Denoiser Driven Iterative Regularization}\label{sec:Convergence}
In this section, we first present the motivation underlying the proposed \texttt{DDIR} method and subsequently provide its theoretical foundation by establishing its monotonicity property, through the analysis of its early stopping criterion. To this end, we consider an admissible family of denoisers $(D_h)_{1> h > 0}$. The subsequent section is devoted to a detailed description of the algorithm and the corresponding step-size selection strategy.
\subsection{The method}\label{subsec:the method}
% In this subsection, we present a detailed description of the proposed \texttt{DDIR} method~\eqref{eqn:DDR}. 
The formulation begins with the introduction of the step-size $\mu_k^\delta$ and the weighted parameter $\lambda_k^\delta$ to enhance the convergence speed of the method~\eqref{eqn:DDR}. It is sensible to update these parameters dynamically at each step so that the iterates consistently remain near a solution of~(\ref{eq:forward_model}). With this motivation, we define the parameters as
\begin{equation}\label{eqn:step size mu}
   0 < \gamma \leq  \mu_k^\delta = 
        \min \left\{ \frac{\gamma
    _0\|\mathcal{G}(u_k^\delta) - v^\delta\|^2}{\|\mathcal{G}'(u_k^\delta)^* (\mathcal{G}(u_k^\delta) -v^\delta)\|^2}, \gamma_1 \right\},
\end{equation}
\begin{equation}\label{eqn:weighted parameter lambda}
   \lambda_k^\delta =  \begin{cases}
      \min \left\{ \frac{\nu_0 \|\mathcal{G}(u_k^\delta) - v^\delta\|^2}{\|u_k^\delta - D_{h_k}(u_k^\delta)\|^2}, \nu_1 \right\} , & \text{if } \|u_k^\delta - D_{h_k}(u_k^\delta)\| \neq 0\\ 
      0, &\text{if } \|u_k^\delta - D_{h_k}(u_k^\delta)\| = 0,
    \end{cases} 
\end{equation}
where $\gamma_0, \gamma_1, \nu_0, \nu_1$, and $\gamma$ denote fixed positive constants, and $\{h_k\}_{k \in \mathbb{N}}$ is a sequence satisfying $h_k \in (0,1)$ for all $k \in \mathbb{N}$ and $h_k \to 0$ as $k \to \infty$.
\begin{remark}
Since $\mathcal{G}'$ is bounded, the step-size $\mu_k^\delta$ admits a uniform positive lower bound independent of the iteration. Indeed,
\[
\mu_k^\delta \ge  \min\left\{\frac{\gamma_0}{\|\mathcal{G}'(u_k^\delta)\|^2},\, \gamma_1\right\} \ge     \min\left\{\frac{\gamma_0}{B_\mathcal{G}^2},\, \gamma_1\right\} 
=: \hat{\gamma}> 0.
\]
Hence, any fixed choice $0 < \gamma \le \hat{\gamma}$ guarantees an iteration-independent admissible lower bound of the step-size \(\mu_k^\delta\). 
In the special case $\|\mathcal{G}'(u_k^\delta)^* (\mathcal{G}(u_k^\delta) -v^\delta)\|=0$, we have $\mu_k^\delta=\gamma_1$, so the same lower bound remains valid.
\end{remark}
The pseudo-code corresponding to the iterative scheme~\eqref{eqn:DDR} 
is presented in \Cref{alg:DDIR}, which summarizes the implementation of \texttt{DDIR}.
\begin{algorithm}
\caption{\texttt{DDIR}}
\label{alg:DDIR}
\begin{algorithmic}[1]
\STATE{Given: $\mathcal{G}, v^\delta, \delta, \tau, D.$}
\STATE{Fixed parameters: $ \gamma, \gamma_0, \gamma_1, \nu_0, \nu_1.$}
\STATE{Initialize: $u_0^\delta = u_0$ and set $k = 0$.}
\WHILE{$\|\mathcal{G}(u_k^\delta) - v^\delta\| > \tau \delta$}
    \STATE{Choose $h_k \in (0, 1)$ and define $D_{h_k} = h_k D + (1-h_k)I.$}
    \STATE{Compute $D_{h_k}(u_k^\delta)$.}
    \STATE{Update
    \vspace{-4mm}
    \[
    u_{k+1}^\delta = u_k^\delta 
    - \mu_k^\delta \mathcal{G}'(u_k^\delta)^* 
    (\mathcal G(u_k^\delta) - v^\delta)
    - \lambda_k^\delta 
    (u_k^\delta - D_{h_k}(u_k^\delta)).
    \]
   Here $\mu_k^\delta$ and $\lambda_k^\delta$ are defined as in 
    \eqref{eqn:step size mu} and 
    \eqref{eqn:weighted parameter lambda}, respectively.}
    \STATE{Set $k \leftarrow k+1$.}
\ENDWHILE
\RETURN{$u_{k_{\mathrm{dp}}}^\delta$ (approximate solution obtained at the stopping index).}
\end{algorithmic}
\end{algorithm}
\begin{remark}
    The parameter $h$ controls the scale of the denoiser-induced prior and determines how strongly the denoiser modifies the current iterate. It governs the resolution at which prior information is enforced. In the proposed  scheme, 
$h_k \to 0$ as $k \to \infty$ ensures asymptotic unbiasedness, while stability is achieved through early stopping and the cumulative effect of the denoiser during the iterations.
\end{remark}
%%%%%%%%%%%%%%%%%%%%%%%%%%%%%%%%%%%%%%%%%%%%%%%%%%%%%%%%%%%%%%%%%%
 \subsection{Monotonicity and finite termination}
We begin by presenting an auxiliary result that underpins the monotonicity and finite termination properties, which are the main focus of this section.
\begin{proposition}
  Under Assumption~\ref{assum:denoiser} and $u^\dagger \in \operatorname{Fix}(D)$, we have
  \begin{equation}\label{Prop:result}
\langle u^\dagger-u,\; u-D_h(u)\rangle
\le - \frac{c_q}{h} 
\|u- D_h(u)\|^2 \le- c_q 
\|u- D_h(u)\|^2,
  \end{equation}
for all $u \in U$,  where $c_q=\frac{1-q}{(1+q)^2}$.
\end{proposition}
\begin{proof}
Since $u^\dagger \in \operatorname{Fix}(D)$, it follows that 
$u^\dagger \in \operatorname{Fix}(D_h)$ for all $h \in (0,1)$. 
Consequently, we have
\(
D_h(u^\dagger) = u^\dagger.
\)
% {\color{red} previously you wrote $\|D_h(u)-D_h(u^\dagger)\| \le q \|u-u^\dagger\|$, but this is wrong, we have $\|D(u)-D(u^\dagger)\|\le q \|u-u^\dagger\|$. But this makes the estimate better}
Hence, we obtain
% \[
% u - D_h(u)
% = (u - u^\dagger) - \bigl(D_h(u) - D_h(u^\dagger)\bigr).
% \]
% Taking the inner product with $u^\dagger - u$, 
\begin{align}\label{eq:key1}
\nonumber \langle u^\dagger - u,\; u - D_h(u) \rangle
&= \langle u^\dagger - u,\; (u - u^\dagger) - \bigl(D_h(u) - D_h(u^\dagger)\bigr) \rangle\\ \nonumber
&=\langle u^\dagger - u,\; u - u^\dagger \rangle
 - \langle u^\dagger - u,\; D_h(u) - D_h(u^\dagger) \rangle \\ \nonumber
&\le -\|u-u^\dagger\|^2 + \|u-u^\dagger\|\|D_h(u)-D_h(u^\dagger)\| \\ \nonumber
&\le - \|u-u^\dagger\|^2 + (1-h(1-q)\|u-u^\dagger\|^2\\
&\le -h(1-q)\|u-u^\dagger\|^2,
\end{align}
 where we have applied the Cauchy--Schwarz inequality together with Proposition~\ref{Prop:2.10} to obtain last inequality.
Using again $D_h(u^\dagger)=u^\dagger$ and Assumption~\ref{assum:denoiser} to estimate
\begin{align*}
\|u - D_h(u)\| = h \| u - D(u)\|
\le h( \|u - u^\dagger\|
   + \|D(u) - D(u^\dagger)\| \le h(1+q)\|u - u^\dagger\|.
\end{align*}
Hence,
\begin{equation}\label{eq:key2}
\|u - u^\dagger\|^2
\ge \frac{1}{h^2(1+q)^2}\|u - D_h(u)\|^2.
\end{equation}
Combining \eqref{eq:key1} and \eqref{eq:key2}, we conclude that
\[
\langle u^\dagger - u,\; u - D_h(u)\rangle
\le -\frac{(1-q)}{h(1+q)^2}\|u - D_h(u)\|^2
= -\frac{c_q}{h} \|u - D_h(u)\|^2.
\]
The proof is thus complete.
\end{proof}
In the next result, we will show the monotonicity of error $\|u_k^\delta - u^\dagger\|$ for $0 \leq k \leq k_{\mathrm{dp}}$. 
%%%%%%%%%%%%%%%%%%%%%%%%%%%%%%%%%%%%%%%%%%%%%%%%%%%%%%%%%%%
\begin{lemma}[Monotonicity]\label{lemma:monotonicity}
Suppose that for Algorithm~\ref{alg:DDIR}, Assumption~\ref{assum:existence of a solution} and Assumption~\ref{assum:denoiser} hold. Assume that $k^\dagger \leq k_{\mathrm{dp}}$ is an integer, where $k_{\mathrm{dp}}$ is the stopping index as in \eqref{eqn:dp} with $\tau >1$. Moreover, assume that 
\begin{equation}\label{assum: On C}
   \mathcal C:=  \gamma - \nu_0(\nu_1 - c_q) - \gamma_1 \left(\gamma_0 + \zeta + \frac{1 + \zeta}{\tau}\right) > 0,
\end{equation}
 with $c_q \leq \nu_1$.
Then, for any solution $u^\dagger$ of $\mathcal{G}(u)=v$, the following holds.
\begin{enumerate}
    \item[\emph{(i)}] For all $0 \leq k < k^\dagger,$ 
    \begin{equation}\label{eqn:monotonicity}
\|u_{k+1}^\delta - u^\dagger\| \leq \|u_{k}^\delta - u^\dagger\|
\end{equation}
\item[\emph{(ii)}]The iterates $u_k^\delta \in \mathcal{B}(2\wp, u_0)$ for all $0 \leq k \leq k^\dagger$ and the partial sum of squared residual norm up to $k^\dagger$ is bounded, i.e.,
% \vspace{-4mm}
\begin{equation*}
 \sum_{k=0}^{k^\dagger}\|\mathcal{G}(u_k^\delta) - v^\delta\|^2 < \frac{\wp^2}{2\mathcal C}.
\end{equation*}
\end{enumerate}
\end{lemma}
\begin{proof} We begin by proving, via mathematical induction, that 
\(
u_k^\delta \in \mathcal{B}(2\wp, u_0) \)  for all \( 0 \leq k \leq k^\dagger.
\)
The claim is clearly satisfied for $k = 0$ as $u_0^\delta = u_0$. 
Assume now that it holds for all indices $k < k^\dagger$. 
We proceed to verify that the assertion remains valid for the subsequent iteration. For that, we define
\[
z_k^\delta := u_k^\delta - u^\dagger \quad \text{and} \quad d_k^\delta := \mu_k^\delta \mathcal{G}'(u_k^\delta)^* (\mathcal G(u_k^\delta) - v^\delta) + \lambda_k^\delta(u_k^\delta - D_{h_k}(u_k^\delta)).
\]
Under these definitions and the update rule of 
Algorithm~\ref{alg:DDIR}, we derive 
\begin{align}\label{eqn: Gamma M and D}
  \nonumber  \|z_{k+1}^\delta\|^2 &- \|z_k^\delta\|^2 = \|d_k^\delta\|^2 - 2\langle z_k^\delta, d_k^\delta \rangle \\ \nonumber
    &\leq 2\Big[
(\mu_k^\delta)^2\|\mathcal{G}'(u_k^\delta)^* (\mathcal G(u_k^\delta) - v^\delta)\|^2
+ \mu_k^\delta \langle \mathcal{G}'(u_k^\delta)(u^\dagger - u_k^\delta),\, \mathcal{G}(u_k^\delta) - v^\delta \rangle \\
&\quad
+ (\lambda_k^\delta)^2 \|u_k^\delta - D_{h_k}(u_k^\delta)\|^2
+ \lambda_k^\delta \langle u^\dagger - u_k^\delta,\, u_k^\delta - D_{h_k}(u_k^\delta) \rangle
\Big]
=: 2(\Gamma_\mathcal{G} + \Gamma_D). 
\end{align}
Here, $\Gamma_\mathcal{G}$ and $\Gamma_{D}$ denote the terms corresponding to  $\mathcal{G}$ and denoiser $D$, respectively, appearing on the right-hand side of the  inequality. We next analyze $\Gamma_\mathcal{G}$ and $\Gamma_{D}$ separately and derive bounds for each in terms of the squared norm of the residual of $\mathcal{G}$ at the $k$-th iteration. By using \eqref{eqn:dp}, \eqref{eqn:TCC} and \eqref{eqn:step size mu}  we get
\begin{align}\label{eqn:Gamma M} 
\nonumber
\Gamma_\mathcal{G} &= (\mu_k^\delta)^2\|\mathcal{G}'(u_k^\delta)^* (\mathcal G(u_k^\delta) - v^\delta)\|^2     + \mu_k^\delta \langle \mathcal{G}'(u_k^\delta)(u^\dagger - u_k^\delta),\, \mathcal{G}(u_k^\delta) - v^\delta \rangle \\ \nonumber
& \leq \gamma_0 \gamma_1\|\mathcal{G}(u_k^\delta) - v^\delta\|^2  + \mu_k^\delta \langle \mathcal{G}(u_k^\delta) - v - \mathcal{G}'(u_k^\delta)( u_k^\delta- u^\dagger), \mathcal{G}(u_k^\delta) - v^\delta \rangle   \\ \nonumber
& \quad - \mu_k^\delta\|\mathcal{G}(u_k^\delta) - v^\delta\|^2 + \mu_k^\delta\langle v - v^\delta, \mathcal{G}(u_k^\delta) - v^\delta \rangle \\ \nonumber
& \leq \left(\gamma_0 \gamma_1   - \gamma  + \frac{\gamma_1}{\tau}\right) \|\mathcal{G}(u_k^\delta) - v^\delta\|^2 + \zeta \mu_k^\delta \| \mathcal{G}(u_k^\delta) - v\| \|\mathcal{G}(u_k^\delta) - v^\delta \|\\ \nonumber
& \leq \left(\gamma_0 \gamma_1   - \gamma  + \frac{\gamma_1}{\tau}\right) \|\mathcal{G}(u_k^\delta) - v^\delta\|^2 + \zeta \gamma_1(1 + 1/\tau) \| \mathcal{G}(u_k^\delta) - v^\delta\|^2 \\ 
& \leq -\left[ \gamma - \gamma_1 \left(\gamma_0 + \zeta + \frac{1 + \zeta}{\tau}\right) \right] \| \mathcal{G}(u_k^\delta) - v^\delta\|^2.
\end{align}
In contrast, to estimate $\Gamma_D$ we use \eqref{eqn:weighted parameter lambda} and \eqref{Prop:result} to get 
\begin{align}\label{eqn:Gamma D}
 \nonumber  \Gamma_D &= (\lambda_k^\delta)^2 \|u_k^\delta - D_{h_k}(u_k^\delta)\|^2  + \lambda_k^\delta \langle u^\dagger - u_k^\delta, u_k^\delta - D_{h_k}(u_k^\delta) \rangle \\ \nonumber
 & \leq  (\lambda_k^\delta)^2 \|u_k^\delta - D_{h_k}(u_k^\delta)\|^2  - \lambda_k^\delta c_q \|u_k^\delta - D_{h_k}(u_k^\delta)\|^2 \\
 & \leq  \lambda_k^\delta (\lambda_k^\delta  -c_q  ) \|u_k^\delta - D_{h_k}(u_k^\delta)\|^2  \leq \nu_0( \nu_1 - c_q)\| \mathcal{G}(u_k^\delta) - v^\delta\|^2. 
\end{align}
% where we have used (\ref{eqn: delta}), (\ref{eqn: discrepancy}) and \eqref{alpha} in deriving \eqref{ek, dk}. 
% Similarly, by using the definitions of $\alpha_k^\delta$ and $ \beta_k^\delta$ with the inequality $(a_1 + a_2)^2 \leq 2(a_1^2 + a_2^2)$ for $a_1, a_2 \in \mathbb{R}$, we can estimate the first term of (\ref{split eqn}) as
% \begin{align}\label{dk}
%  \nonumber   \|d_k^\delta\|^2 &= \| \alpha_k^\delta A^*(Au_k^\delta - v^\delta) + \beta_k^\delta \Delta_{u_k^\delta}u_k^\delta\|^2 \\ \nonumber
%     & \leq 2 \eta_0\eta_1 \|Au_k^\delta - v^\delta\|^2 + 2\nu_0 \nu_1\|Au_k^\delta - v^\delta\|^2 \\ 
%     & \leq 2(\eta_0\eta_1+ \nu_0 \nu_1)\|Au_k^\delta - v^\delta\|^2. 
% \end{align}
where $c_q = \frac{1-q}{(1+q)^2}$.
 Further, by using \eqref{eqn:Gamma M} and \eqref{eqn:Gamma D} in \eqref{eqn: Gamma M and D}, we get
\begin{equation}\label{Concluding eqn of first Lemma}
     \|z_{k+1}^\delta\|^2 - \|z_k^\delta\|^2 \leq -2\left[ \gamma - \nu_0(\nu_1 - c_q) - \gamma_1 \left(\gamma_0 + \zeta + \frac{1 + \zeta}{\tau}\right) \right]\| \mathcal{G}(u_k^\delta) - v^\delta\|^2.
\end{equation}
Finally, incorporating (\ref{assum: On C}) into (\ref{Concluding eqn of first Lemma}) leads to
\begin{equation}\label{eqn:lemma1 short sum}
     \|z_{k+1}^\delta\|^2 - \|z_k^\delta\|^2 \leq -2\mathcal C\| \mathcal{G}(u_k^\delta) - v^\delta\|^2.
\end{equation}
which gives the required monotonicity for $0 \leq k <k^\dagger$.
We now turn to the proof of assertion (ii).  
From inequality \eqref{eqn:lemma1 short sum}, it follows that the error sequence 
is monotonically non-increasing, that is,
\[
\|u_{k+1}^\delta - u^\dagger\| 
\leq \|u_k^\delta - u^\dagger\| 
\leq \cdots 
\leq \|u_0^\delta - u^\dagger\| =  \|u_0 - u^\dagger\|
\leq \wp.
\]
In conjunction with Assumption~\ref{assum:existence of a solution} (A1), this implies that 
$u_{k+1}^\delta \in \mathcal{B}(2\wp, u_0)$.
Next, summing inequality \eqref{eqn:lemma1 short sum} 
from $k = 0$ to $k = k^\dagger$, we obtain
\begin{equation}\label{Concluding eqn of first Lemma in C}
\sum_{k=0}^{k^\dagger} 
\left( \|z_{k+1}^\delta\|^2 - \|z_k^\delta\|^2 \right)
\leq
-2\mathcal C \sum_{k=0}^{k^\dagger}
\|\mathcal{G}(u_k^\delta) - v^\delta\|^2.
\end{equation}
Since the left-hand side is telescopic, we deduce
% \[
% \|z_{k^\dagger+1}^\delta\|^2 - \|z_0^\delta\|^2
% \leq
% -2\mathcal C \sum_{k=0}^{k^\dagger}
% \|\mathcal{G}(u_k^\delta) - v^\delta\|^2.
% \]
% Consequently,
\[
\sum_{k=0}^{k^\dagger}
\|\mathcal{G}(u_k^\delta) - v^\delta\|^2
\leq
\frac{1}{2\mathcal C} \|z_0^\delta\|^2
=
\frac{1}{2\mathcal C} \|u_0 - u^\dagger\|^2 \leq \frac{\wp^2}{2\mathcal C}
< \infty.
\]
This completes the proof of assertion (ii).
\end{proof}
\subsection*{Positivity of $\mathcal C$} 
Observe that \( c_q = \frac{1 - q}{(1 + q)^2} < 1 \) for \( q \in (0,1) \), where \( q \) is associated with the selected denoiser, and \( \zeta \in [0,1) \) is dictated by the underlying problem. The remaining parameters of \eqref{assum: On C} are at our disposal, chosen such that \( \tau > 1 \) and \( 0 < \gamma \leq \gamma_1 \). Furthermore, the inequality \( c_q \leq \nu_1 \) implies that \( \nu_1 - c_q \geq 0 \), which ensures that both subtractive terms appearing in \( \mathcal{C} \) are nonnegative. Hence
the condition $\mathcal C>0$ is equivalent to
\begin{equation}\label{eq:C_positive_condition}
\nu_0(\nu_1 - c_q)
+
\gamma_1\left(
\gamma_0 + \zeta + \frac{1+\zeta}{\tau}
\right)
<
\gamma.
\end{equation}
Now, let
\(
\mathcal{H} := \gamma_0 + \zeta + \frac{1+\zeta}{\tau}.
\)
Then \eqref{eq:C_positive_condition} can be written as
\(
\nu_0(\nu_1 - c_q) + \gamma_1 \mathcal{H} < \gamma.
\)
Since $0<\gamma\le\gamma_1$, dividing by $\gamma_1$ yields
\[
\frac{\nu_0(\nu_1 - c_q)}{\gamma_1} + \mathcal{H}
<
\frac{\gamma}{\gamma_1}
\le 1.
\]
Consequently, a necessary structural condition for $\mathcal C>0$ is
\(
\mathcal{H} < 1.
\)
If this condition fails, then even in the most favorable case 
$\nu_0=0$ and $\gamma=\gamma_1$ ($\mu_k^\delta$ is constant and $\lambda_k^\delta = 0$), we obtain
\(
\mathcal C = \gamma_1(1-\mathcal{H}) \le 0,
\)
and therefore $\mathcal C>0$ is impossible.

Assuming $\mathcal{H}<1$, the positivity of $\mathcal C$ is ensured by choosing
the parameters such that
\(
\nu_0(\nu_1 - c_q)
<
\gamma - \gamma_1 \mathcal{H}.
\)
In particular:
\begin{itemize}
\item If $\nu_1 = c_q$, then the first term vanishes and the condition
reduces to
\(
\gamma_1 \mathcal{H} < \gamma.
\)
\item If $\nu_1 > c_q$, then it suffices to impose
\[
\nu_0
<
\frac{\gamma - \gamma_1 \mathcal{H}}{\nu_1 - c_q}, \quad \text{provided that } \gamma - \gamma_1 \mathcal{H} > 0.
\]
\end{itemize}
Thus, under the structural restriction 
$\gamma_0 + \zeta + \frac{1+\zeta}{\tau} < 1$,
the positivity of $\mathcal C$ can always be achieved by choosing 
$\nu_0$ sufficiently small (and $\nu_1$ close to $c_q$ if desired).
\begin{remark}\label{eqn:C_for_linear}
If the forward operator is linear, then $\zeta = 0$, and consequently the constant $\mathcal C$ reduces to 
\vspace{-2mm}
\[
\mathcal C
=
\gamma
-
\nu_0(\nu_1 - c_q)
-
\gamma_1\left(
\gamma_0 + \frac{1}{\tau}
\right).
\vspace{-2mm}
\]
In this case, the necessary structural condition becomes
\(
\gamma_0 + \frac{1}{\tau} < 1.
\)
If this inequality holds, then $\mathcal C>0$ provided that
\(
\nu_0(\nu_1 - c_q)
<
\gamma -
\gamma_1\left(
\gamma_0 + \frac{1}{\tau}
\right).
\)
In particular, if $\nu_1 = c_q$, the positivity condition reduces to
\(
\gamma_1\left(\gamma_0 + \frac{1}{\tau}\right) < \gamma.
\)
Compared to the nonlinear case, where $\zeta\in(0,1)$, this restriction is 
weaker, since no additional $\zeta$ dependent term appears.
\end{remark}
Building on Lemma~\ref{lemma:monotonicity}, we are now prepared to show that Algorithm~\ref{alg:DDIR} is well-defined, i.e., the iterative procedure terminates in finitely many steps.

\begin{theorem}
Suppose the assumptions of Lemma~\ref{lemma:monotonicity} hold and let Algorithm~\ref{alg:DDIR} be initialized with an initial guess $u_0^\delta$. Then, the algorithm is guaranteed to terminate after a finite number of iterations. Specifically, there exists a  finite stopping index $k_{\mathrm{dp}} \in \mathbb{N}$ such that
\[
\| \mathcal{G}(u_{k_{\mathrm{dp}}}^\delta) - v^\delta \| \leq \tau \delta 
< \| \mathcal{G}(u_k^\delta) - v^\delta \|, 
\qquad 0 \leq k < k_{\mathrm{dp}}.
\]
\end{theorem}
\begin{proof}
Assume that there exists an integer $n \geq 0$ such that
\begin{equation}\label{eqn:Thm_3.6}
    \|\mathcal{G}(u_k^\delta) - v^\delta \| > \tau \delta,
\qquad k = 0,1,\ldots,n .
\end{equation}
Then, from Lemma~\ref{lemma:monotonicity} (ii), we obtain
\(
 \sum_{k=0}^{n} \| \mathcal{G}(u_k^\delta) - v^\delta \|^2
< \frac{\wp^2}{2\mathcal C} < \infty.
\)
Using \eqref{eqn:Thm_3.6}, we further deduce that
\begin{equation}\label{eq:finite_termination}
(n+1) \tau^2 \delta^2
\leq  \sum_{k=0}^{n} \| \mathcal{G}(u_k^\delta) - v^\delta \|^2 < \infty.
\end{equation}
If no finite index $k_{\mathrm{dp}}$ exists for which the stopping 
criterion~\eqref{eqn:dp} is satisfied, then passing to the limit 
$n \to \infty$ in~\eqref{eq:finite_termination} yields a contradiction. 
Consequently, Algorithm~3.1 necessarily terminates after finitely many iterations.
\end{proof}
\section{Convergence analysis}\label{Sec:Con_analysis}
This section is devoted to establishing the convergence analysis of 
Algorithm~\ref{alg:DDIR} by examining the behavior of 
$u_{k_{\mathrm{dp}}}^\delta$ as $\delta \to 0$.   For that purpose, first we establish the convergence for exact data case.
\subsection{Convergence for exact data}\label{Subsec: Convergence for exact data}
In this subsection, we examine the counterpart of 
Algorithm~\ref{alg:DDIR} in the exact data setting, 
which is formulated as follows.
\begin{algorithm}[H]
\caption{\texttt{DDIR} with switching to gradient descent}
\label{alg:exactdata}
\begin{algorithmic}[1]
\STATE{Given: $\mathcal{G}, v, D$.}
\STATE{Fixed parameters: $\gamma, \gamma_0, \gamma_1, \nu_0, \nu_1$.}
\STATE{Initialize: $u_0 = u_0^0$, $k=0$, and set $\texttt{switch}= \texttt{false}$.}
\WHILE{$k \ge 0$}
    \STATE{Choose $h_k \in (0, 1)$ and define $D_{h_k} = h_k D + (1-h_k)I.$}
    \STATE{Compute $D_{h_k}(u_k)$.}
\IF{$\texttt{switch} = \texttt{false}$}
    \STATE Compute
    \[
    \mu_k = \min\left\{
    \frac{\gamma_0\|\mathcal{G}(u_k)-v\|^2}
    {\|\mathcal{G}'(u_k)^*(\mathcal{G}(u_k)-v)\|^2},
    \gamma_1
    \right\}.
    \]

    \IF{$\|u_k - D_{h_k}(u_k)\| = 0$}
        \STATE Set $\lambda_k = 0$ and $\texttt{switch}=\texttt{true}$.
    \ELSE
        \STATE Set
        \vspace{-2mm}
        \[
        \lambda_k =
        \min\left\{
        \frac{\nu_0\|\mathcal{G}(u_k)-v\|^2}
        {\|u_k - D_{h_k}(u_k)\|},
        \nu_1
        \right\}.
        \]
    \ENDIF
\ELSE
    \STATE Set $\lambda_k = 0$.
    \STATE Compute $\mu_k$ as above.
\ENDIF

\STATE Update
\vspace{-4mm}
\[
u_{k+1}
=
u_k
-
\mu_k \mathcal{G}'(u_k)^*(\mathcal{G}(u_k)-v)
-
\lambda_k (u_k - D_{h_k}(u_k)).
\vspace{-4mm}
\]
\STATE $k \leftarrow k+1$.
\ENDWHILE
\end{algorithmic}
\end{algorithm}
\begin{remark}\label{remakr:denoiser fix}
If, at some iteration $k \ge 0$, the iterate satisfies 
$u_k \in \operatorname{Fix}(D) = \operatorname{Fix}(D_{h_k})$, 
that is, $u_k = D_{h_k}(u_k)$, then the \texttt{DDIR} iteration reduces to
\begin{equation}\label{eqn:adaptive_landweber}
    u_{k+1} = u_k - \mu_k \mathcal{G}'(u_k)^*(\mathcal{G}(u_k) - v),
\end{equation}
which coincides with the adaptive Landweber method~\cite{clason2019bouligand,hanke1991accelerated}. Consequently, all
subsequent iterations proceed according to \eqref{eqn:adaptive_landweber}, since the additional denoising term becomes inactive and may no longer contribute beneficially to the reconstruction.
\end{remark}
To establish convergence of the proposed method in the exact data setting, we show that the sequence $\{u_k\}_{k \geq 0}$ generated by 
Algorithm~\ref{alg:exactdata} is a Cauchy sequence. 
We begin by stating a preliminary result concerning the behavior 
of this sequence.
%%%%%%%%%%%%%%%%%%%%%%%%%%%%%%%%%%%%%%%%%%%%%%%%%%%%%%%%%%%%%%%%%%%%%%%%%%%%%%%%%%%%%%%%%%%%%%%%%%%%%%%%%%%%%%%%%%%%%%%%%%%%%%%%%%%%%%%%
\begin{lemma}\label{lemma: monotonicity for exact data}
    Let Assumptions~\ref{assum:existence of a solution} and~\ref{assum:denoiser} hold. Let $\{u_k\}_{k \in \mathbb{N}}$ be the sequence of iterates generated by the \Cref{alg:exactdata}. Then for the solution $u^\dagger$ of $\mathcal{G}(u)=v$, there holds
 \begin{equation}
     \|u_{k+1} - u^\dagger\|^2 - \|u_k - u^\dagger\|^2 \leq -2\mathcal C_0 \|\mathcal{G}(u_k) - v\|^2 \quad \forall \; k \in \mathbb{N},
 \end{equation}
 where $\mathcal C_0: = \gamma - \gamma_1(\gamma_0  + \zeta) - \nu_0(\nu_1 - c_q) > 0.$ Also, $u_k \in \mathcal{B}(2\wp, u_0)$ for all $k \in \mathbb{N}$ and the sequence $\{\|u_k - u^\dagger\|\}_{k \in \mathbb{N}}$ is monotonically decreasing and 
 \begin{equation}\label{sum for non noisy case}
     \sum_{k=0} ^\infty \|\mathcal{G}(u_k) - v\|^2 \leq \frac{1}{2\mathcal C_0} \|u_0 - u^\dagger\|^2 < \infty.
 \end{equation}
 This means that $\|\mathcal{G}(u_k) - v\| \to 0$ as $k \to \infty$ i.e.,  the residual norm converges to zero. 
\end{lemma}
\begin{proof}
  The assertion follows by a direct analogy with the proof of Lemma~\ref{lemma:monotonicity}. Moreover, the positivity of $\mathcal C_0$ can be established using the same arguments employed to verify the positivity of $\mathcal C$.
\end{proof}
In the following result, we discuss several consequences of Algorithm~\ref{alg:exactdata} that will be used in the subsequent analysis.
%%%%%%%%%%%%%%%%%%%%%%%%%%%%%%%%%%%%%%%%%%%%%%%%%%%%%%%%%%%%%%%%%%%%%%%%%%%%%%%%%%%%%%%%%%%%%%%%%%%%%%%%%%%%%%%%%%%%%%
\begin{lemma}\label{lem:DDIR-stagnation}
Let Assumptions~\ref{assum:existence of a solution} 
and~\ref{assum:denoiser} hold, and consider 
Algorithm~\ref{alg:exactdata}.
\begin{enumerate}
\item[\emph{(i)}]
If $\mathcal{G}(u_k) = v$ for some $k \ge 0$, then $u_m = u_k$ for all $m > k$.
\item[\emph{(ii)}]
If $u_{k+1} = u_k$ for some $k \ge 0$, then $\mathcal{G}(u_k) = v$ and $u_m= u_k$ for all $m > k$.
\item[\emph{(iii)}]
Assume that for some $k \ge 0$,
\(
u_k\in \operatorname{Fix}(D)
 \text{ and } 
u_{k+1} = u_k.
\)
Moreover, assume that
\[
\operatorname{Fix}(D) \cap \{u \in U : \mathcal{G}(u) = v\} = \{u^\dagger\}.
\]
Then $u_k = u^\dagger$.
\end{enumerate}
\end{lemma}
\begin{proof}
(i)
Suppose that $\mathcal{G}(u_k) = v$ for some $k \ge 0$. Then
\(
\mathcal{G}'(u_k)^*(\mathcal{G}(u_k)-v)=0.
\)
Moreover, by the fixed-point property of the denoiser,
\(
u_k - D_{h_k}(u_k) = 0.
\)
% {\color{blue} why does this hold? I think we may not need this propertry, as $\lambda_k$ is zero if the discrepancy vanishes} 
Hence, from the update rule,
\[
u_{k+1}
= u_k - \mu_k \mathcal{G}'(u_k)^*(\mathcal{G}(u_k)-v) - \lambda_k (u_k - D_{h_k}(u_k))
= u_k.
\]
Repeating the same argument inductively yields $u_m = u_k$ for all $m > k$.

\noindent
(ii) Assume that $u_{k+1} = u_k$ for some $k \ge 0$. Then, by the update formula,
\[
\mu_k \mathcal{G}'(u_k)^*(\mathcal{G}(u_k)-v) + \lambda_k(u_k - D_{h_k}(u_k)) = 0.
\]
Taking the inner product with $u_k - u^\dagger$ gives
\begin{align*}
0
&= \mu_k \langle \mathcal{G}'(u_k)^*(\mathcal{G}(u_k)-v), u_k - u^\dagger \rangle
 + \lambda_k \langle u_k - D_{h_k}(u_k), u_k - u^\dagger \rangle \\
 &\geq  \mu_k\langle \mathcal{G}(u_k)-v,\mathcal{G}'(u_k)(u_k - u^\dagger )\rangle + \lambda_k c_q \|u_k - D_{h_k}(u_k)\|^2 \\
&\geq \mu_k(1- \zeta) \|\mathcal{G}(u_k) - v\|^2
 + \lambda_k c_q \|u_k - D_{h_k}(u_k)\|^2 ,
\end{align*}
where we have used TCC~\eqref{eqn:TCC} and  Proposition~\ref{Prop:result}, which gives
\(
\langle u_k - u^\dagger,\; u_k - D_{h_k}(u_k) \rangle
\ge c_q \|u_k - D_{h_k}(u_k)\|^2 \ge 0.
\)
Hence,
\(
\mu_k \|\mathcal{G}(u_k) - v\|^2 \le 0,
\)
which implies
\(
\mathcal{G}(u_k) = v.
\)
Using part~(i), we conclude that $u_m = u_k$ for all $m > k$.

\noindent
(iii) Observe 
from the update rule of Algorithm~\ref{alg:exactdata}, we have
\[
u_{k+1}
= u_k - \mu_k \mathcal{G}'(u_k)^*(\mathcal{G}(u_k)-v)
- \lambda_k \bigl(u_k - D_{h_k}(u_k)\bigr).
\]
Since $u_k\in \operatorname{Fix}(D) =\operatorname{Fix}(D_{h_k})$ it implies that $u_k = D_{h_k}(u_k)$. By assumption, the denoiser term vanishes and
\(
u_{k+1} = u_k - \mu_k \mathcal{G}'(u_k)^*(\mathcal{G}(u_k)-v).
\)
Using the additional assumption $u_{k+1} = u_k$, we obtain
\(
\mu_k \mathcal{G}'(u_k)^*(\mathcal{G}(u_k)-v)=0.
\)
If $\mathcal{G}(u_k) \neq v$, then by the definition of $\mu_k$ we have $\mu_k > 0$, which yields
\(
\mathcal{G}'(u_k)^*(\mathcal{G}(u_k)-v) =0,
\)
and applying TCC~\eqref{eqn:TCC} yields $0 \ge (1-\zeta)\|\mathcal{G}(u_k)-v\|^2$,
a contradiction. Hence it must hold that
\(
\mathcal{G}(u_k) = v.
\)
Therefore, $u_k \in \operatorname{Fix}(D) \cap \{u \in U : \mathcal{G}(u) = v\}$.  
By the uniqueness assumption on this intersection, we conclude that
\(
u_k = u^\dagger.
\)
% This completes the proof.
\end{proof}
%%%%%%%%%%%%%%%%%%%%%%%%%%%%%%%%%%%%%%%%%%%%%%%%%%%%%%%%%%%%%%%%%%%%%%%%%%%%%%%%%%%%%%%%%%%%%%%%%%%%%%%%%%%%%%%%%%%%%
  \noindent In what follows, we characterize the convergence behavior for the noise-free case.
\begin{theorem}\label{Convergence for exact data}
   Suppose the conditions stated in \Cref{lemma: monotonicity for exact data} are satisfied. Then the iterates $\{u_k\}_{k \geq 0}$ produced by \Cref{alg:exactdata} converges to $u^\dagger$, which is a solution of $(\ref{eq:forward_model})$ with $\delta = 0$.
\end{theorem}
\begin{proof}
    For $\delta =0,$ we define $z_k:= u_k - u^\dagger$, $r_k:=\mathcal{G}(u_k) -v$ and $k_* := \inf\{k: k\in {\mathbb N} \mbox{ and } u_k \in \operatorname{Fix}(D)\}$. Now suppose that  $k_* >l \geq k$, we choose an integer $m$ with $l \geq m \geq k$ such that
    \begin{equation}\label{rn rp ineq}
        \|r_m\| \leq \| r_n\|, \quad \forall \hspace{1mm} k \leq n \leq l.
    \end{equation}
By the triangle inequality, we obtain the following estimate
\begin{equation}\label{xk triangular}
    \|z_l - z_k\| \leq \|z_l - z_m\| + \|z_m - z_k\|,
\end{equation}    
where 
\begin{align}
    \|z_l - z_m\|^2 &= 2\langle z_m - z_l, z_m \rangle + \|z_l\|^2 - \|z_m\|^2, \label{xl_xn} \\
    \|z_m - z_k\|^2 &= 2\langle z_k - z_m, z_m \rangle + \|z_m\|^2 - \|z_k\|^2. \label{xn_xm}
\end{align}
From Lemma \ref{lemma: monotonicity for exact data}, it follows that the sequence $\{\|z_k\|\}_{k \geq 0}$ is monotonically decreasing and bounded below by $0$. Hence, the sequence is convergent. In particular, there exists a constant $z_* \geq 0$ such that
\(
\lim_{k \to \infty} \|z_k\| = z_*.
\)
Consequently, we have
\begin{align}
     \lim_{k \to \infty }\|z_l - z_m\|^2 &= 2 \lim_{k \to \infty }\langle z_m - z_l, z_m \rangle + z_*^2 - z_*^2, \label{xl xm limit}\\
     \lim_{k \to \infty }\|z_m - z_k\|^2 &= 2 \lim_{k \to \infty }\langle z_k - z_m, z_m \rangle + z_*^2 - z_*^2.\label{xm xk limit}
\end{align} 
Our objective is to show that $\{z_k\}_{k \geq 0}$ is a Cauchy sequence. In support of this, we claim that $\langle z_m - z_l, z_m \rangle \to 0$ as $k \to \infty$. To verify this claim, observe that
\begin{align}
 \nonumber   |\langle z_m - z_l, z_m\rangle| &= |\langle u_m - u_l, z_m\rangle| \\ \nonumber
    &\hspace{-12mm} \leq \left| \left\langle \sum_{j=m}^{l-1} \mu_j \mathcal{G}'(u_j)^*(\mathcal{G}(u_j) - v) + \lambda_j(u_{j} - D_{h_j}(u_j)), u_m - u^\dagger \right\rangle \right| \\ \nonumber
     & \hspace{-12mm} \leq \sum_{j=m}^{l-1}\left| \left\langle  \mu_j \mathcal{G}'(u_j)^*(\mathcal{G}(u_j) - v), u_m - u^\dagger \right\rangle \right| + \sum_{j=m}^{l-1}\left| \left\langle \lambda_j( u_j - D_{h_j}(u_j)) , u_m - u^\dagger \right\rangle \right| \\ \nonumber
     &\hspace{-12mm} \leq \gamma_1 \sum_{j=m}^{l-1}\left| \left\langle  \mathcal{G}(u_j) - v,\mathcal{G}'(u_j)(u_m - u^\dagger) \right\rangle \right| + \sum_{j=m}^{l-1}\lambda_j \left| \left\langle  u_j - D_{h_j}(u_j), u_m - u^\dagger \right\rangle \right|\\ \nonumber
     &\hspace{-12mm} \leq \gamma_1(1+\zeta) \sum_{j=m}^{l-1}\|r_j\|(2\|r_j\| + \|r_m\|) + 3\wp\nu_0\sum_{j=m}^{l-1}\|r_j\|^2, 
\end{align}
where, we have utilized the bound $\mu_j \le \gamma_1$, decomposed $\mathcal{G}'(u_j)(u_m-u^\dagger) = \mathcal{G}'(u_j)(u_j-u^\dagger) - \mathcal{G}'(u_j)(u_j-u_m)$ to apply the TCC~\eqref{eqn:TCC} , and used the relation $\mathcal{G}(u^\dagger) =v$ in deriving the first term. For the second term we
 incorporated    the fact that $u_k \in \mathcal{B}(2\wp, u_0)$ for all $k \geq 0$,  along with the bound $\lambda_j \|u_j - D_{h_j}(u_j)\| \le \nu_0 \|r_j\|^2$ derived from the definition of $\lambda_j$. By applying (\ref{rn rp ineq}), we reformulate the preceding inequality as
\begin{align}\label{eqn:z_m x_l_in_terms_of_residual}
  |\langle z_m - z_l, z_m\rangle|  \leq 3\left(\gamma_1(1+\zeta) + \wp\nu_0\right)\sum_{j=m}^{l-1}\|r_j\|^2.
\end{align}
Proceeding similarly for $|\langle z_k - z_m, z_m \rangle|$, we may write
\begin{equation}\label{eqn:z_k x_m_in_terms_of_residual}
 |\langle z_k - z_m, z_m \rangle| \leq   3\left(\gamma_1(1+\zeta) + \wp\nu_0\right)\sum_{j=k}^{l-1}\|r_j\|^2.
\end{equation}
These bounds, together with \eqref{sum for non noisy case}, imply that $|\langle z_m - z_l, z_m \rangle| \to 0$ and $|\langle z_k - z_m, z_m \rangle| \to 0$ in the limit. Accordingly, substituting these into \eqref{xl xm limit} and \eqref{xm xk limit} yields
\begin{align*}
     \lim_{k \to \infty }\|z_l - z_m\|^2 &= 2 \lim_{k \to \infty }\langle z_m - z_l, z_m \rangle =0, \\
     \lim_{k \to \infty }\|z_m - z_k\|^2 &= 2 \lim_{k \to \infty }\langle z_k - z_m, z_m \rangle =0.
\end{align*}
Consequently, by virtue of \eqref{xk triangular}, \eqref{xl_xn}, and \eqref{xn_xm}, we establish that $\{z_k\}_{k \geq 0}$ is a Cauchy sequence. Given the definition of $z_k$, it follows immediately that $\{u_k\}_{k \geq 0}$ inherits this Cauchy property and thus converges to a limit $\hat{u}$ in the underlying space. Furthermore, as the residuals $\mathcal{G}(u_k) - v$ vanish as $k \to \infty$, the continuity of $\mathcal{G}$ ensures that $\hat{u}$ satisfies  $\mathcal{G}(\hat{u}) = v$. This concludes that $\hat{u}=u^\dagger$, confirming that the iterates converge to a solution of $\mathcal{G}(u) = v$.

\noindent For the case $k_* \leq k \leq l$, we have $\lambda_k = 0$, leading to  $\nu_0=0$ in \eqref{eqn:z_m x_l_in_terms_of_residual} and \eqref{eqn:z_k x_m_in_terms_of_residual}. The subsequent steps of the proof remain unchanged~\footnote{It remains to address the case \( k \leq k_* \leq l \), for which the proof of Theorem~\ref{Convergence for exact data} proceeds without alteration.}.
\end{proof}

\subsection{Stability}
To prove the stability of Algorithm~\ref{alg:DDIR} with its exact data counterpart Algorithm~\ref{alg:exactdata}, we require the following result on the admissible family of denoisers 
$\{D_{h}(\cdot)\}_{1> h > 0}$.
\begin{proposition}
\label{lem:Dk-delta-contract}
Let Assumption~\ref{assum:denoiser} holds. For a fixed $k \in \mathbb N$, assume that 
\(\
u_k^\delta \to u_k \text{ as } \delta \to 0.
\)
Then, for any sequence $\{h_k\} \subset (0, 1)$, it follows that
\[
\|u_k^\delta - D_{h_k}(u_k^\delta)\|
\to
\|u_k - D_{h_k}(u_k)\|
\quad \text{as } \delta \to 0.
\] 
\end{proposition}
\begin{proof}
We estimate
\[
\begin{aligned}
\big|
\|u_k^\delta - D_{h_k}(u_k^\delta)\|
-
\|u_k - D_{h_k}(u_k)\|
\big|
 &\le
\|(u_k^\delta - D_{h_k}(u_k^\delta)) - (u_k - D_{h_k}(u_k))\| \\
& =
\|(u_k^\delta - u_k) - (D_{h_k}(u_k^\delta) - D_{h_k}(u_k))\|.
\end{aligned}
\]
By the triangle inequality and Proposition~\ref{Prop:2.10},
\begin{align*}
    \|(u_k^\delta - u_k) - (D_{h_k}(u_k^\delta) - D_{h_k}(u_k))\|
&\le
\|u_k^\delta - u_k\|
+
\|D_{h_k}(u_k^\delta) - D_{h_k}(u_k)\| \\
&\le
(2-h_k (1-q))\|u_k^\delta - u_k\|.
\end{align*}
Since $u_k^\delta \to u_k$ as $\delta \to 0$, the right-hand side converges to
zero, and hence
\(
\|u_k^\delta - D_{h_k}(u_k^\delta)\|
\to
\|u_k - D_{h_k}(u_k)\|,
\)
which completes the proof.
\end{proof}
%%%%%%%%%%%%%%%%%%%%%%%%%%%%%%%%%%%%%%%%%%%%%%%%%%%%%%%%%%%%%%%%%%%%%%%%%%%%%%%%%%%%%%%%%%%%%%%%%%%%%%%%%%%%%%%%%%%%%%%%%%%%%%%%%%%%%%%%%%%%
\begin{lemma}\label{lemma: stability}
Let Assumptions~\ref{assum:existence of a solution} and \ref{assum:denoiser} be satisfied. Let $\tau > 1$, and suppose that the remaining parameters in \eqref{eqn:step size mu} and \eqref{eqn:weighted parameter lambda} are chosen such that \eqref{assum: On C} holds. Let $\{v^{\d_l}\}$ be a sequence of 
noisy data satisfying $\|v^{\d_l} - v\| \le \d_l$ with $0 < \d_l \to 0$ as $l \to \infty$,
and let $u_k^{\d_l}$, $0\le k \le k_{\mathrm{dp}}^{\d_l}$, be defined by Algorithm \ref{alg:DDIR}
using noisy data $v^{\d_l}$, where $k_{\mathrm{dp}}^{\d_l}$ denotes the corresponding stopping index. Let $\{u_k\}_{k \geq 0}$ be 
defined by  Algorithm \ref{alg:exactdata} using the exact data $v$. Then, for any 
finite integer $\hat k \le \liminf_{l \to \infty} k_{\mathrm{dp}}^{\d_l}$ there hold
$$
u_k^{\d_l}\to u_k \quad \text{ as } l \to \infty
$$
for all $0\le k \le \hat k$. 
 % Then for each fixed integer $k \geq 0$, there holds 
 % \begin{equation}
 %    u_k^\delta \to u_k \quad \text{as} \quad \delta \to 0,
 % \end{equation} 
 % where $u_k^\delta$ and $u_k$ are same as in \Cref{alg:DDIR} and \Cref{alg:exactdata}.
 \end{lemma}
\begin{proof}
Given that $\hat k \le \liminf_{l\to \infty} k_{\mathrm{dp}}^{\d_l}$, it follows  that $k^{\d_l}_{\mathrm{dp}} \ge \hat k$ for sufficiently 
large $l$. Consequently, the iterates $u_k^{\d_l}$ are well-defined for all $0\le k \le \hat k$. Let 
\begin{align*}
k_* := \inf\{k: k\in {\mathbb N} \mbox{ and } u_k \in \operatorname{Fix}(D) \}.
\end{align*}
Then $0 \le k_* \le \infty$. 
Based on the definition of $k_*$ and the invariant fixed-point  property $\operatorname{Fix}(D)  = \operatorname{Fix}(D_{h_k})$, it follows that $ \|u_k - D_{h_k}({u_k})\| \ne 0$ for $1\le k < k_*$. Conversely, as dictated by Algorithm~\ref{alg:exactdata}, we have $\|u_k - D_{h_k}(u_k)\| = 0$ for  $k_* \leq k < \infty$. 
\noindent We distinguish between two cases:

\textbf{Case (i):} Let $0 \leq k < \min\{k_*, \hat{k}\}$. In this case, the result is established by induction for all $k$ satisfying $0 \leq k < \min\{k_*, \hat{k}\}$.  For the base case \( k = 0 \), it is trivially true. Furthermore, we suppose that the assertion holds for all \( 0 \leq j \leq k \). Our objective is to demonstrate that it also holds for \( j = k+1 \), i.e., \( u_{k+1}^{\delta_l} \to u_{k+1} \) as \( l \to \infty \). 
    From (\ref{eqn:DDR}), we have 
    \begin{align}\label{x_k+1 to x_k}
    \nonumber u_{k+1}^{\delta_l} - u_{k+1} &= (u_k^{\delta_l} - u_k) -  \left( \mu_k^{\delta_l} \mathcal{G}'(u_k^{\delta_l})^*(\mathcal{G}(u_k^{\delta_l}) - v^{\delta_l}) - \mu_k \mathcal{G}'(u_k)^*(\mathcal{G}(u_k) - v) \right)\\
     & \quad -\left(\lambda_k^{\delta_l} (u_k^{\delta_l} - D_{h_k}(u_{k}^{\delta_l})) - \lambda_k (u_k - D_{h_k}(u_k)) \right).
    \end{align}
We begin by establishing that  
\begin{equation}\label{Eqn: (3.26)}
    \mu_k^{\delta_l} \mathcal{G}'(u_k^{\delta_l})^*(\mathcal{G}(u_k^{\delta_l}) - v^{\delta_l})\to \mu_k \mathcal{G}'(u_k)^*(\mathcal{G}(u_k) - v) \quad \text{as } l \to \infty.
\end{equation}
To this end, we utilize that \( u_k^{\delta_l} \to u_k \) and \( v^{\delta_l} \to v \) as \( l\to \infty\). When $\mathcal{G}(u_k) - v = 0$, using $0 \leq \mu_k^{\delta_l} \leq \gamma_1$, it follows that
\[
   \left\| \mu_k^{\delta_l} \mathcal{G}'(u_k^{\delta_l})^*(\mathcal{G}(u_k^{\delta_l}) - v^{\delta_l}) - \mu_k \mathcal{G}'(u_k)^*(\mathcal{G}(u_k) - v) \right\|
   % = \left\| \mu_k^\delta M^*(Mu_k^\delta - v^\delta)\right\|
   \leq \gamma_1 \left\| \mathcal{G}'(u_k^{\delta_l})^*(\mathcal{G}(u_k^{\delta_l}) - v^{\delta_l})\right\|.
\]
Since $\mathcal{G}(u_k^{\delta_l}) \to \mathcal{G}(u_k)$ by Assumption~\ref{assum:existence of a solution} (A2) and $v^{\delta_l} \to v$ as $l \to \infty$, we obtain
\[
   \gamma_1\left\|  \mathcal{G}'(u_k^\delta)^*(\mathcal{G}(u_k^\delta) - v^\delta)\right\| 
   \to \gamma_1 \left\| \mathcal{G}'(u_k)^*(\mathcal{G}(u_k) - v)\right\| = 0.
\]
On the other hand, when $\mathcal{G}(u_k) - v \neq 0$, we may use TCC~\eqref{eqn:TCC} to have
\begin{align*}
     \langle \mathcal{G}'(u_k)^*(\mathcal{G}(u_k) - v),&\, u_k - u^\dagger \rangle
   = \langle \mathcal{G}(u_k) - v,\, \mathcal{G}'(u_k)(u_k - u^\dagger) \rangle \\
   & = \|\mathcal{G}(u_k) - v\|^2 + \langle\mathcal{G}(u_k) - v, v - \mathcal{G}(u_k) - \mathcal{G}'(u_k)(u^\dagger - u_k)  \rangle \\
   &\geq \|\mathcal{G}(u_k) - v\|^2 - \|\mathcal{G}(u_k) - v\| \|v - \mathcal{G}(u_k) - \mathcal{G}'(u_k)(u^\dagger - u_k) \| \\
   & \geq (1 - \zeta)\|\mathcal{G}(u_k) - v\|^2 > 0,
\end{align*}
which implies that $\mathcal{G}'(u_k)^*(\mathcal{G}(u_k) - v) \neq 0$ and we get 
\[
\mu_k^{\delta_l} = \min \left\{ \frac{\gamma_0 \|\mathcal{G}(u_k^{\delta_l}) - v^{\delta_l}\|^2}{\|\mathcal{G}'(u_k^{\delta_l})^*(\mathcal{G}(u_k^{\delta_l}) - v^{\delta_l}))\|^2}, \gamma_1 \right\}
\to
\min \left\{ \frac{\gamma_0 \|\mathcal{G}(u_k) - v\|^2}{\|\mathcal{G}'(u_k)^*(\mathcal{G}(u_k) - v)\|^2}, \gamma_1 \right\}
\]
as $l \to \infty$. Consequently, we obtain
\begin{align*}
&\left\| \mu_k^{\delta_l} \mathcal{G}'(u_k^{\delta_l})^*(\mathcal{G}(u_k^{\delta_l}) - v^{\delta_l}) - \mu_k \mathcal{G}'(u_k)^*(\mathcal{G}(u_k) - v) \right\|  \\
& \hspace{30mm} \leq \mu_k^{\delta_l} \left\| \mathcal{G}'(u_k^{\delta_l})^*(\mathcal{G}(u_k^{\delta_l}) - v^{\delta_l}) -  \mathcal{G}'(u_k)^*(\mathcal{G}(u_k) - v) \right\| \nonumber \\
& \hspace{30mm} + \left| \mu_k^{\delta_l} - \mu_k \right| \left\| \mathcal{G}'(u_k)^*(\mathcal{G}(u_k) - v)\right\| \to 0 \text{ as } l \to \infty. \nonumber
\end{align*}
Thus, \eqref{Eqn: (3.26)} is established. By an analogous argument, we also obtain that
\begin{equation}\label{eqn: (3.24)}
    \lambda_k^{\delta_l} (u_k^{\delta_l} - D_{h_k}(u_k^{\delta_l}))
\to
\lambda_k (u_k - D_{h_k}(u_k)).
\end{equation}
Recall that \( \|u_k - D_{h_k}(u_k) \| \neq 0 \) for $1\le k < k_*$, we thus have \( \|u_k^{\delta_l} - D_{h_k}(u_k^{\delta_l}) \| \neq 0 \) for large $l$. Then, by the definitions of \( \lambda_k^{\delta_l} \) and \( \lambda_k \), we obtain
\[
\lambda_k^{\delta_l} = 
\min \left\{ \frac{\nu_0 \|\mathcal{G}(u_k^{\delta_l}) - v^{\delta_l}\|^2}{\|u_k^{\delta_l} - D_{h_k}(u_k^{\delta_l})\|},  \nu_1 \right\} \to 
\min \left\{ \frac{\nu_0 \|\mathcal{G}(u_k) - v\|^2}{\|u_k - D_{h_k}( u_k)\|}, \nu_1 \right\} = \lambda_k
\]
as  $l \to \infty.$ This convergence follows from the induction hypothesis and Proposition~\ref{lem:Dk-delta-contract}, which ensures that\( \|u_k^{\delta_l} - D_{h_k}(u_k^{\delta_l})\| \to \|u_k - D_{h_k}( u_k)\| \) as \( l \to \infty \). Now,
we decompose the difference as
\[
\begin{aligned}
&\lambda_k^{\delta_l} \big(u_k^{\delta_l} - D_{h_k}(u_k^{\delta_l})\big)
-
\lambda_k \big(u_k - D_{h_k}(u_k)\big)
\\
&\quad =
(\lambda_k^{\delta_l} - \lambda_k)\big(u_k^{\delta_l} - D_{h_k}(u_k^{\delta_l})\big)
+
\lambda_k\Big[
(u_k^{\delta_l} - u_k) - \big(D_{h_k}(u_k^{\delta_l}) - D_{h_k}(u_k)\big)
\Big].
\end{aligned}
\]
Taking norms and using the triangle inequality, we obtain
\[
\begin{aligned}
\big\|
\lambda_k^{\delta_l} (u_k^{\delta_l} - D_{h_k}(u_k^{\delta_l}))
-
\lambda_k (u_k - D_{h_k}(u_k))
\big\|
& \le
|\lambda_k^{\delta_l} - \lambda_k|
\,
\|u_k^{\delta_l} - D_{h_k}(u_k^{\delta_l})\|
\\
& \hspace{-10mm} +|\lambda_k|\Big(\|u_k^{\delta_l} - u_k\|+\|D_k(u_k^{\delta_l}) -D_{h_k}(u_k)\|\Big).
\end{aligned}
\]
Since \(u_k^{\delta_l} \to u_k\) and \(D_{h_k}(u_k^{\delta_l}) \to D_{h_k}(u_k)\),
the sequence \(\|u_k^{\delta_l} - D_{h_k}(u_k^{\delta_l})\|\) is bounded.
Together with \(\lambda_k^{\delta_l} \to \lambda_k\), this implies that the first term
converges to zero.
The second term converges to zero by the assumed convergences of
\(u_k^{\delta_l}\) and \(D_{h_k}(u_k^{\delta_l})\).
Therefore, the result stated in \eqref{eqn: (3.24)} follows.
% \begin{equation}\label{eqn: (3.24)}
%     \lambda_k^{\delta_l} (u_k^{\delta_l} - D_{h_k}(u_k^{\delta_l}))
% \to
% \lambda_k (u_k - D_{h_k}(u_k)).
% \end{equation}
% where the  boundedness of $\|\Delta_{u_k^\delta}\|$ follows directly from \cref{eqn:last_eqn} by setting $u_k = 0$, which yields $\| \Delta_{u_k^\delta}\| \leq \| \Delta_0\| + \mathcal{H}  \|u_k^\delta\|.$
% Hence, applying Lemma~\ref{lemma: contraction of Delta} to the above expression together with the induction hypothesis, we obtain
% \[
% \beta_k^\delta \Delta_{u_k^\delta} u_k^\delta \to \beta_k \Delta_{u_k} u_k \quad \text{as } \delta \to 0.
% \]
Consequently, it follows from \eqref{x_k+1 to x_k}, \eqref{Eqn: (3.26)}, \eqref{eqn: (3.24)} and the induction hypothesis that
\(
u_{k+1}^\delta \to u_{k+1} \text{ as } \delta \to 0.
\) This completes the proof  for 
$0\le k < \min\{k_*, \hat k\}$.

\textbf{Case (ii):} Let $k_* \le k \le \hat k$.
We again use an induction argument to show that 
\(
 u_k^{\d_l} \to u_k  \mbox{ as } l \to \infty 
\)
for $k_* \le k \le \hat k$. Recall $\|u_{k_*} - D_{h_{k_*}}(u_{k_*})\| =0$, which implies that $\|u_{k+1} - D_{h_{k+1}}(u_{k+1})\| = 0$ for all $k \ge k_*-1$. Hence, for $k > k_*$, the iteration reduces to a gradient descent step. We first establish the induction hypothesis for the index \(k = k_*\), that is, \(u_{k_*}^{\delta_l} \to u_{k_*}\) as \(l \to \infty\). This follows directly from Case (i).
% Note that  
% $$
% u_{k_*} = u_{k_*-1} + \mu_{k_*-1} \mathcal{G}'(u_{k_* -1})^*(\mathcal{G}(u_{k_* -1}) - v).
% $$
% Therefore, by the definition of $u_{k_*}^{\d_l}$, $0\le \lambda_{k_*-1}^{\d_l} \le \nu_1$ 
% we have 
% \begin{align*}
% \|u_{k_*}^{\d_l} - u_{k_*}\| 
% & \le \|u_{k_*-1}^{\d_l} - u_{k_*-1}\| + \nu_1 \| u_{k_*-1}^{\d_l} - D_{h_{k_*-1}}(u_{k_*-1}^{\d_l})\| \\
% & \quad \, + \|\mu_{k_*-1}^{\d_l} \mathcal{G}'(u_k^{\d_l})^*(\mathcal{G}(u_{k_*}^{\d_l}) - v^{\d_l}) - \mu_{k_*-1} \mathcal{G}'(u_{k_*})^*(\mathcal{G}(u_{k_*}) - v)\| \\
% & \to \nu_1 \|u_{k_*-1}  - D_{h_{k_* -1}}(u_{k_*-1})\| = 0
% \end{align*}
% and consequently, we have $u_{k_*}^{\d_l} \to u_{k_*}$ as 
% $l \to \infty$.
Now assume that the assertion holds for $k_*\le k \le n$ for some $k_*\le n < \hat k$. 
% By using $r_k =0$ and $m_k=0$ we have $\xi_{k+1} = \xi_k$. 
Thus, by the induction hypothesis and results established in case (i), we have 
\begin{align*}
\|u_{n+1}^{\delta_l} - u_{n+1}\|
&\leq \|u_n^{\delta_l} - u_n\| + \|
\mu_n^{\delta_l} \mathcal{G}'(u_n^{\delta_l})^*
(\mathcal{G}(u_n^{\delta_l}) - v^{\delta_l})
- \mu_n \mathcal{G}'(u_n)^*
(\mathcal{G}(u_n) - v)
\| \\
&\quad + \lambda_n^{\delta_l}
\|u_n^{\delta_l} - D_{h_n}(u_n^{\delta_l})\| \to \nu_1 \|u_n - D_{h_n}(u_n)\|=0.
\end{align*}
as $l \to \infty$. This establishes that the assertion also holds for $k = n+1$.  
Combining the results obtained in Case (i) and Case (ii), the proof is complete.
\end{proof}
\subsection{Convergence for noisy data}
We are now in a position to establish the main strong convergence result for Algorithm~\ref{alg:DDIR} applied to \eqref{eq:forward_model} with noisy data.
\begin{theorem}\label{AHB.thm4}
Let Assumptions~\ref{assum:existence of a solution} and \ref{assum:denoiser} hold and consider Algorithm \ref{alg:DDIR}, 
where  $\tau>1$ and the other parameters of \eqref{eqn:step size mu} and \eqref{eqn:weighted parameter lambda} are chosen such that \eqref{assum: On C} holds. 
Let $k_\mathrm{dp}$ be the output integer. Then there exists a solution $u^\dagger$ of \eqref{eq:forward_model} in 
$B(\wp, u_0)$ such that 
$$
\lim_{\d \to 0}\|u^\delta_{k_\mathrm{dp}} - u^\dagger\|=0.
$$
% If, in addition, $\emph{Ker}((x^\dag))\subset\emph{Ker}(L(x))$ for all $x\in B_{2\rho}(x_0)$, 
% then $x^* = x^\dag$.
\end{theorem}
		
\begin{proof}
Let $u^\dagger$ denote the solution of \eqref{eq:forward_model} identified in Theorem~\ref{Convergence for exact data}, for which $\|u_k - u^\dagger\| \to 0$ as $k \to \infty$, where $\{u_k\}$ is the sequence generated by the counterpart of Algorithm~\ref{alg:DDIR} with exact data. We aim to show that $\|u_{k_{\mathrm{dp}}}^\delta - u^\dagger\| \to 0$ as $\delta \to 0$. Now assume that there is a sequence $\{v^{\d_l}\}$ of noisy data satisfying $\|v^{\d_l}-v\|\le \d_l$ 
with $\d_l \to 0$. Let $k_l:=k^{\d_l}_{\mathrm{dp}}$ be the corresponding stopping index.

We prove the claim by distinguishing two cases as $l \to \infty$. 
\begin{enumerate}[label=(\roman*)]
    \item $k_l \to \hat k, \quad \hat k \in \mathbb{N}_0$;
    \item $k_l \to +\infty.$
\end{enumerate}
			
\noindent\textbf{Case (i):} In this case we can assume that  $k_l = \hat{k}$ for all large $l$. According to the definition of $k_l:=k^{\d_l}_{\mathrm{dp}}$ 
we have
$$
\|\mathcal{G}(u_{\hat{k}}^{\d_l}) - v^{\d_l}\| \le \tau \d_l.
$$
By taking $l\to\infty$ and using Lemma~\ref{lemma: stability} along with Assumption~\ref{assum:existence of a solution}, we can obtain $\mathcal{G}(u_{\hat{k}}) = v$. Thus, 
we may use Lemma \ref{lem:DDIR-stagnation} (i) to obtain $u_k = u_{\hat k}$ for all $k \ge \hat k$. Since 
$u_k \to u^\dagger$ as $k \to \infty$, we must have $u_{\hat k} = u^\dagger$ and thus, 
$$
\lim_{\d \to 0}\|u^\delta_{k_\mathrm{dp}} - u^\dagger\|=0.
$$	
\noindent\textbf{Case (ii):}
Given $\varepsilon > 0$, it follows from Theorem~\ref{Convergence for exact data} that there exists $k^* = k^*(\varepsilon)$ such that
\begin{equation}\label{eqn: estimate}
\|u_k - u^\dagger\| < \frac{\varepsilon}{2}, \quad \forall\, k > k^*.
\end{equation}
Let us fix a specific integer $k > k^*$. Since $k_l \to \infty$, there exists $l_1$ such that $k_l > k$ for all $l > l_1$. For such $l$, Lemma~\ref{lemma:monotonicity} implies that the error is monotonically non-increasing up to index $k_l$, allowing us to write
\begin{equation}\label{eqn: estimate_split}
\|u_{k_l}^{\delta_l} - u^\dagger\| \leq \|u_{k}^{\delta_l} - u^\dagger\|
\leq \|u_{k}^{\delta_l} - u_k\| + \|u_k - u^\dagger\|.
\end{equation}
Furthermore, for this fixed $k$, Lemma~\ref{lemma: stability} guarantees that $u_k^{\delta_l} \to u_k$ as $l \to \infty$. Thus, there exists $l_2$ such that for all $l > l_2$,
\begin{equation}\label{eqn:stability_estimate}
\|u_k^{\delta_l} - u_k\| < \frac{\varepsilon}{2}.
\end{equation}
Choosing $l^* = \max\{l_1, l_2\}$, we combine \eqref{eqn: estimate}, \eqref{eqn: estimate_split}, and \eqref{eqn:stability_estimate} to conclude that for all $l > l^*$,
\[
\|u_{k_l}^{\delta_l} - u^\dagger\| < \frac{\varepsilon}{2} + \frac{\varepsilon}{2} = \varepsilon.
\]
This implies that $u_{k_l}^{\delta_l} \to u^\dagger$ as $l \to \infty$, completing the proof.
\end{proof}

%%%%%%%%%%%%%%%%%%%%%%%%%%%%%%%%%%%%%%%%%%%%%%%%%%%%%%%%
\section{Numerical experiments and discussion}\label{sec:numerical}
In this section, we assess the performance of the proposed \texttt{DDIR} framework as in Algorithm~\ref{alg:DDIR}. The primary objective of this work is to establish a theoretical guarantee of convergence and to prove that \texttt{DDIR} method is a convergent regularization, we will do so by analyzing our method on image deblurring and phase retrieval CT problems. 
 In addition to that, we compare our method with baseline FBP  and PnP~\cite{ebner2024plug} methods for the task of image deblurring. 
% \textcolor{blue}{Moreover, we also provide the comparison with \texttt{IRMGL+$\Psi$} method for the task of phase retrieval CT~\cite{fannjiang2020numerics}.}
In this study, we consider three denoisers within the \texttt{DDIR} framework. First, the median filter is employed as a simple yet effective regularizer, illustrating the flexibility of \texttt{DDIR} in accommodating basic denoising operators. Secound, we incorporate the TNRD method, representing a high-performance, state-of-the-art  denoiser. Third, we use a proximal-based TV denoiser, which is theoretically well-motivated and promotes piecewise smooth reconstructions.

\subsection{Denoisers}
Let $U=\mathbb{R}^{n}$ denote the discrete image space endowed with the Euclidean inner product. An image $u\in U$ is identified with a vectorized grayscale image of size $n=N_xN_y$.
In this section we define the following realizations of the denoiser $D:U\to U$ used in Algorithm~\ref{alg:DDIR}.

%%%%%%%%%%%%%%%%%%%%%%%%%%%%%%%%%%%%%%%%%%%%%%%%%%%%%%%%%%%%
\textbf{Median filter.}
Let $\Omega\subset\mathbb{Z}^2$ denote the pixel grid.
For each pixel $x\in\Omega$ of $u$ let $W(x)\subset\Omega$ be a fixed square window of size $w\times w$ centered at $x$.
The median denoiser $D_{\mathrm{med}}:U\to U$ is defined componentwise by
\(
\big(D_{\mathrm{med}}(u)\big)(x)
=
\operatorname{median}\{u(y): y\in W(x)\}.
% \label{eq:median_definition}
\)
For each pixel $x$, the median value admits the variational characterization
\begin{equation*}
\big(D_{\mathrm{med}}(u)\big)(x)
=
\arg\min_{m\in\mathbb{R}}
\sum_{y\in W(x)} |u(y)-m|,
\label{eq:median_variational}
\end{equation*}
i.e., it minimizes a local $\ell_1$ fidelity functional.
Hence the median filter is particularly robust to impulsive noise.
A detailed discussion of order-statistics filters and their robustness properties can be found in~\cite{huang1979fast}.
In our experiments we employ $D_{\mathrm{med}}(u)$ with $w=3$.
% The averaged median denoiser used in Algorithm~\ref{alg:DDIR} is given by
% \begin{equation}
% D_h^{\mathrm{med}}(u)
% =
% (1-h)u + h D_{\mathrm{med}}(u).
% \label{eq:Dh_median}
% \end{equation}

%%%%%%%%%%%%%%%%%%%%%%%%%%%%%%%%%%%%%%%%%%%%%%%%%%%%%%%%%%%%
\textbf{TNRD.}
As second denoiser we employ the Trainable Nonlinear Reaction Diffusion (TNRD) model introduced in~\cite{chen2016trainable}.
Let $K_i:\mathbb{R}^{n}\to\mathbb{R}^{n}$, $i=1,\dots,N_t$, denote convolution operators with kernels $k_i\in\mathbb{R}^{r\times r}$ and
let $\phi_i:\mathbb{R}\to\mathbb{R}$ be nonlinear influence functions. In our experiments we use $N_t = 24$ filters with a kernel size of $r = 5$. To ensure a robust baseline for the denoising component, we utilize pretrained Gaussian denoising parameters from \cite{chen2016trainable}. Now
define the regularization functional as
\begin{equation*}
R(u)
=
\sum_{i=1}^{N_t}\sum_{x\in\Omega}
\rho_i\big((K_i\circ u)(x)\big),
\qquad
\phi_i = \rho_i'.
\label{eq:TNRD_energy}
\end{equation*}
A single diffusion step of the TNRD model reads
\(
D_{\mathrm{TNRD}}(u)
=
u
-
\kappa
\sum_{i=1}^{N_t}
K_i^\top \phi_i(K_i \circ u),
\label{eq:TNRD_single_step}
\)
where $\kappa>0$ denotes the diffusion step size and $K_i^\top$ is the adjoint convolution operator.
Thus
\(
D_{\mathrm{TNRD}}(u)
=
u - \kappa \nabla R(u),
\)
i.e., the TNRD denoiser corresponds to one gradient descent step applied to the learned regularizer $R$.
% In the multi-stage setting with $T$ stages,
% \begin{equation}
% u^{t+1}
% =
% u^t
% -
% \tau_t
% \sum_{i=1}^{N_f}
% K_i^{t\top}
% \phi_i^t(K_i^t u^t),
% \qquad t=0,\dots,T-1,
% \end{equation}
% and we define
% \begin{equation}
% D_{\mathrm{TNRD}}(u) := u^{T}.
% \end{equation} 
% The averaged operator used in Algorithm~\ref{alg:DDIR} is
% \begin{equation}
% D_h^{\mathrm{TNRD}}(u)
% =
% (1-h)u + h D_{\mathrm{TNRD}}(u).
% \label{eq:Dh_TNRD}
% \end{equation}
% and hence
% \begin{equation}
% u - D_h^{\mathrm{TNRD}}(u)
% =
% h\tau
% \sum_{i=1}^{N_f}
% K_i^\top \phi_i(K_i u).
% \end{equation}

% Therefore, within Algorithm~\ref{alg:DDIR}, the regularization term
% \[
% \lambda_k^\delta \big(u_k^\delta - D_h(u_k^\delta)\big)
% \]
% corresponds to a relaxed local $\ell^1$-type correction in the median case, and to a learned gradient regularization term in the TNRD case.

\textbf{TV proximal.} As discussed in Section~\ref{sec:assumptions}, we also adopt the scaled proximity operator $\frac{1}{1+\omega}\operatorname{prox}_{\omega g}$, where $\operatorname{prox}_{\omega g}$ is defined in Definition~\ref{eq:prox_def}, as a denoiser associated with the two-dimensional total variation (TV) functional, i.e.,
\begin{equation*}
g(y)=\mathrm{TV}(y) = \sum_{i,j} \sqrt{|y_{i+1, j} - y_{i, j}|^2 + |y_{i, j+1} - y_{i,j}|^2}.
\end{equation*}
In our experiments, we use $\omega = 0.2$ and the corresponding 2D TV denoising step is implemented using the algorithm proposed in \cite{chambolle2004algorithm}. Furthermore, according to \cite[Remark 18]{ebner2024plug}, the scaled operator $ D_{\mathrm{TV}}(u) = \frac{1}{1+\omega}\operatorname{prox}_{\omega \mathrm{TV}}$ satisfies Assumption~\ref{assum:denoiser}.

The value of $q$ as in Assumption~\ref{assum:denoiser} for each denoiser is estimated empirically by sampling 100 random image pairs $u,v$ and computing
\(
q_i = \frac{\|D(u_i) - D(v_i)\|}{\|u_i - v_i\|}.
\)
The maximum value across all samples provides an estimate of the Lipschitz constant.
% \begin{figure}[htbp]
%     \centering
%     \includegraphics[width=0.6\textwidth]{lipschitz_boxplot.png}
%     \caption{Reconstructed image with $\delta = 0.005$.}
%     \label{fig:q_values}
% \end{figure}
\begin{table}[htbp]
\centering
\footnotesize
\begin{tabular}{lcc}
\hline
\textbf{Denoiser} & \textbf{Estimated value of $q$} & $c_q=\frac{1-q}{(1+q)^2}$ \\
\hline
Median & 0.5295 & 0.2011 \\
TNRD & 1.4058 & -0.0701 \\
TV proximal & 0.8333 & 0.0496 \\
\hline
\end{tabular}
\caption{Empirical estimates of $q$  obtained using random image pairs.}
\label{tab:lipschitz_constants}
\end{table}
The results presented in Table~\ref{tab:lipschitz_constants} shows that the TV proximal denoiser satisfies the theoretical contractivity assumption with
\(
q = \frac{1}{1+\omega},
\)
while the median filter behaves contractively in practice. The TNRD implementation, however, appears to be expansive.
\begin{remark}
 Numerical evaluations involving the non-contractive TNRD denoiser are included to highlight the empirical stability of our framework. These experiments confirm that, in practice, the algorithm often converges even when the theoretical assumptions placed on the denoising operator are not fully satisfied.
\end{remark}

%%%%%%%%%%%%%%%%%%%%%%%%%%%%%%%%%%%%%%%%%%%%%%%%%%%%%%%%%%%%%%%%%%%%%%%%%%%%
\subsection{Experimental setup}\label{experimental setup}
In all simulations, the noisy data $v^\delta$ is generated by $v^\delta = \mathcal{G}(u) + \delta_{\mathrm{rel}}\|\mathcal{G}(u)\|z,$ where $u$ represents the ground truth, $\delta_{\mathrm{rel}}$ is the relative noise level, and $z$ is the Gaussian noise satisfying $\|z\|=1$ so that the noise level $\delta = \delta_{\mathrm{rel}}\|\mathcal{G}(u)\|$. For the numerical validation of the theoretical results, the parameters of Algorithm~\ref{alg:DDIR} are selected to ensure that the positivity condition of $\mathcal{C}$ (cf.~\eqref{assum: On C}) is satisfied. Accordingly, we set $\nu_0 = 0.1$, $\nu_1 = 0.3$, $\gamma_0 = 0.1$, $\gamma = 0.3$, $\gamma_1 = 0.4$, and $\tau = 2$, with the sequence $h_k = \frac{1}{k+1}$ for linear inverse problems. The maximum number of iterations is set to 1000. In this case, since $\zeta = 0$ (for linear case), all feasibility conditions required to ensure the positivity of $\mathcal{C}$ are satisfied, as indicated in Remark~\ref{eqn:C_for_linear}, for both median and TV proximal denoisers, using the corresponding values of $c_q$ listed in Table~\ref{tab:lipschitz_constants}. The qualitative evaluation of reconstructed images are performed using PSNR [dB], SSIM~\cite{wang2004image} and Relative Error (RE), where
$$\mathrm{RE}(u) = \frac{\|u^\delta_{k_{\mathrm{dp}}} - u^\dagger\|}{\|u^\dagger\|}, \qquad \mathrm{PSNR}(u) = 20\log_{10}\left( \frac{1}{\|u^\dagger - u^\delta_{k_{\mathrm{dp}}}\|} \right),$$
with $u^\delta_{k_{\mathrm{dp}}}$ denoting the reconstructed image.
%%%%%%%%%%%%%%%%%%%%%%%%%%%%%%%%%%%%%%%%%%%%%%%%%%%%%%%%%%%%%%%%%
\subsection{Image deblurring}
In accordance with the image deblurring framework established in~\cite{vogel2002computational}, we consider a forward operator  
$\mathcal{G}: L^2(\Delta) \to L^2(\Delta)$, with $\Delta \subset \mathbb{R}^2$, which is modeled as a spatial convolution with a Gaussian point spread function. Specifically, for $u \in L^2(\Delta)$, the operator $\mathcal{G}$ acts as  
\[
(\mathcal{G}u)(w_1, w_2) = \int_{\Delta} \mathcal{A}_\varphi(w_1 - w_1',\, w_2 - w_2') \, u(w_1', w_2') \, dw_1' \, dw_2',
\]
where $\mathcal{A}_\varphi$ denotes the Gaussian kernel given by  
\(
\mathcal{A}_\varphi(w_1, w_2) = \frac{1}{2\pi \varphi^2} \exp\left(-\frac{w_1^2 + w_2^2}{2\varphi^2}\right),
\)
and the parameter $\varphi > 0$ controls the strength of blurring. In the experimental setup, the domain $\Delta$ is partitioned into a $256 \times 256$ uniform pixel grid. The continuous forward operator $\mathcal{A}_\varphi$ is then approximated by a discrete convolution operator, denoted by  
\(
A_\varphi: \mathbb{R}^{256 \times 256} \to \mathbb{R}^{256 \times 256}.
\)
Accordingly, the action of the operator $G$ on a discrete image $u \in \mathbb{R}^{256 \times 256}$ is defined as
\[
(Gu)_{i,j} = \sum_{k,\ell} A_\varphi(i - k, j - \ell) \cdot u_{k,\ell}, \quad i, j, k, \ell \in \{0, \ldots, 255\}.
\]
For the numerical implementation, the convolution operation is carried out using the \texttt{scipy.ndimage.gaussian\_filter} routine available in Python. Let $u^\dagger \in \mathbb{R}^{256 \times 256}$ represent the true underlying image. The corresponding noisy observation $v^\delta \in \mathbb{R}^{256 \times 256}$ is then formulated as 
\[
v^\delta = G u^\dagger + \eta, \qquad \|\eta\| \leq \delta,
\]
where \( \eta \in \mathbb{R}^{256 \times 256} \) denotes the additive measurement noise, with $\delta > 0$ representing the prescribed noise level.
To recover $u^\dagger$ from the noisy observation $v^\delta$, we employ \Cref{alg:DDIR}. In this experiment, we take $\varphi = 1.5$ and all remaining experimental parameters are chosen in accordance with those described in~\Cref{experimental setup}. 

\subsubsection{Reconstruction and quantitative results}
The effectiveness of the proposed method~\eqref{eqn:DDR} with $\delta_{\mathrm{rel}}=0.005$ is illustrated in Fig.~\ref{fig:reconstruction_results} for four different images from \texttt{skimage} library and the reconstruction results with $\delta_{\mathrm{rel}}=0.001, 0.0005$ are moved Appendix~\ref{app:extra_results}. Additional quantitative results are presented in Table~\ref{tab:shepp_full}, Table~\ref{tab:coins_full}, Table~\ref{tab:moon_full}, and Table~\ref{tab:text_full}, which summarize the reconstruction performance across various noise levels using the initial estimate $u_0^\delta = v^\delta$. These tables report the stopping index $k_{\mathrm{dp}}$, RE, PSNR, and SSIM, thereby providing a comprehensive assessment of the reconstruction quality. The relative error plots for the image \texttt{shepp\_logan} are shown in Fig.~\ref{fig:rel_error_plots} across different noise levels. A key advantage of the proposed approach over RED~\cite{romano2017little}, RED-PRO~\cite{cohen2021regularization} and PnP~\cite{ebner2024plug} lies in the fact that the stopping index is not fixed \emph{a priori}, but is instead determined in an \emph{a posteriori} manner.
\begin{figure}[htbp]
\centering
\includegraphics[width=0.45\textwidth]{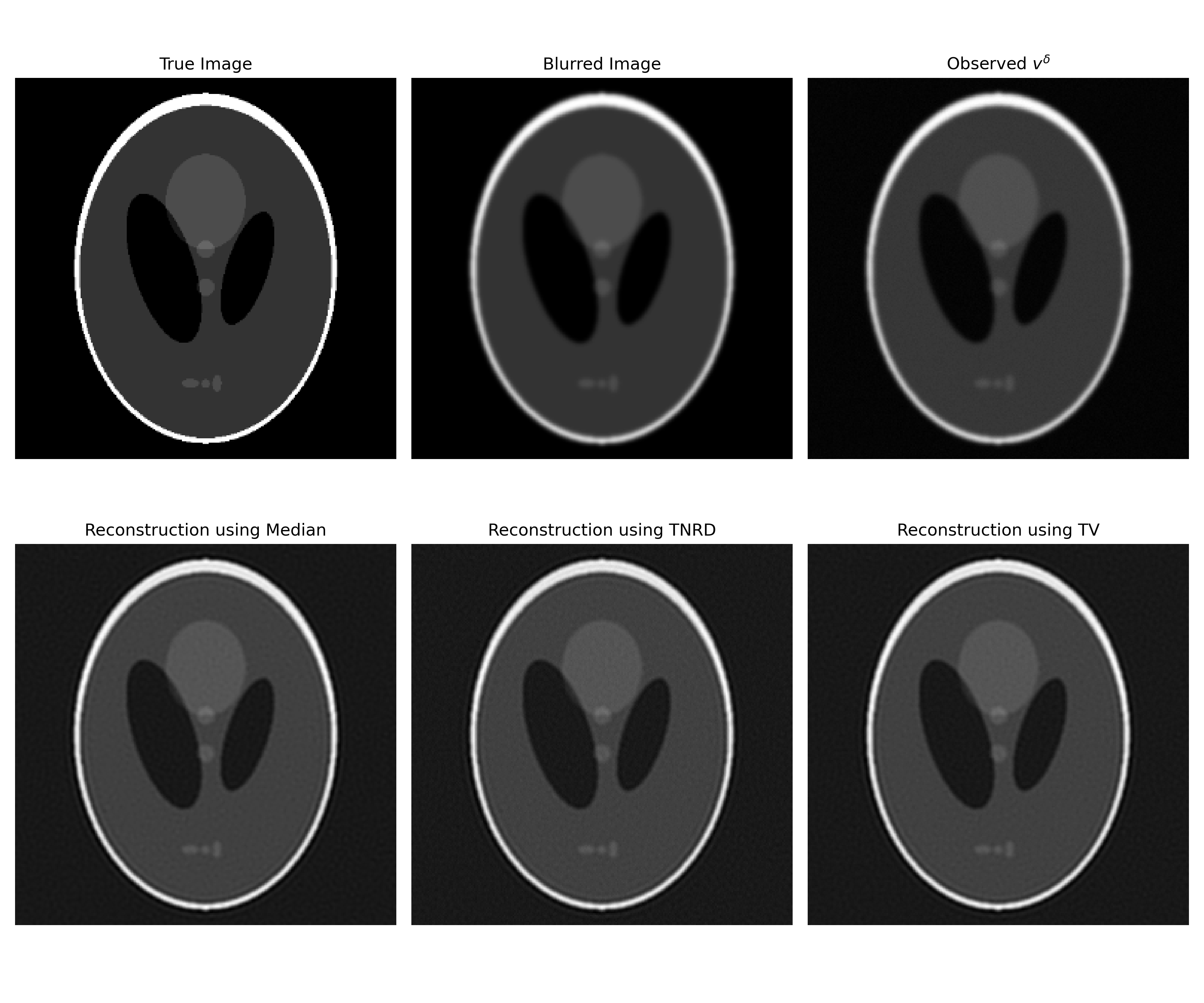}
\hfill
\includegraphics[width=0.45\textwidth]{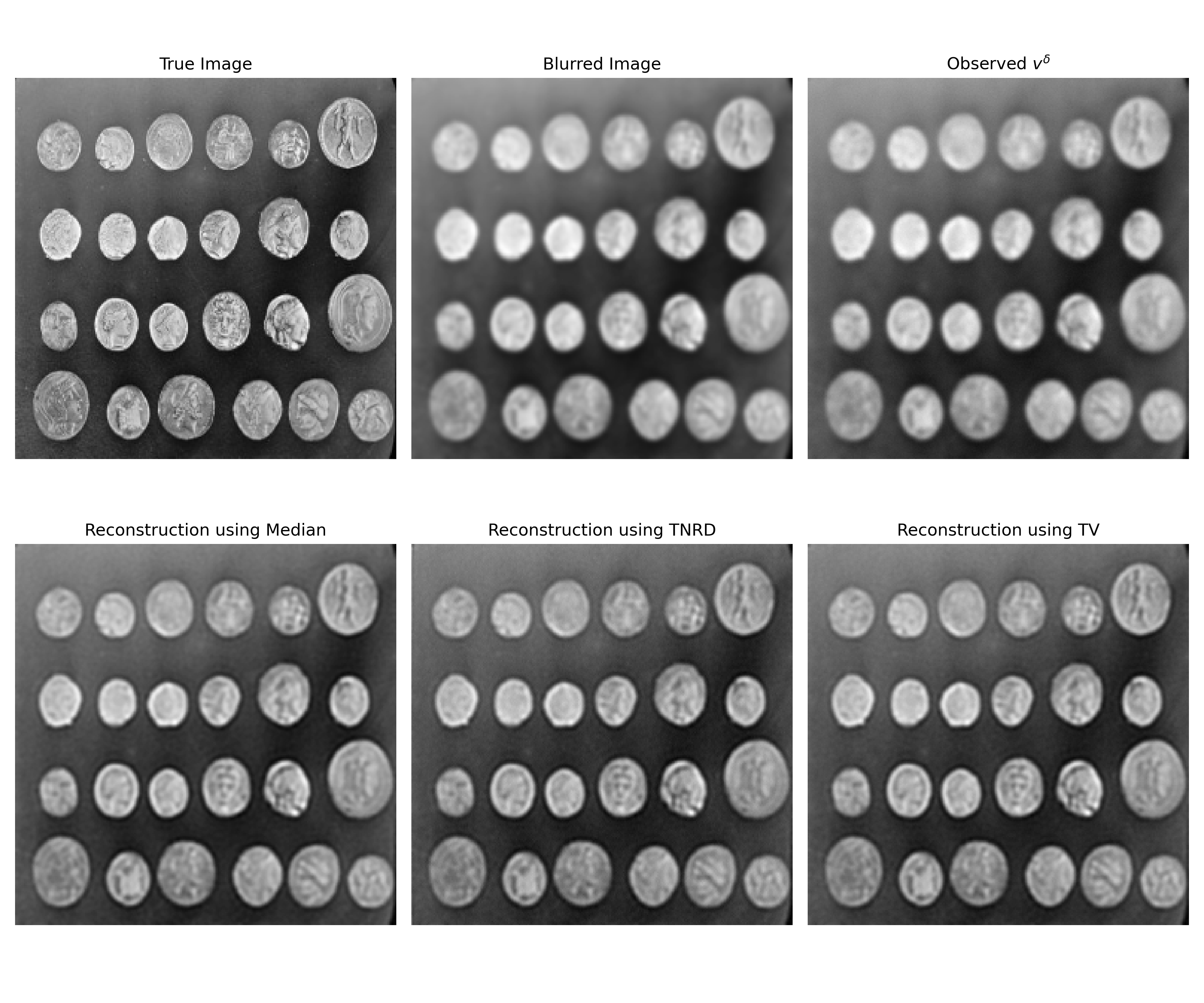}

\vspace{0.5cm}

\includegraphics[width=0.45\textwidth]{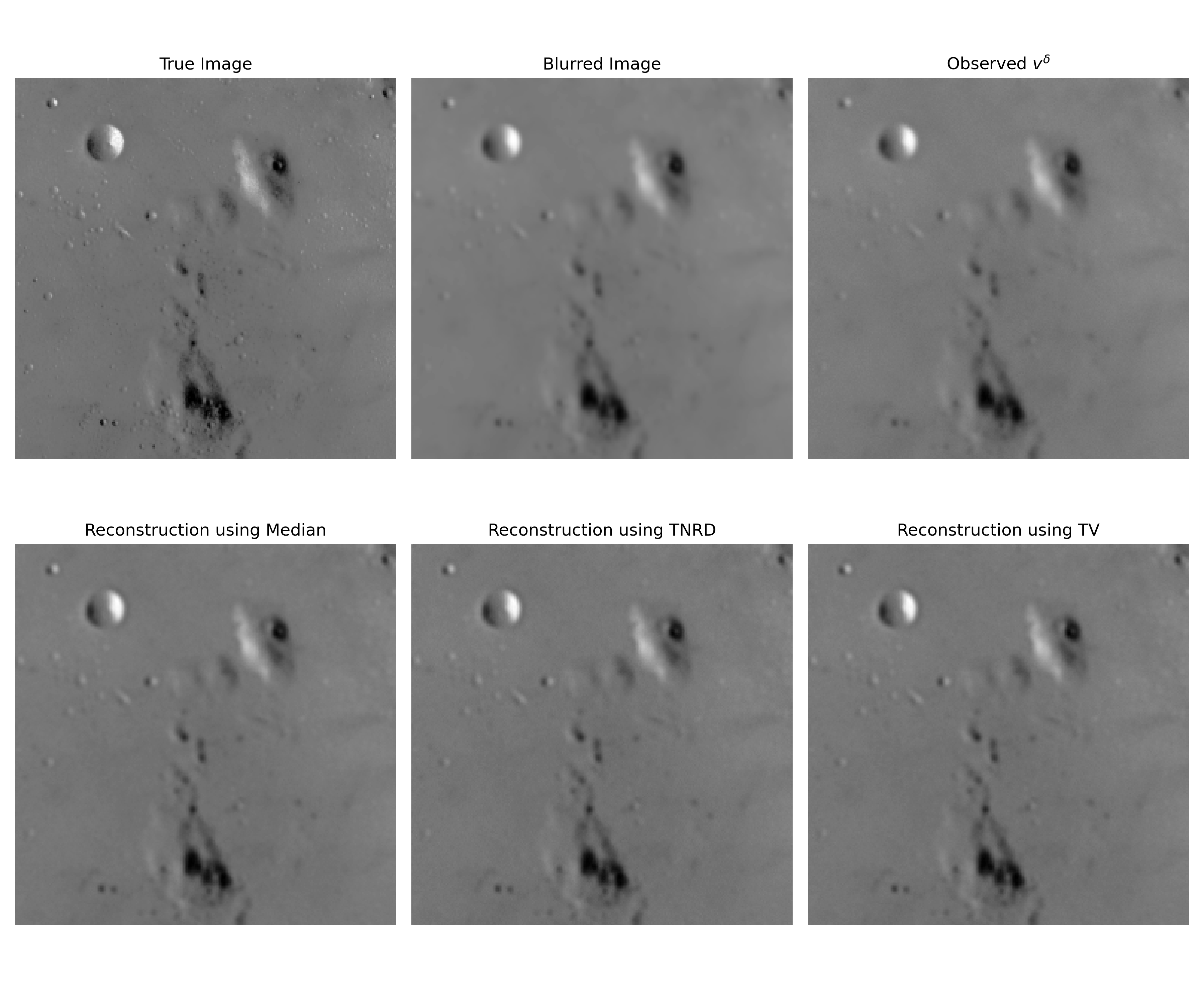}
\hfill
\includegraphics[width=0.45\textwidth]{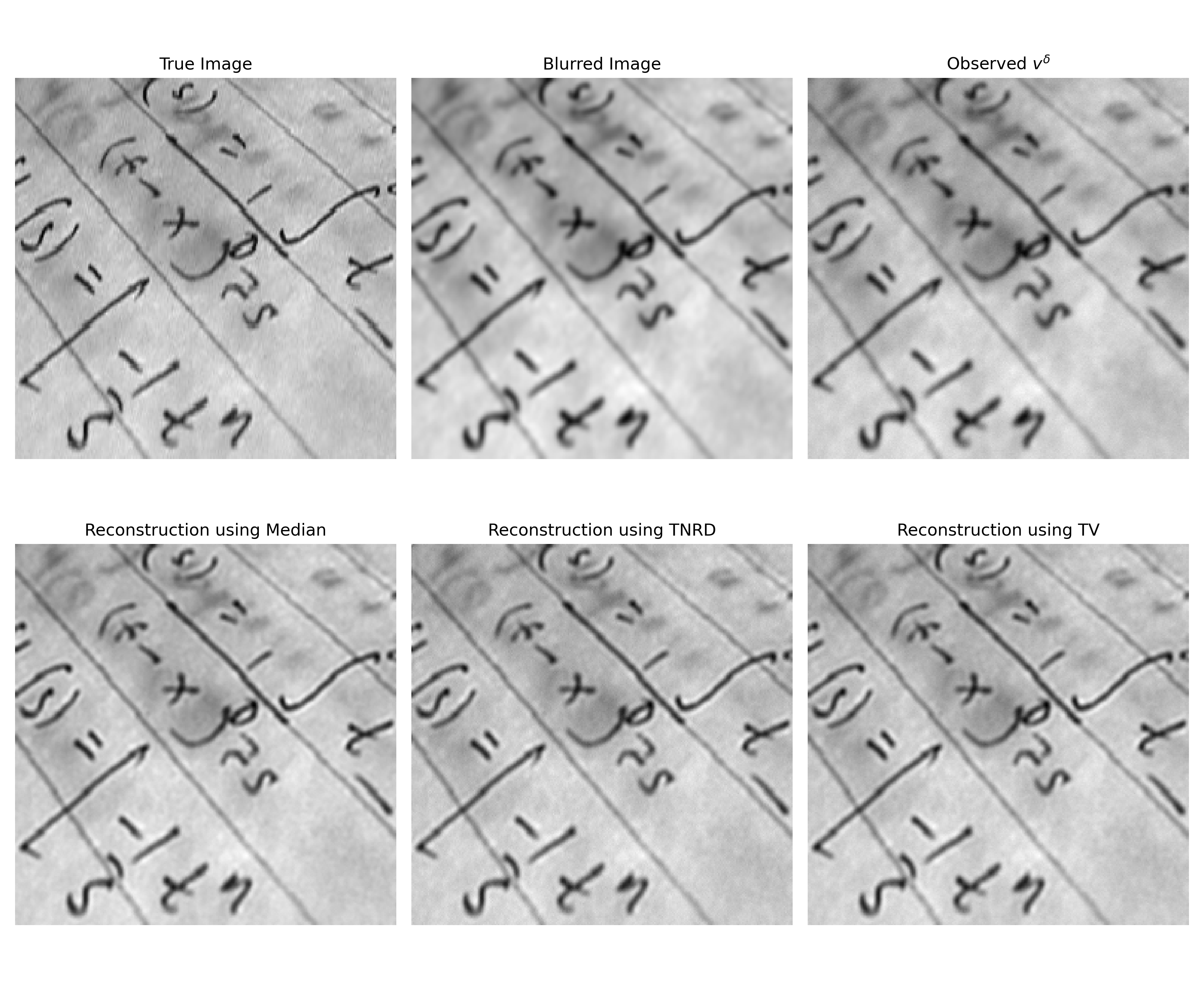}

\caption{Reconstruction results for Shepp-Logan (top left), Coins (top right), Moon (bottom left) and Text (bottom right)  with $\delta_{\mathrm{rel}} = 0.005$.}
\label{fig:reconstruction_results}
\end{figure}
%%%%%%%%%%%%%%%%%%%%%%%%%%%%%%%%%%%%%%%%%%%%%%%%%%%%%%%%%%%%%%%%
\begin{figure}[htbp]
\centering
\begin{subfigure}{0.48\textwidth}
    \includegraphics[width=\linewidth]{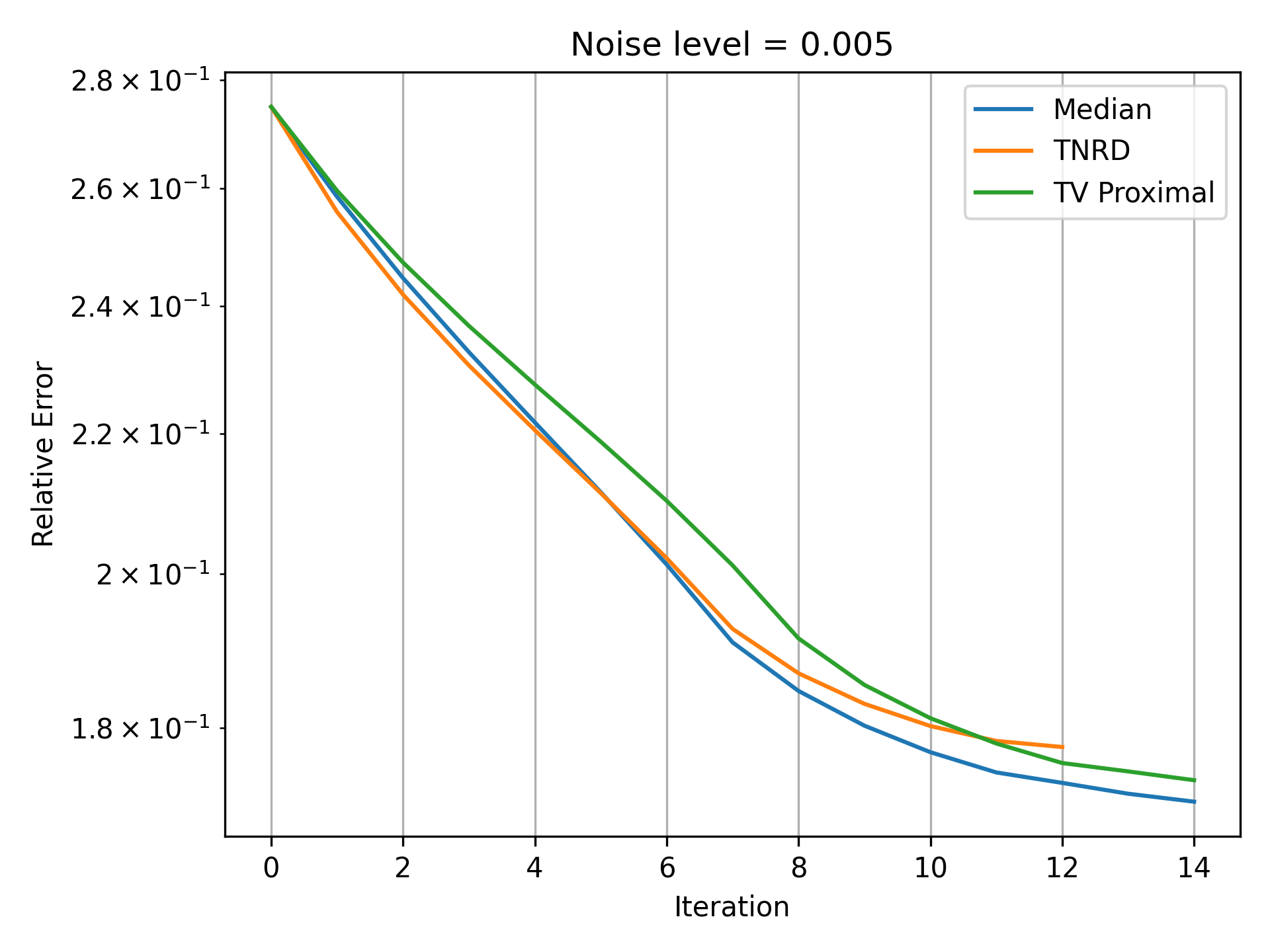}
    \caption{$\delta_{\mathrm{rel}} = 0.005$}
\end{subfigure}
\hfill
\begin{subfigure}{0.48\textwidth}
    \includegraphics[width=\linewidth]{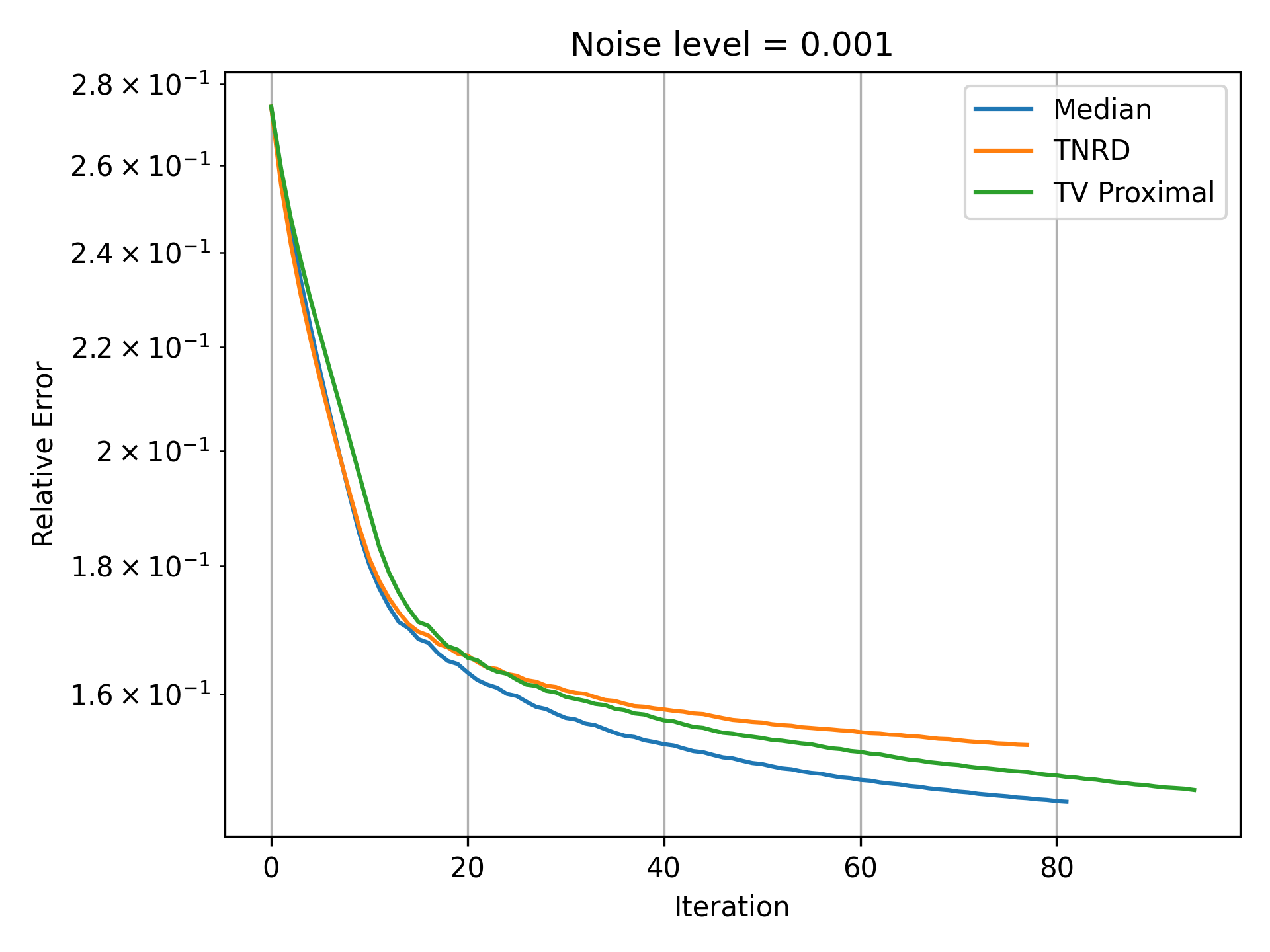}
    \caption{$\delta_{\mathrm{rel}} = 0.001$}
\end{subfigure}

\vspace{0.3cm}

\begin{subfigure}{0.48\textwidth}
    \includegraphics[width=\linewidth]{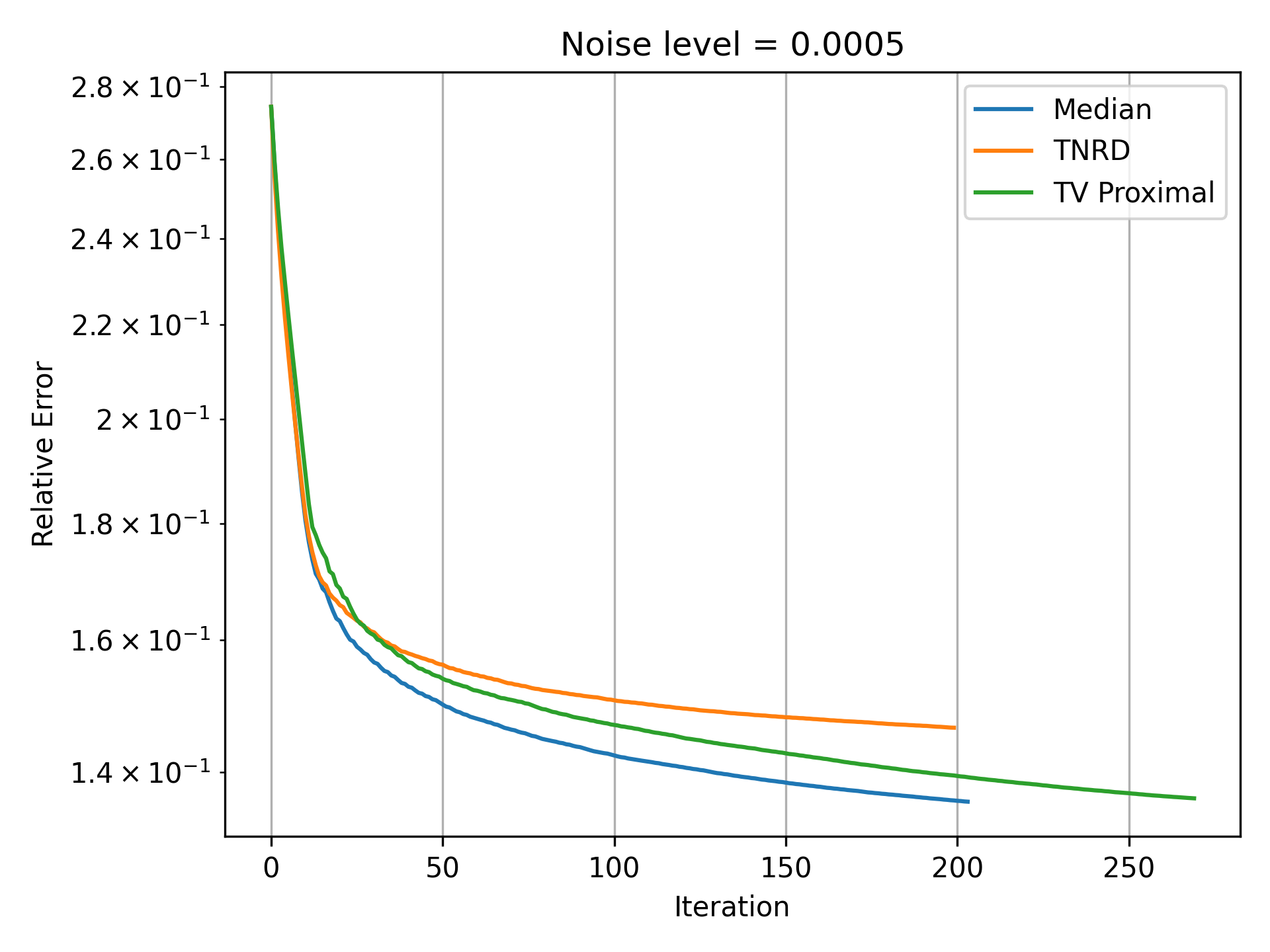}
    \caption{$\delta_{\mathrm{rel}} = 0.0005$}
\end{subfigure}
\hfill
\begin{subfigure}{0.48\textwidth}
    \includegraphics[width=\linewidth]{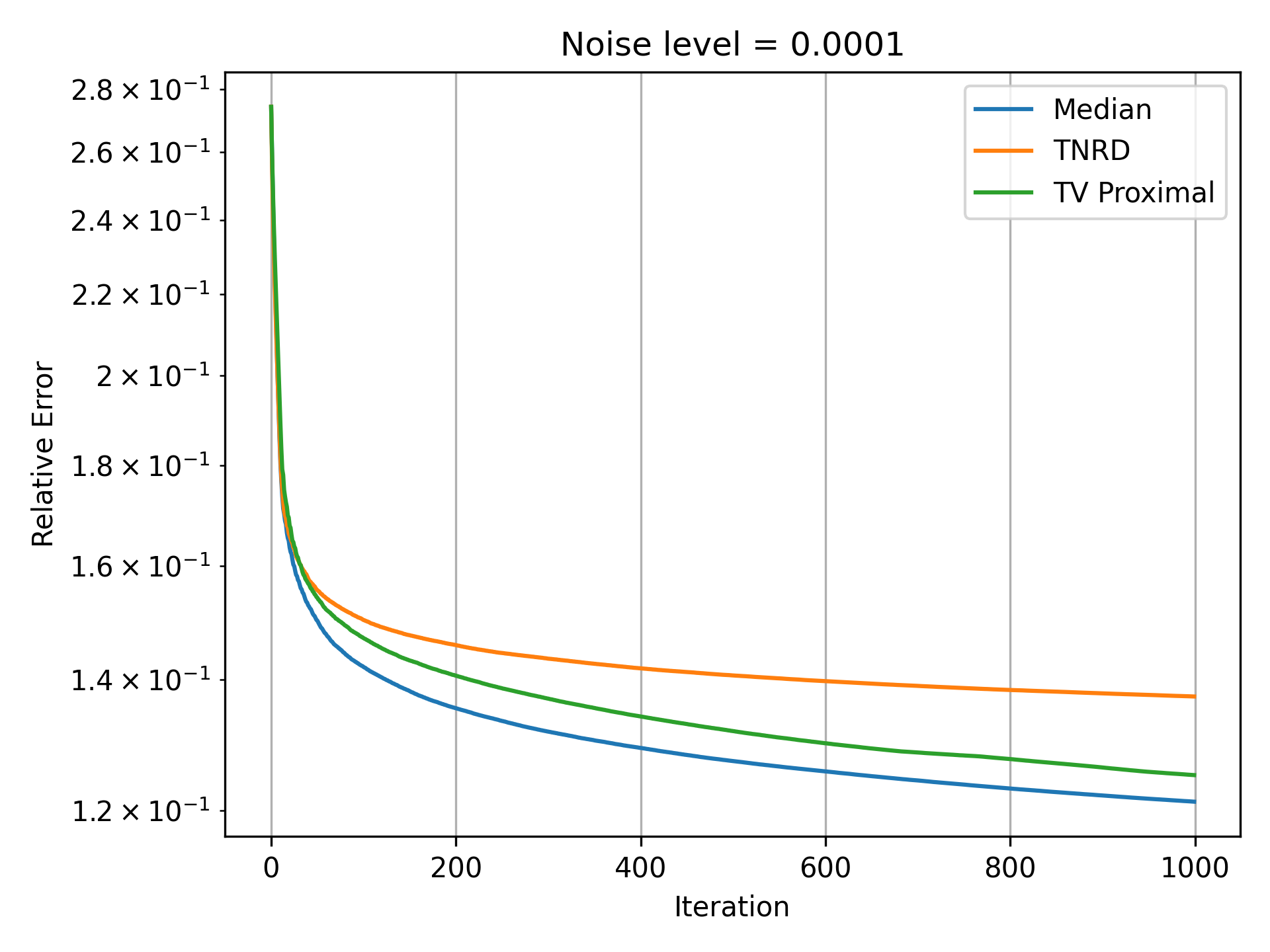}
    \caption{$\delta_{\mathrm{rel}} = 0.0001$}
\end{subfigure}
\caption{Relative error versus iteration for different denoisers (Median, TNRD, TV proximal) under varying noise levels. The TV proximal and median denoisers exhibit stable and contractive behavior, while the TNRD denoiser shows slower or non-monotone convergence.}
\label{fig:rel_error_plots}
\end{figure}
%%%%%%%%%%%%%%%%%%%%%%%%%%%%%%%%%%%%%%%%%%%%%%%%%%%%%%%%%%%%%%
\begin{table}[htbp]
\centering
\caption{\texttt{DDIR} results for the Shepp--Logan image.}
\footnotesize
\setlength{\tabcolsep}{4pt}
\begin{tabular}{c|cccc|cccc|cccc}
\hline
& \multicolumn{4}{c|}{Median} 
& \multicolumn{4}{c|}{TNRD} 
& \multicolumn{4}{c}{TV Proximal} \\
\hline
% \cmidrule(lr){2-5} \cmidrule(lr){6-9} \cmidrule(lr){10-13}
$\delta_{\mathrm{rel}}$ 
& $k_{dp}$ & RE & PSNR & SSIM
& $k_{dp}$ & RE & PSNR & SSIM
& $k_{dp}$ & RE & PSNR & SSIM \\
\hline
0.005  & 14 & 0.17 & 27.63 & 0.9053 & 12 & 0.18 & 27.31 & 0.8633 & 15 & 0.17 & 27.57 & 0.8960 \\
0.001  & 80 & 0.14 & 29.08 & 0.9534 & 76 & 0.15 & 28.61 & 0.9431 & 100 & 0.15 & 29.02 & 0.9518 \\
0.0005 & 203 & 0.14 & 29.64 & 0.9629 & 199 & 0.15 & 29.01 & 0.9533 & 267 & 0.14 & 29.62 & 0.9614 \\
0.0001 & 999 & 0.12 & 30.63 & 0.9737 & 999 & 0.14 & 29.57 & 0.9636 & 999 & 0.13 & 30.35 & 0.9708 \\
\hline
\end{tabular}
\label{tab:shepp_full}
\end{table}
\begin{table}[htbp]
\centering
\caption{\texttt{DDIR} results for the Coins image.}
\footnotesize
\setlength{\tabcolsep}{4pt}
\begin{tabular}{c|cccc|cccc|cccc}
\hline
& \multicolumn{4}{c|}{Median} 
& \multicolumn{4}{c|}{TNRD} 
& \multicolumn{4}{c}{TV Proximal} \\
\hline
$\delta_{\mathrm{rel}}$
& $k_{dp}$ & RE & PSNR & SSIM
& $k_{dp}$ & RE & PSNR & SSIM
& $k_{dp}$ & RE & PSNR & SSIM \\
\hline
0.005  & 10 & 0.09 & 27.99 & 0.8509 & 10 & 0.09 & 27.91 & 0.8342 & 10 & 0.09 & 27.95 & 0.8443 \\
0.001  & 93 & 0.07 & 29.84 & 0.8995 & 95 & 0.08 & 29.64 & 0.8933 & 103 & 0.07 & 29.77 & 0.8982 \\
0.0005 & 285 & 0.07 & 30.64 & 0.9150 & 253 & 0.07 & 30.27 & 0.9073 & 344 & 0.07 & 30.61 & 0.9147 \\
0.0001 & 999 & 0.06 & 31.57 & 0.9314 & 999 & 0.06 & 31.03 & 0.9232 & 999 & 0.06 & 31.37 & 0.9291 \\
\hline
\end{tabular}
\label{tab:coins_full}
\end{table}
\begin{table}[htbp]
\centering
\caption{\texttt{DDIR} results for the Moon image.}
\footnotesize
\setlength{\tabcolsep}{4pt}
\begin{tabular}{c|cccc|cccc|cccc}
\hline
& \multicolumn{4}{c|}{Median} 
& \multicolumn{4}{c|}{TNRD} 
& \multicolumn{4}{c}{TV Proximal} \\
\hline
$\delta_{\mathrm{rel}}$
& $k_{dp}$ & RE & PSNR & SSIM
& $k_{dp}$ & RE & PSNR & SSIM
& $k_{dp}$ & RE & PSNR & SSIM \\
\hline
0.005  & 2 & 0.03 & 38.15 & 0.9283 & 2 & 0.03 & 37.71 & 0.9129 & 2 & 0.03 & 37.90 & 0.9195 \\
0.001  & 15 & 0.02 & 40.23 & 0.9559 & 14 & 0.02 & 40.34 & 0.9557 & 13 & 0.02 & 40.26 & 0.9556 \\
0.0005 & 40 & 0.02 & 40.93 & 0.9618 & 33 & 0.02 & 41.10 & 0.9626 & 34 & 0.02 & 41.07 & 0.9623 \\
0.0001 & 490 & 0.01 & 42.48 & 0.9722 & 352 & 0.02 & 42.68 & 0.9730 & 394 & 0.01 & 42.78 & 0.9731 \\
\hline
\end{tabular}
\label{tab:moon_full}
\end{table}
\begin{table}[htbp]
\centering
\caption{\texttt{DDIR} results for the Text image.}
\footnotesize
\setlength{\tabcolsep}{4pt}
\begin{tabular}{c|cccc|cccc|cccc}
\hline
& \multicolumn{4}{c|}{Median} 
& \multicolumn{4}{c|}{TNRD} 
& \multicolumn{4}{c}{TV Proximal} \\
\hline
$\delta_{\mathrm{rel}}$
& $k_{dp}$ & RE & PSNR & SSIM
& $k_{dp}$ & RE & PSNR & SSIM
& $k_{dp}$ & RE & PSNR & SSIM \\
\hline
0.005  & 6 & 0.03 & 35.11 & 0.9337 & 5 & 0.04 & 34.54 & 0.9148 & 5 & 0.04 & 34.70 & 0.9241 \\
0.001  & 33 & 0.02 & 37.94 & 0.9607 & 31 & 0.02 & 38.11 & 0.9604 & 28 & 0.02 & 38.14 & 0.9608 \\
0.0005 & 79 & 0.02 & 38.82 & 0.9663 & 63 & 0.02 & 39.00 & 0.9665 & 67 & 0.02 & 39.13 & 0.9669 \\
0.0001 & 709 & 0.02 & 40.23 & 0.9741 & 547 & 0.02 & 40.48 & 0.9751 & 588 & 0.02 & 40.82 & 0.9753 \\
\hline
\end{tabular}
\label{tab:text_full}
\end{table}
\subsubsection{Comparison and validation}
We evaluate the performance of the proposed \texttt{DDIR} method against the PnP framework~\cite{ebner2024plug} and a classical Wiener deconvolution baseline, the latter of which is denoted as FBP for consistency with standard inverse methods. The experiments are conducted on the standard \texttt{shepp\_logan} image. Both PnP and \texttt{DDIR} leverage the same TV proximal denoiser. To facilitate a fair comparison, both methods are initialized with the same initial guess, $u_0^\delta = v^\delta$. For PnP reconstruction we iterate using 
\[u_{\alpha, k+1}^\delta = \Phi(\alpha, \cdot) \circ (u_{\alpha, k}^\delta - sG^*(Gu_{\alpha, k}^\delta - v^\delta)),\]
where $\Phi(\alpha, \cdot) = \operatorname{prox}_{\alpha s \mathrm{TV}}$ with $\alpha = 0.01$ and $s =1$ (smaller than $2/ \|G\|^2$) as suggested in \cite{ebner2024plug} for the convergence of fixed-point iteration.

The visual comparison of reconstruction quality with FBP and PnP are shown in Fig.~\ref{fig:Rec_PnP_VS_DDIR} along with the qualitative metrices as in Table~\ref{tab:comparision}.
\begin{figure}[htbp]
    \centering
    % First row
    \begin{subfigure}[b]{0.8\textwidth}
        \centering
        \includegraphics[width=\textwidth]{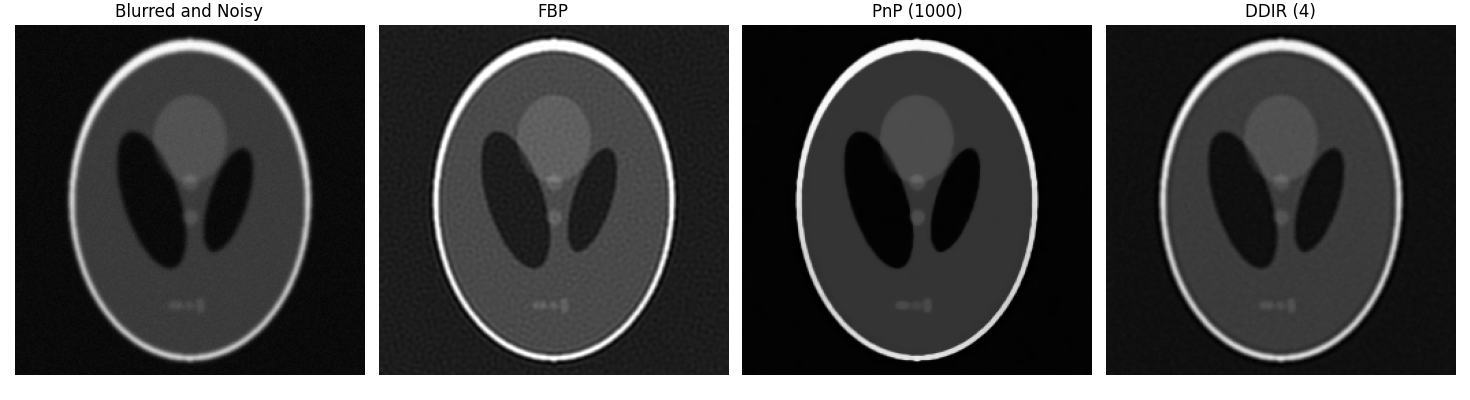}
    \end{subfigure}
    \hfill
    \begin{subfigure}[b]{0.8\textwidth}
        \centering
        \includegraphics[width=\textwidth]{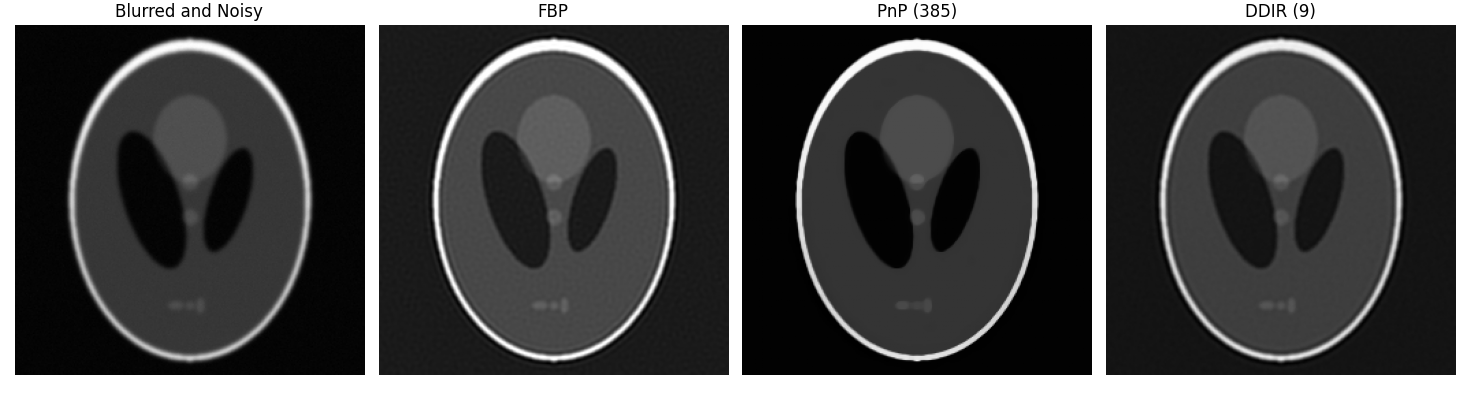}
    \end{subfigure}
    \vspace{0.1cm}
    
    % Second row
    \begin{subfigure}[b]{0.8\textwidth}
        \centering
        \includegraphics[width=\textwidth]{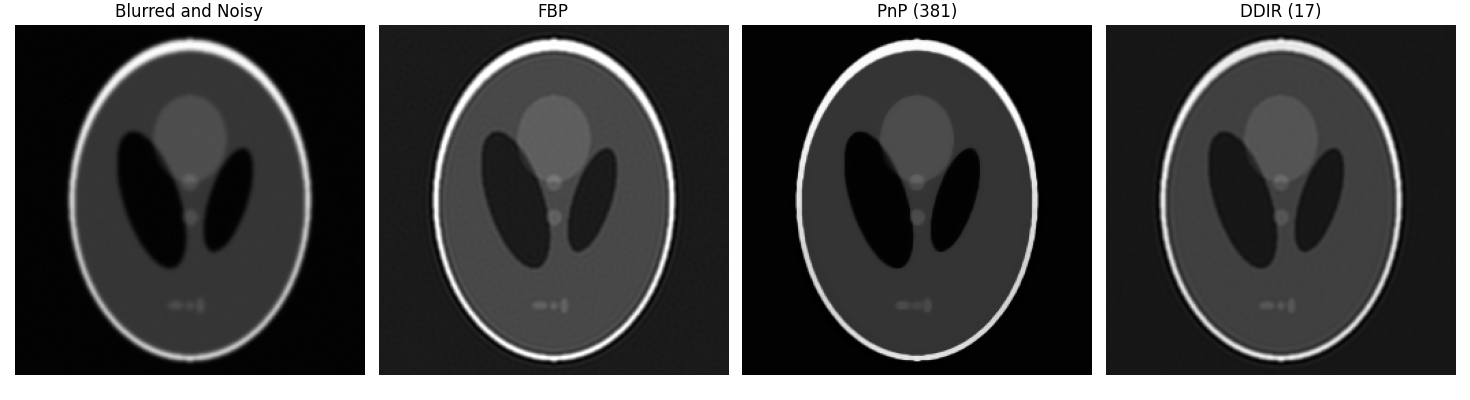}
    \end{subfigure}
    \hfill
    \begin{subfigure}[b]{0.8\textwidth}
        \centering
        \includegraphics[width=\textwidth]{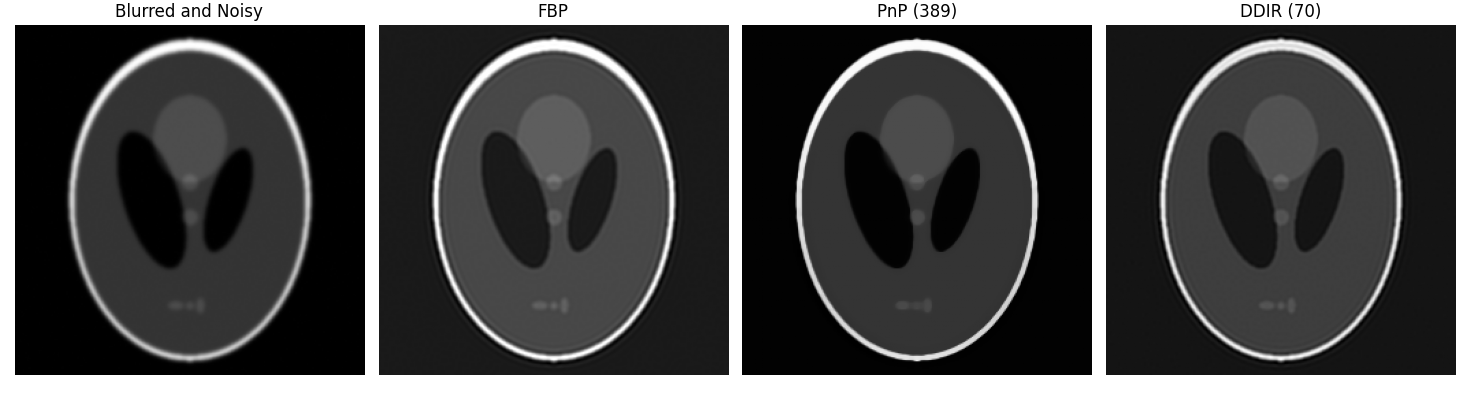}
    \end{subfigure}
   
     \vspace{0.1cm}
    
    % Second row
    \begin{subfigure}[b]{0.8\textwidth}
        \centering
        \includegraphics[width=\textwidth]{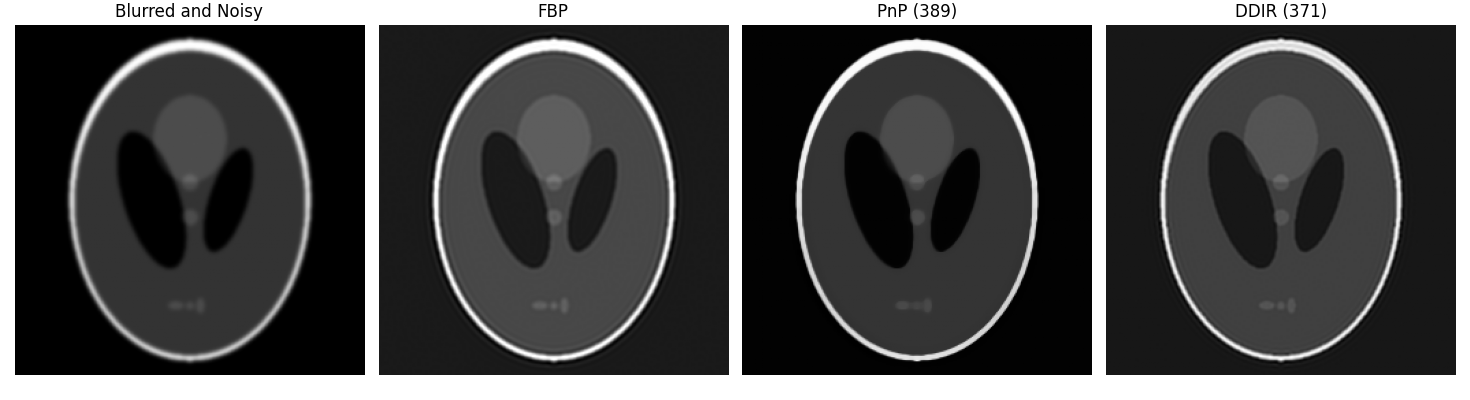}
    \end{subfigure}
    \hfill
    \begin{subfigure}[b]{0.8\textwidth}
        \centering
        \includegraphics[width=\textwidth]{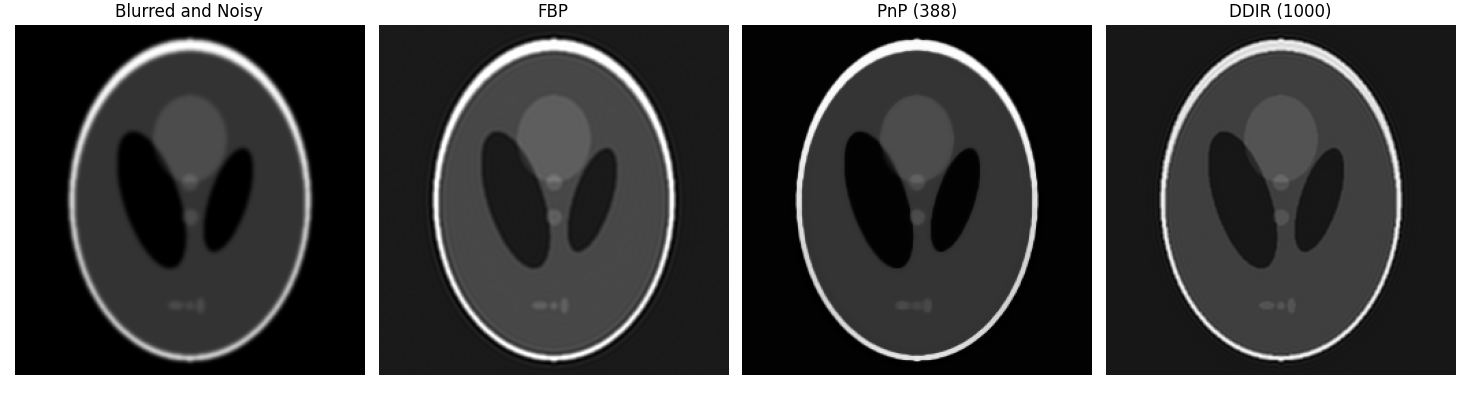}
    \end{subfigure}
    \caption{Visual comparison of reconstructed images  with $\delta_{\mathrm{rel}} = 0.01, 0.005, 0.003, 0.001, 0.0003$ and $0.0001$ respectively.}
    \label{fig:Rec_PnP_VS_DDIR}
\end{figure}
The convergence of PnP and \texttt{DDIR} are shown in Fig.~\ref{fig:St_PnP_VS_DDIR} where the relative error is plotted against the iteration number.
\begin{figure}[htbp]
    \centering
    % First row
    \begin{subfigure}[b]{0.48\textwidth}
        \centering
        \includegraphics[width=\textwidth]{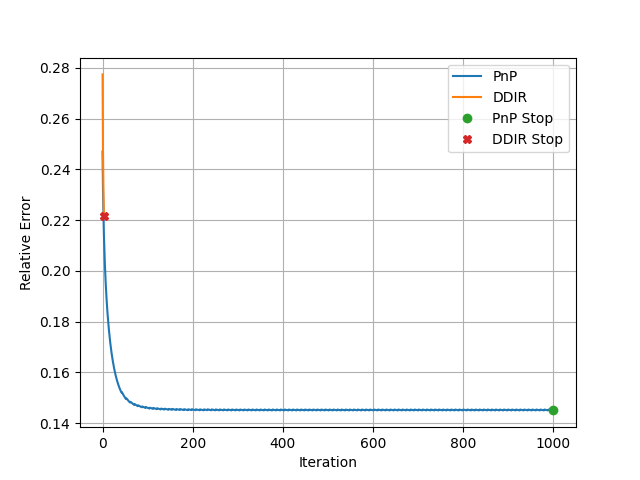}
    \end{subfigure}
    \hfill
    \begin{subfigure}[b]{0.48\textwidth}
        \centering
        \includegraphics[width=\textwidth]{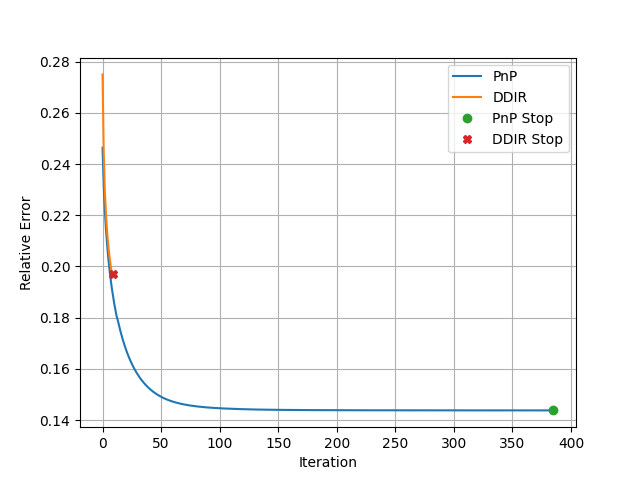}
    \end{subfigure}
    \vspace{0.1cm}
    
    % Second row
    \begin{subfigure}[b]{0.48\textwidth}
        \centering
        \includegraphics[width=\textwidth]{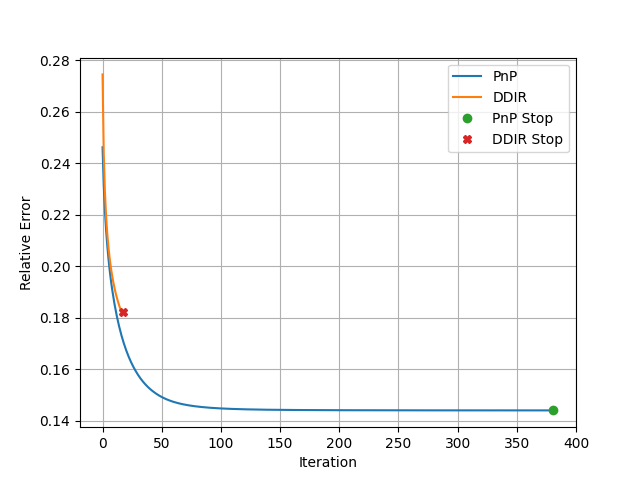}
    \end{subfigure}
    \hfill
    \begin{subfigure}[b]{0.48\textwidth}
        \centering
        \includegraphics[width=\textwidth]{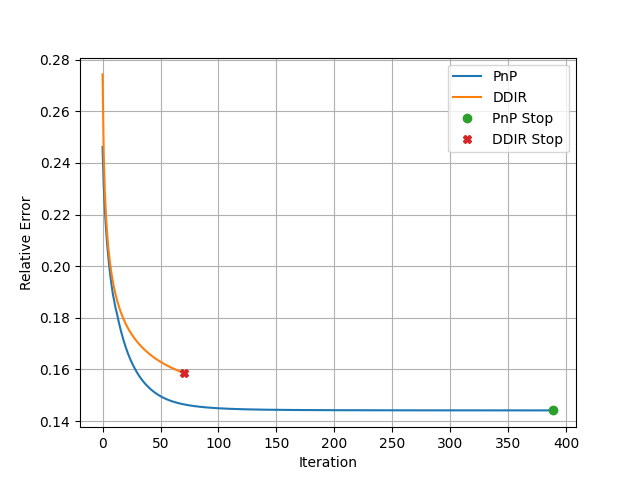}
    \end{subfigure}
   
     \vspace{0.1cm}
    
    % Second row
    \begin{subfigure}[b]{0.48\textwidth}
        \centering
        \includegraphics[width=\textwidth]{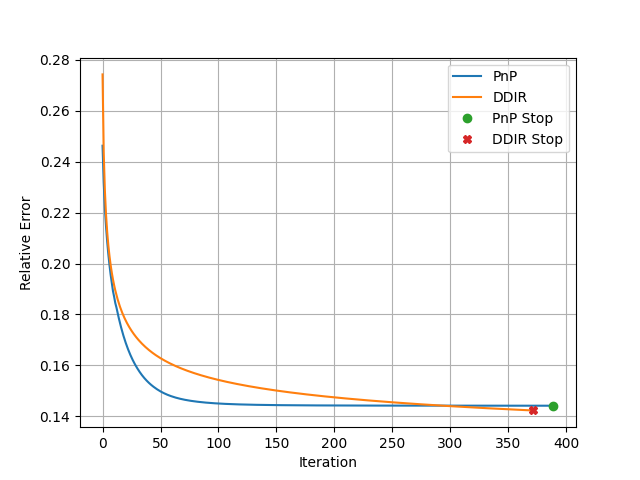}
    \end{subfigure}
    \hfill
    \begin{subfigure}[b]{0.48\textwidth}
        \centering
        \includegraphics[width=\textwidth]{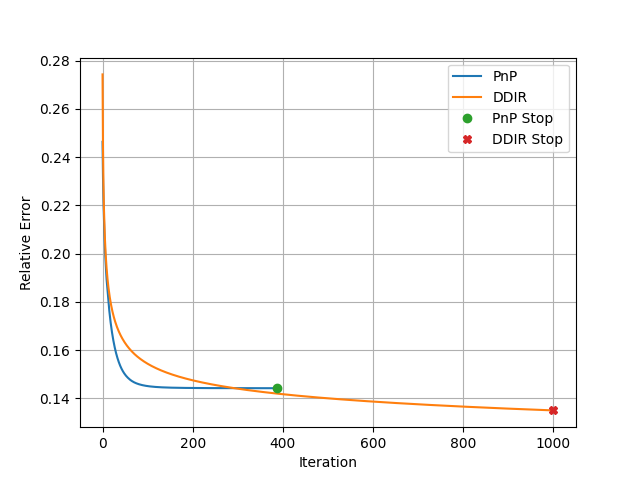}
    \end{subfigure}
    \caption{Stopping index comparision with $\delta_{\mathrm{rel}} = 0.01, 0.005, 0.003, 0.001, 0.0003$ and $0.0001$ respectively.}
    \label{fig:St_PnP_VS_DDIR}
\end{figure}
PnP exhibits a gradual decrease in error, but its stopping index is governed solely by the relative change between successive iterates $\|u_{k+1}^\delta - u_k^\delta \| / \|u_k^\delta\| \leq 10^{-6}$ or after reaching a maximum number of iterations, $K=1000.$ This criterion does not explicitly account for the noise level in the data and may lead to either premature stopping or unnecessary iterations, see Fig.~\ref{fig:St_PnP_VS_DDIR}. In contrast, \texttt{DDIR} utilizes the discrepancy principle~\eqref{eqn:dp} to optimize termination. By stopping as soon as the residual matches the known noise level \texttt{DDIR} prevents over-fitting and achieves superior accuracy with fewer iterations. In summary, while the PnP framework provides a flexible and powerful approach for incorporating priors via denoisers, its performance is highly dependent on parameter tuning and stopping criteria. The proposed \texttt{DDIR} method addresses these limitations by integrating the discrepancy principle, leading to
improved reconstruction quality,
more reliable stopping behavior and
reduced sensitivity to parameter choices.

\begin{table}[htbp]
\centering
\caption{Comparison results for image deblurring on Shepp--Logan image}.
\footnotesize
\setlength{\tabcolsep}{4pt}
\begin{tabular}{c|cccc|cccc|cccc}
\hline
& \multicolumn{4}{c|}{FBP} 
& \multicolumn{4}{c|}{PnP} 
& \multicolumn{4}{c}{\texttt{DDIR}} \\
\hline
% \cmidrule(lr){2-5} \cmidrule(lr){6-9} \cmidrule(lr){10-13}
$\delta_{\mathrm{rel}}$ 
& $k_{dp}$ & RE & PSNR & SSIM
& $k_{dp}$ & RE & PSNR & SSIM
& $k_{dp}$ & RE & PSNR & SSIM \\
\hline
0.01  & - & 0.19 & 26.99 & 0.834 & 1000 & 0.15 & 29.06 & 0.974 & 4 & 0.22 & 25.39 & 0.883 \\
0.005  & - & 0.18 & 27.22 & 0.913 & 385 & 0.14 & 29.15 & 0.979 & 9 & 0.19 & 26.41 & 0.927 \\
0.003  & - & 0.17 & 27.27 & 0.932 & 381 & 0.14 & 29.14 & 0.980 & 17 & 0.18 & 27.09 & 0.938 \\
0.001  & - & 0.18 & 27.30 & 0.942 & 389 & 0.14 & 29.13 & 0.979 & 70 & 0.16 & 28.29 & 0.953 \\
0.0005 & - & 0.14 & 29.64 & 0.963 & 387 & 0.15 & 29.01 & 0.953 & 176 & 0.14 & 29.62 & 0.961 \\
0.0003 & - & 0.18 & 27.30 & 0.943 & 389 & 0.14 & 29.13 & 0.979 & 371 & 0.14 & 29.24 & 0.965 \\
0.0001 & - & 0.17 & 27.30 & 0.943 & 388 & 0.14 & 29.13 & 0.979 & 1000 & 0.13 & 29.70 & 0.970 \\
\hline
\end{tabular}
\label{tab:comparision}
\end{table}

\subsection{Phase retrieval CT} In this section we consider the phase retrieval problem, a fundamental challenge  in high-resolution imaging modalities, including X-ray crystallography, coherent diffraction microscopy, and others \cite{fannjiang2020numerics}.   The phase retrieval problem aims to recover a function $u$ from phaseless measurements, i.e., from the intensity-only data $|\mathcal{G}_c u|^2$ of the transmitted wave field. In our experimental setup, $\mathcal{G}_c  \in \mathbb{R}^{m \times n}$ denotes the discrete Radon transform, implemented using the \texttt{radon} function from the \texttt{skimage.transform} library, with 60 projection angles uniformly distributed over the interval $[1^\circ, 180^\circ]$. We define the forward operator as
\begin{equation}\label{eqn:model_PR_compact}
\mathcal{G}_p(u) := |\mathcal{G}_c u|^2 = v,
\end{equation}
where $|\mathcal{G}_c u|^2$ denotes the element-wise squared magnitude of the projection data. 
We tackle this nonlinear inverse problem through the deployment of Algorithm~\ref{alg:DDIR}.
Note that the operator $\mathcal{G}_p$ is continuously Fr\'{e}chet differentiable. Its  derivative $\mathcal{G}_p'(u) \in \mathbb{R}^{m \times n}$ and its corresponding adjoint $\mathcal{G}_p'(u)^* \in \mathbb{R}^{n \times m}$ are given by
\begin{equation*}
    \mathcal{G}_p'(u)[q] = 2(w \odot (\mathcal{G}_cq))  = 2 (\operatorname{diag}(w) \mathcal{G}_c)q,  \qquad
\mathcal{G}_p'(u)^*[r] =2 \mathcal{G}_c^{\top} ((\mathcal{G}_cu) \odot r)
\end{equation*}
where $q \in \mathbb{R}^n, r\in \mathbb{R}^m$, $w= \mathcal{G}_cu$ and $\odot$ denotes the Hadamard product. Thus, the corresponding Jacobian matrix are given by 
\begin{equation}\label{eqn:PR_frechet}
   \mathbf{J}(u)= \mathcal{G}_p'(u) = 2 (\operatorname{diag}(w) \mathcal{G}_c), \qquad  \mathbf{J}(u)^* =
\mathcal{G}_p'(u)^* =2 \mathcal{G}_c^{\top} \operatorname{diag}(w) .
\end{equation}
We employ the formulations \eqref{eqn:model_PR_compact} and \eqref{eqn:PR_frechet}, utilizing only the noisy measurements 
$v^\delta$, to implement Algorithm~\ref{alg:DDIR}. The numerical experiments are carried out on the \texttt{binary\_blobs} image provided by \texttt{skimage.data}. The true image along with the noisy data with $\delta_{\mathrm{rel}} = 0.005$ is shown in Fig.~\ref{fig:PR_true}. 
\begin{figure}[htbp]
    \centering
    \includegraphics[width=0.6\textwidth]{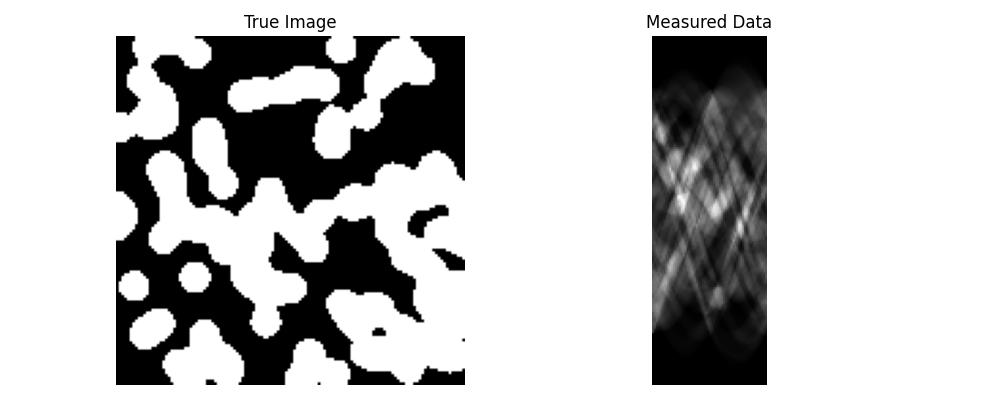}
    \caption{Ground truth image for phase-retrieval CT and the corresponding measurements corrupted by  $\delta_{\mathrm{rel}}=0.005$.}
    \label{fig:PR_true}
\end{figure}
The experimental setup considers an image domain of size $128 \times 128$ (i.e., $n = 16{,}384$ pixels) and in that case the number of detector elements are $d = \lceil \sqrt{2}\cdot 128 \rceil = 182$. With 60 projection angles, this yields a total of $m = 60 \times 182 = 10{,}920$ measurements. Consequently, the resulting system satisfies $m < n$, leading to an under determined  ill-posed  problem. The fixed parameters of Algorithm~\ref{alg:DDIR} are chosen as $\tau = 1.5$, $\gamma_0 = 0.01$, $\gamma_1 = 2$, $\nu_0 = 0.05$, and $\nu_1 = 0.1$, while the maximum number of iterations is set to 1000. The reconstruction results and the corresponding relative error curves for different noise levels~\footnote{ For $\delta_{\mathrm{rel}} = 0.0005$, the corresponding reconstruction results and relative error plots are provided in Appendix~\ref{app:extra_results}.} using median and TV denoisers are presented in Fig.~\ref{fig:PR_rec_results} and Fig.~\ref{fig:PR_RE_curves}, respectively. Additional quantitative metrics, together with the stopping index, are reported in Table~\ref{Tab:PR}~\footnote{The residual curves have also been computed and are presented in Appendix~\ref{app:extra_results}.}.
\begin{table}[ht]
\footnotesize
\centering
\begin{tabular}{c|cccc|cccc}
\hline
$\delta_{\mathrm{rel}}$ & \multicolumn{4}{c|}{Median} & \multicolumn{4}{c}{TV} \\
\hline
&$k_{\mathrm{dp}}$ & RE & PSNR & SSIM & $k_{\mathrm{dp}}$ & RE & PSNR & SSIM \\
\hline
0.01 & 47 & 0.1755 & 18.12 & 0.6633 & 56& 0.1748 & 18.15 & 0.6574 \\
0.005 & 114 & 0.1397 & 20.11 & 0.7218 & 124 & 0.1381 & 20.21 & 0.7259 \\
0.003 & 214 & 0.1198 & 21.45 & 0.7618  & 220 & 0.1205 & 21.39 & 0.7605 \\
0.001 & 737 & 0.0888 & 24.04 & 0.8136 & 783 & 0.0899 & 23.94 & 0.8127 \\
0.0005 & 1000 & 0.0830 & 24.63 & 0.8218  & 1000 & 0.0859 & 24.34 & 0.8185  \\
\hline
\end{tabular}
\caption{Reconstruction performance on phase retrieval CT under different noise levels.}
\label{Tab:PR}
\end{table}
%%%%%%%%%%%%%%%%%%%%%%%%%%%%%%%%%%%%%%%%%%%%%%%%%%%%%%%%%%%%%
\begin{figure}[htbp]
    \centering

    % First row
    \begin{subfigure}[b]{0.49\textwidth}
        \centering
        \includegraphics[width=\textwidth]{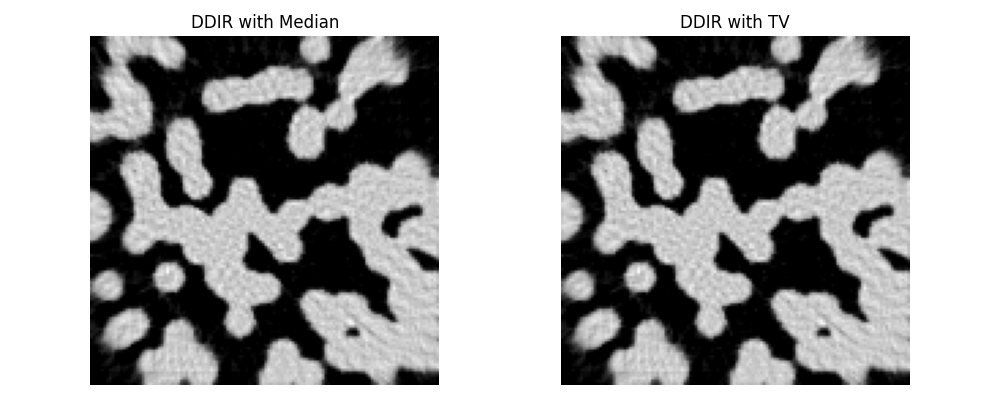}
    \end{subfigure}
    % \hfill
    \begin{subfigure}[b]{0.49\textwidth}
        \centering
        \includegraphics[width=\textwidth]{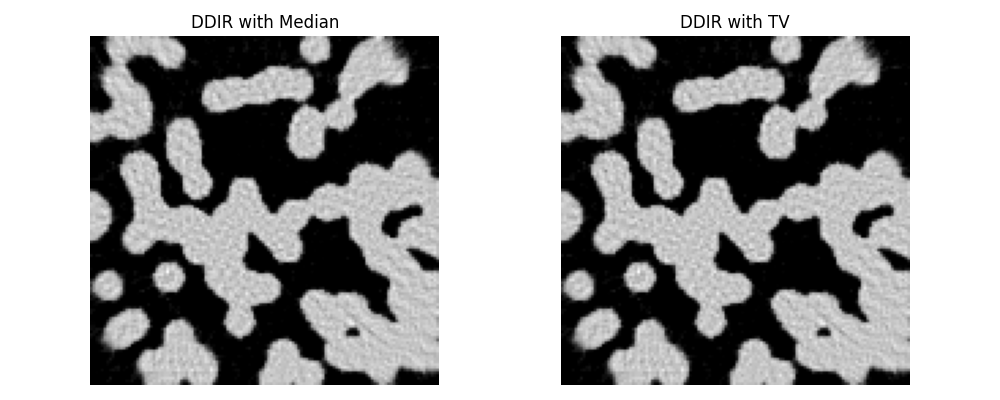}
    \end{subfigure}
    \vspace{0.2cm}
    % Second row
    \begin{subfigure}[b]{0.49\textwidth}
        \centering
        \includegraphics[width=\textwidth]{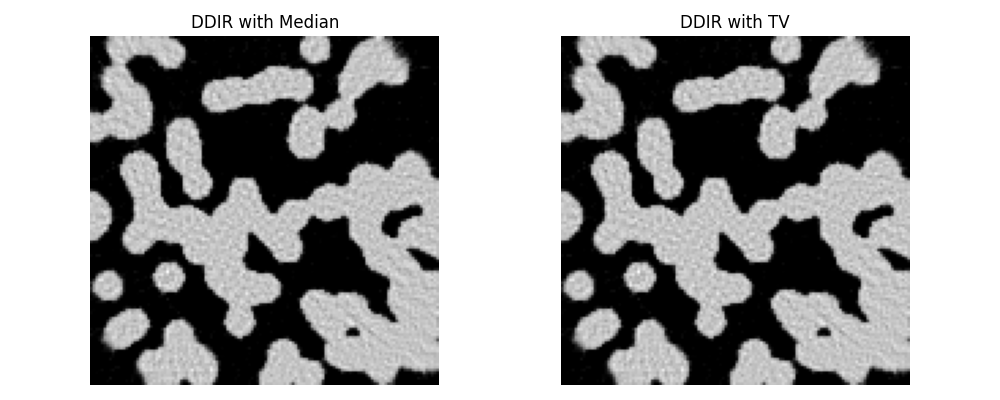}
    \end{subfigure}
    % \hfill
    \begin{subfigure}[b]{0.49\textwidth}
        \centering
        \includegraphics[width=\textwidth]{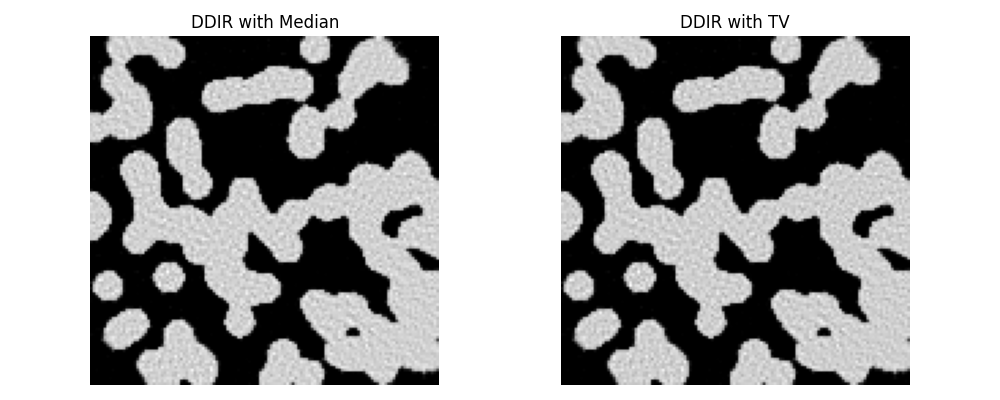}
    \end{subfigure}
    \caption{Reconstruction results for phase retrieval CT with $\delta_{\mathrm{rel}}     = 0.01$ (top left), $0.005$ (top right), $0.003$ (bottom left) and $0.001$ (bottom right).}
    \label{fig:PR_rec_results}
\end{figure}
\begin{figure}[htbp]
    \centering

    % First row
    \begin{subfigure}[b]{0.49\textwidth}
        \centering
        \includegraphics[width=\textwidth]{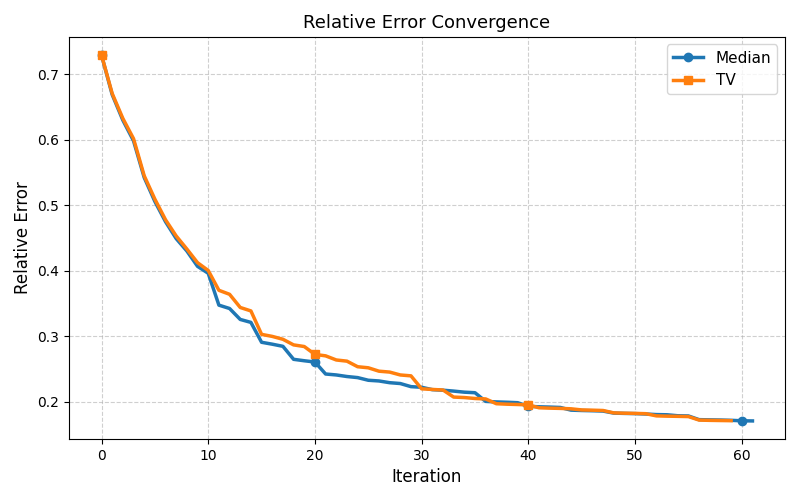}
    \end{subfigure}
    \hfill
    \begin{subfigure}[b]{0.49\textwidth}
        \centering
        \includegraphics[width=\textwidth]{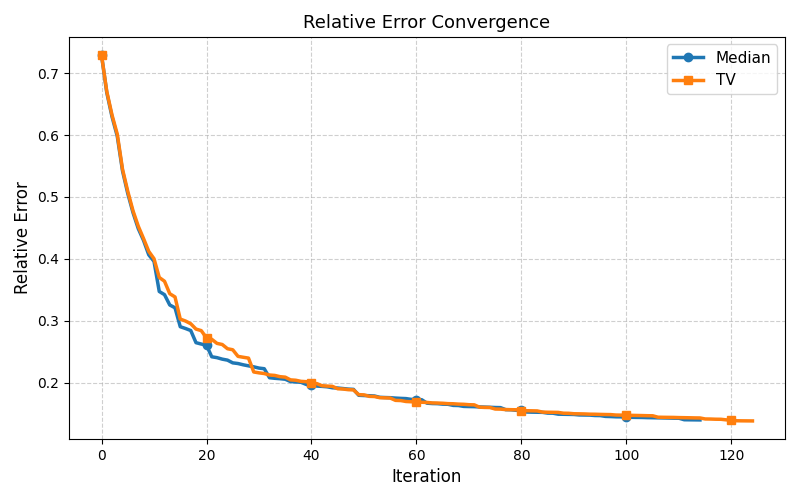}
    \end{subfigure}
    
    \vspace{0.1cm}
    
    % Second row
    \begin{subfigure}[b]{0.49\textwidth}
        \centering
        \includegraphics[width=\textwidth]{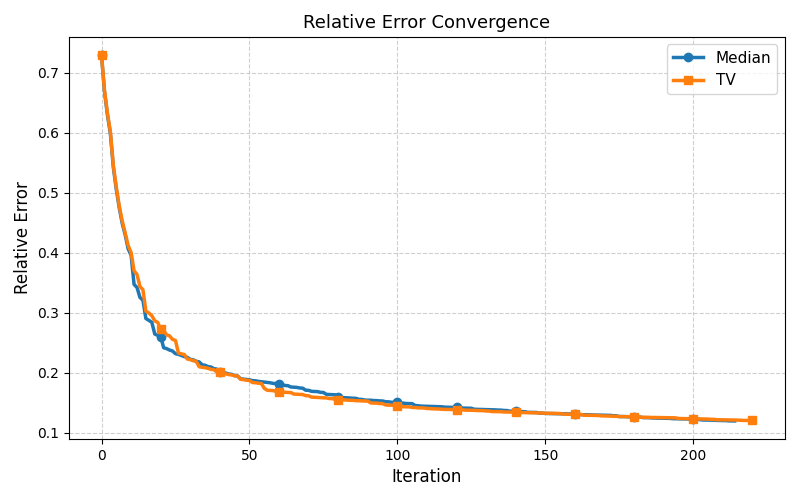}
    \end{subfigure}
    \hfill
    \begin{subfigure}[b]{0.49\textwidth}
        \centering
        \includegraphics[width=\textwidth]{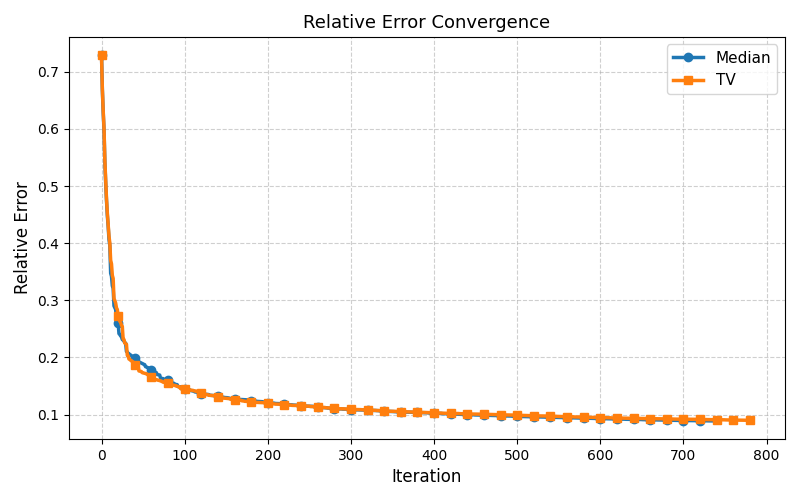}
    \end{subfigure}
    \caption{Relative error plots for phase retrieval CT with $\delta_{\mathrm{rel}}     = 0.01$ (top left), $0.005$ (top right), $0.003$ (bottom left) and $0.001$ (bottom right).}
    \label{fig:PR_RE_curves}
\end{figure}
%%%%%%%%%%%%%%%%%%%%%%%%%%%%%%%%%%%%%%%%%%%%%%%%%%%%%%%%
\section{Conclusion}\label{sec:conclusion} In recent years, image denoisers have emerged as powerful tools for solving general inverse problems by acting as regularizers. A rigorous theoretical understanding of their stability and convergence within the framework of iterative regularization remains largely unexplored. Building upon this line of research, we propose a novel iterative regularization framework that incorporates an averaged denoiser to implicitly enforce prior information. Within this framework, we introduce and rigorously analyze an adaptive denoiser-driven iterative regularization (\texttt{DDIR}) method. The proposed approach is designed to address nonlinear ill-posed image reconstruction problems, and we establish that it constitutes a convergent regularization method when equipped with an \emph{a posteriori} stopping rule.
In particular, under appropriate assumptions, we prove the stability and finite termination of the proposed method. Consequently, it follows that the iterates produced by the \texttt{DDIR} algorithm converge to the true solution in the limit as the noise level vanishes.

Future research will focus on accelerating the proposed method, for instance by incorporating conjugate gradient techniques or, more effectively, by employing the sequential subspace optimization (SESOP) algorithm \cite{elad2007coordinate}. Another promising direction is to enrich the \texttt{DDIR} framework by integrating alternative mechanisms, beyond conventional denoisers, that can effectively encode prior information about the unknown image. Additionally, it would be of interest to develop heuristic or statistical discrepancy stopping criteria, following ideas similar to those presented in~\cite{bajpai2026graph,harrach2020beyond}.

%%%%%%%%%%%%%%%%%%%%%%%%%%%%%%%%%%%%%%%%%%%%%%%%%%%%%

\section*{Data availability}
The source code and underlying datasets utilized in this work are available for the purpose of reproducibility from the authors upon reasonable request.
\section*{Acknowledgment}
The authors would like to thank Dr. Andrea Ebner (University of Mannheim, Germany) and Prof. Markus Haltmeier
(University of Innsbruck, Austria) for sharing their code. Their work provided an essential baseline, and we are grateful for the opportunity to compare our numerical experiments against their approach. Part of this work was done while HB enjoyed the hospitality of the
Department of Mathematics, Indian Institute of Technology Gandhinagar, India.

\bibliographystyle{abbrv}
\bibliography{references}
% --- Beginning of Appendix Section ---
\appendix
% \section*{Appendix} % Optional: main appendix heading
\addcontentsline{toc}{section}{Appendix}

% The following commands ensure Figures/Tables are numbered A.1, A.2, etc.
\counterwithin{figure}{section}
\counterwithin{table}{section}

\section{Additional experimental results} \label{app:extra_results}
In this appendix, we provide supplementary numerical evaluations to further validate the performance of the proposed \texttt{DDIR} method.
\begin{figure}[p] % [p] suggests putting the figure on its own dedicated page
    \centering
    
    % --- First noise level: 0.001 ---
    \begin{minipage}{1.0\textwidth}
        \centering
        \includegraphics[width=0.23\textwidth]{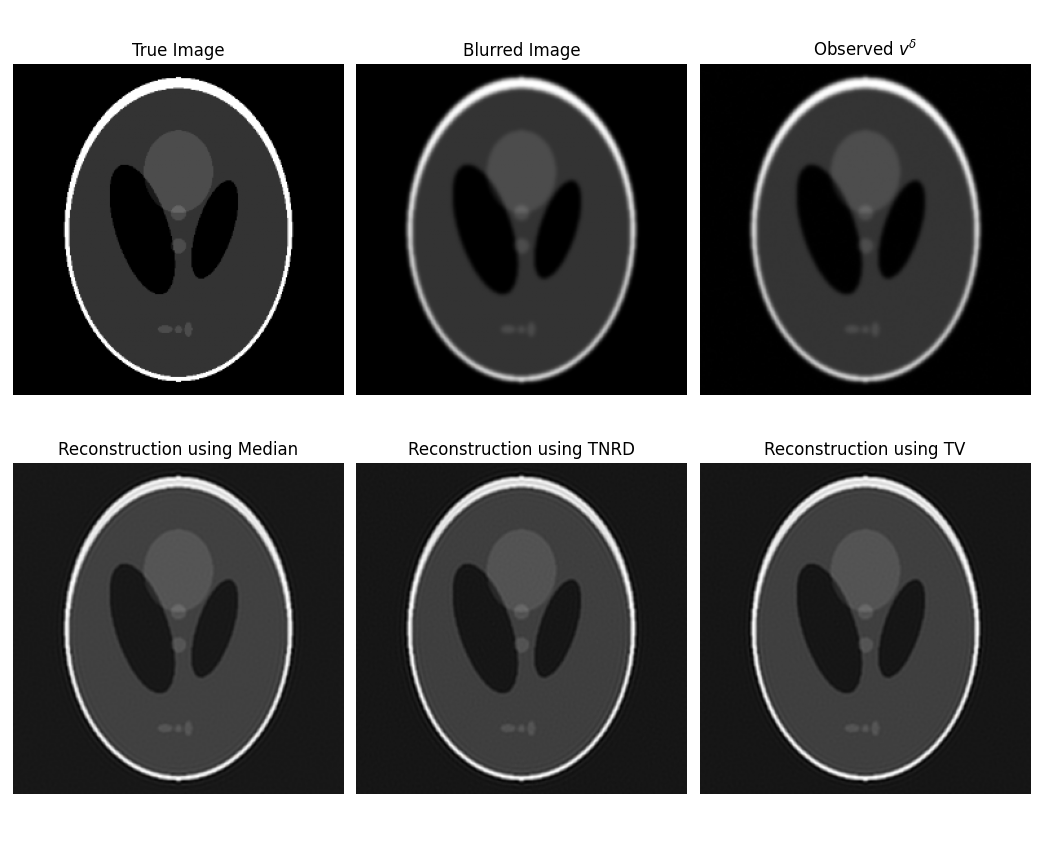} \hfill
        \includegraphics[width=0.23\textwidth]{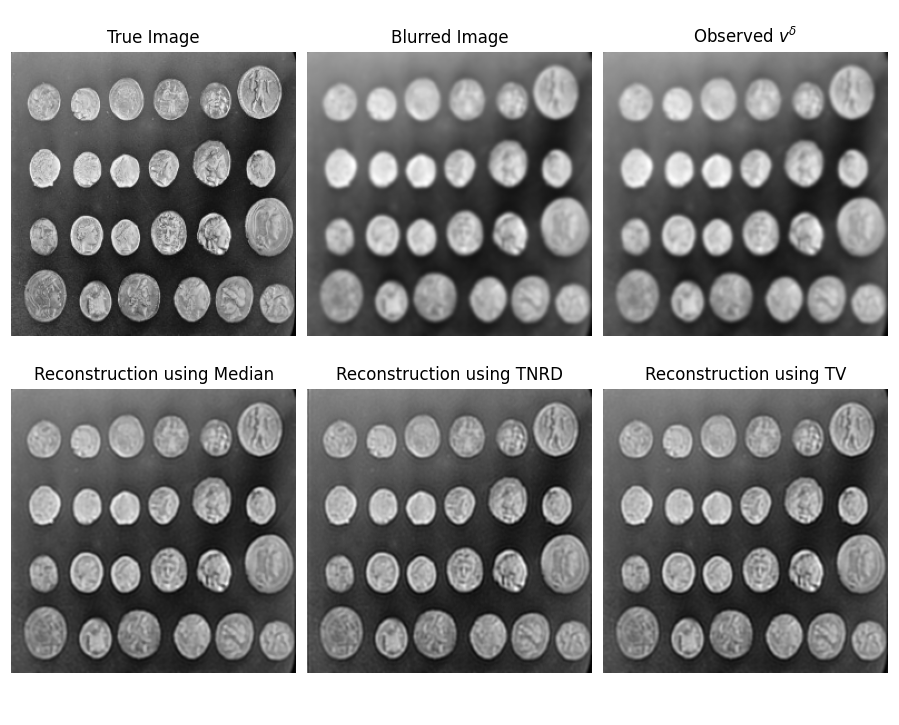} \hfill
        \includegraphics[width=0.23\textwidth]{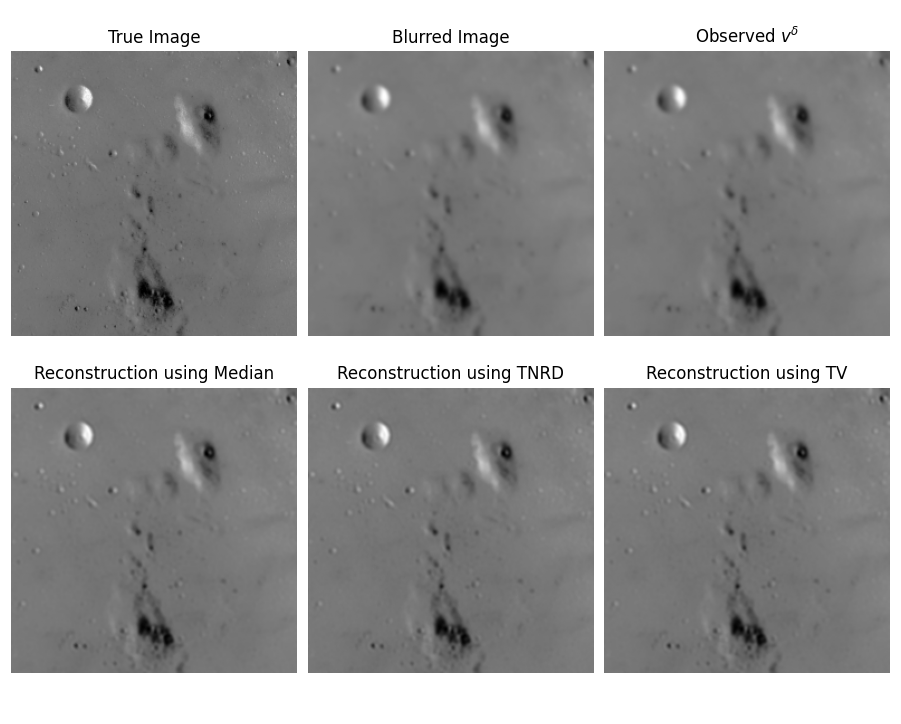} \hfill
        \includegraphics[width=0.23\textwidth]{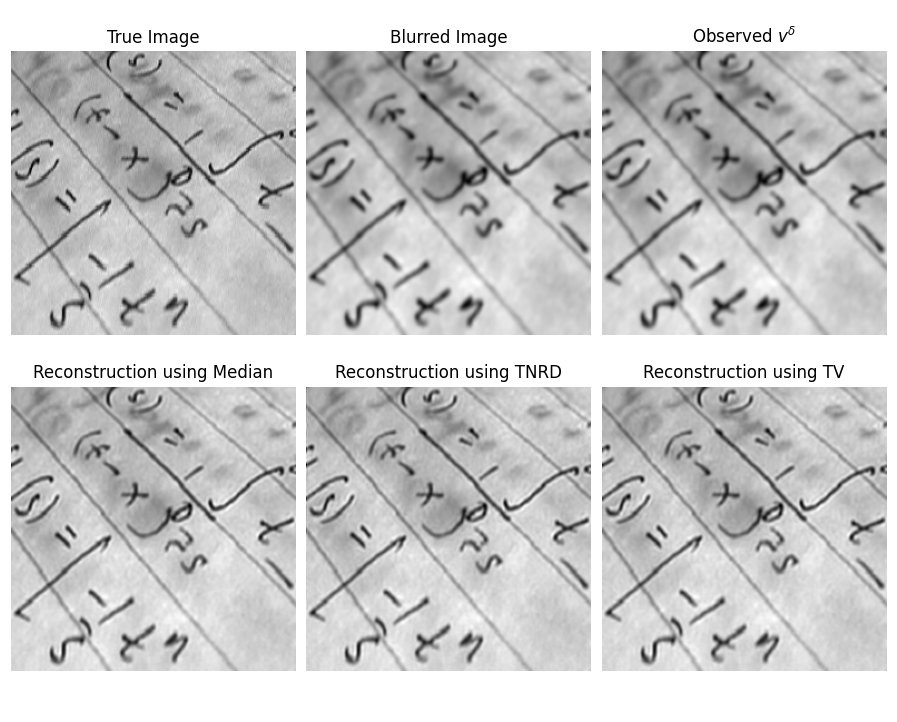}
        \caption*{ (a) $\delta_{\mathrm{rel}} = 0.001$} 
    \end{minipage}

    \vspace{0.5cm} % Vertical gap between the two experimental sets

    % --- Second noise level: 0.0005 ---
    \begin{minipage}{1.0\textwidth}
        \centering
        \includegraphics[width=0.23\textwidth]{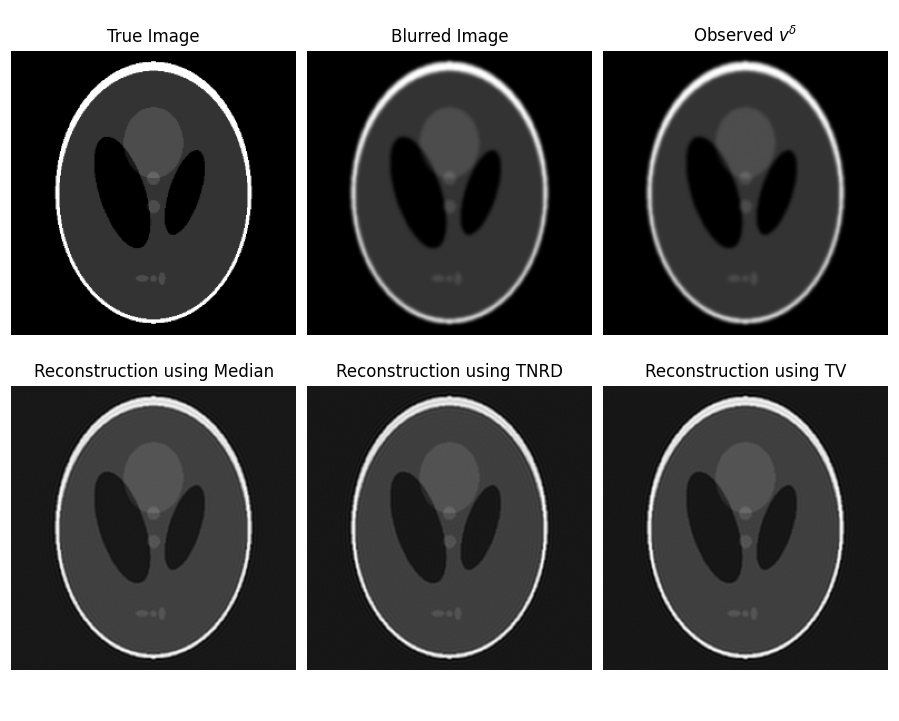} \hfill
        \includegraphics[width=0.23\textwidth]{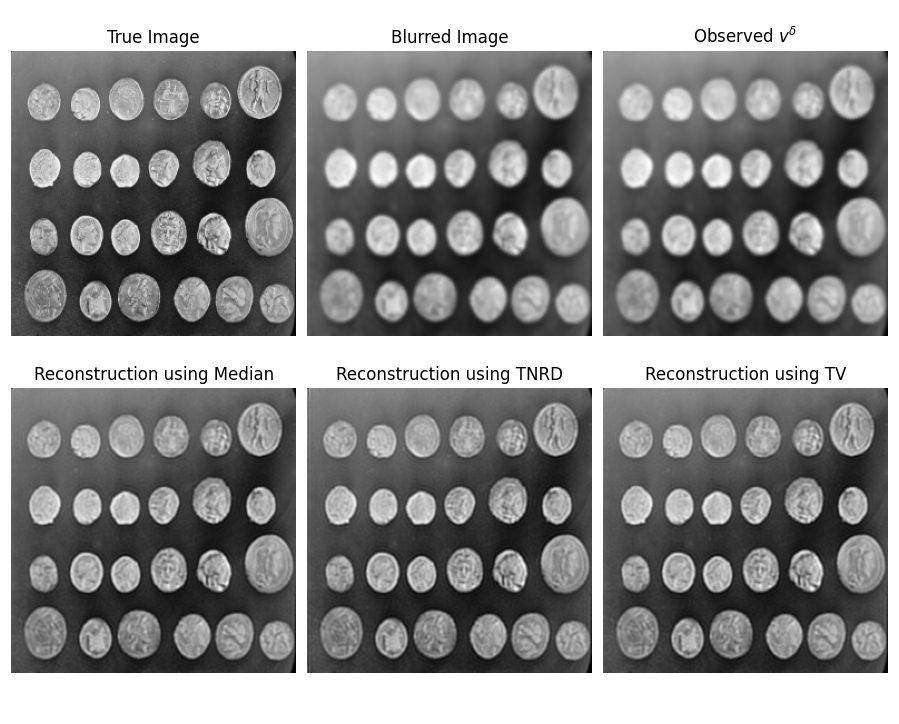} \hfill
        \includegraphics[width=0.23\textwidth]{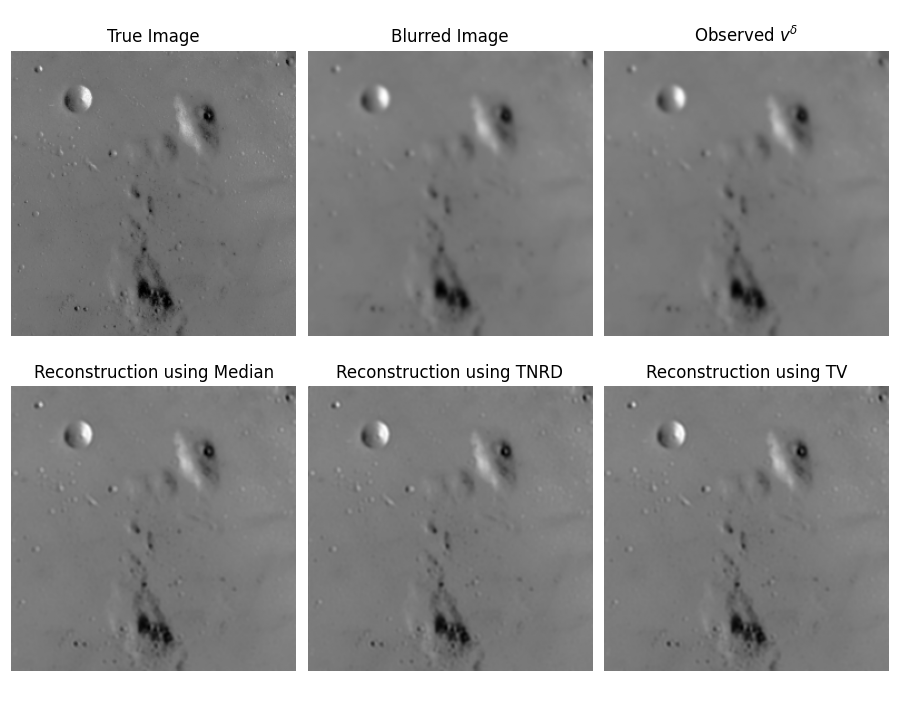} \hfill
        \includegraphics[width=0.23\textwidth]{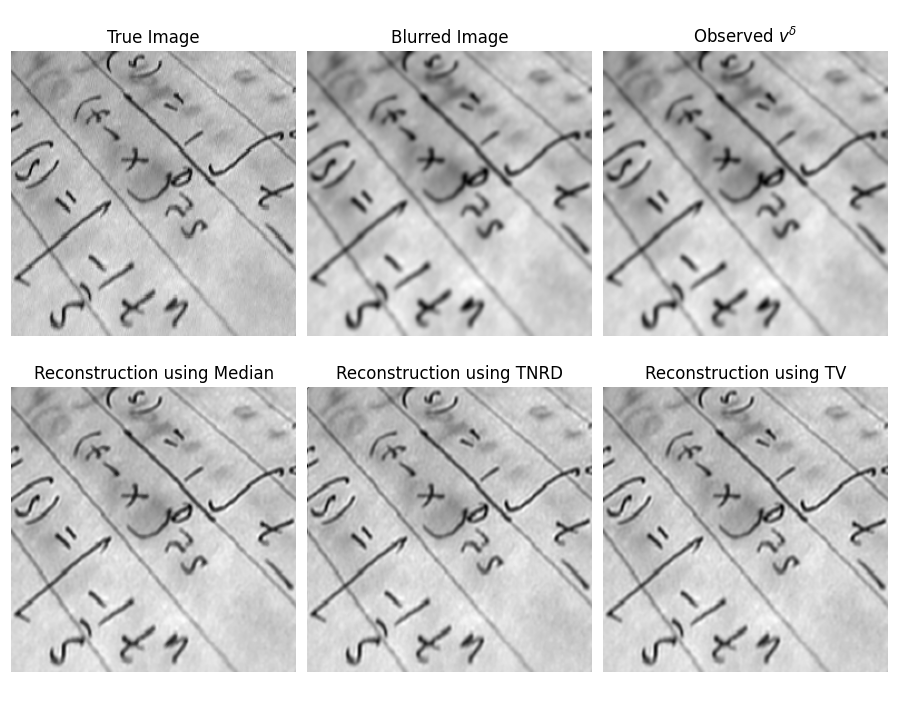}
        \caption*{ (b) $\delta_{\mathrm{rel}} = 0.0005$}
    \end{minipage}

    \vspace{0.3cm}
    \caption{Comparative reconstruction results for image deblurring under two noise levels: $\delta_{\mathrm{rel}} = 0.001$ (top row) and $\delta_{\mathrm{rel}} = 0.0005$ (bottom row).}
    \label{fig:combined_deblurring}
\end{figure}

\begin{figure}[htbp]
\centering
\includegraphics[width=0.49\textwidth]{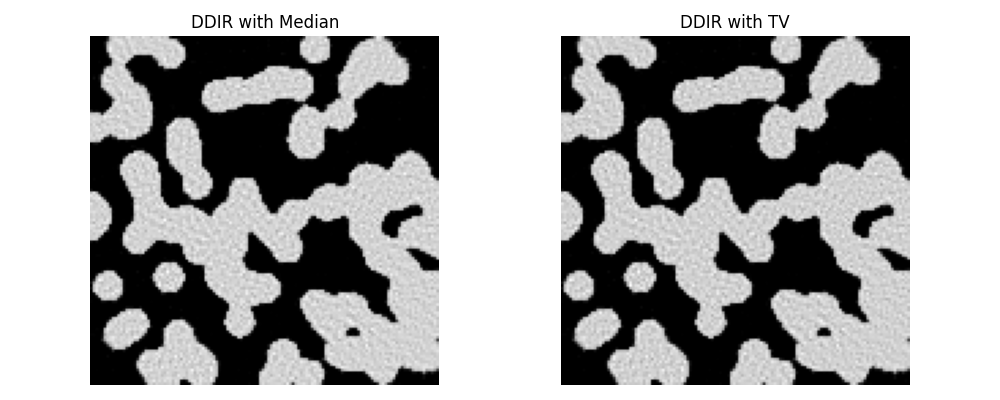}

\vspace{0.4cm}

\includegraphics[width=0.49\textwidth]{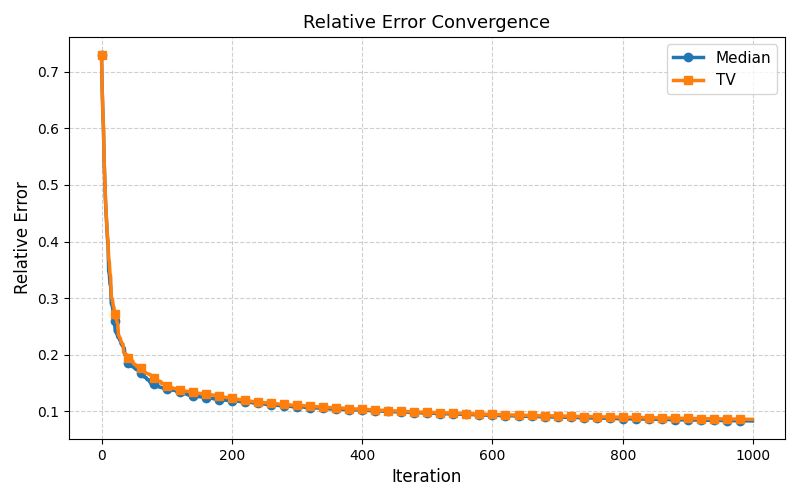}

\caption{Reconstruction results and relative error curves for phase retrieval CT with $\delta_{\mathrm{rel}} = 0.0005$.}
\label{fig:PR_0005}
\end{figure}

% ============================================================
% \subsection{Phase Retrieval CT: Residual Curves}
% ============================================================

\begin{figure}[htbp]
\centering
\includegraphics[width=0.49\textwidth]{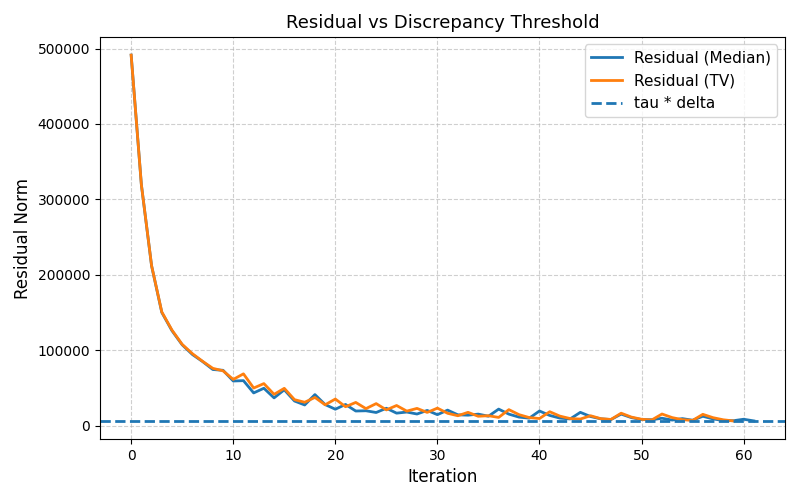}
\hfill
\includegraphics[width=0.49\textwidth]{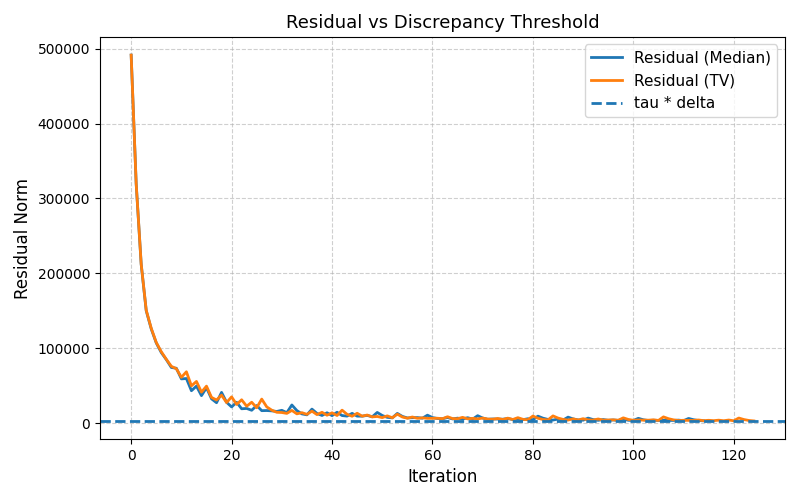}

\vspace{0.3cm}

\includegraphics[width=0.49\textwidth]{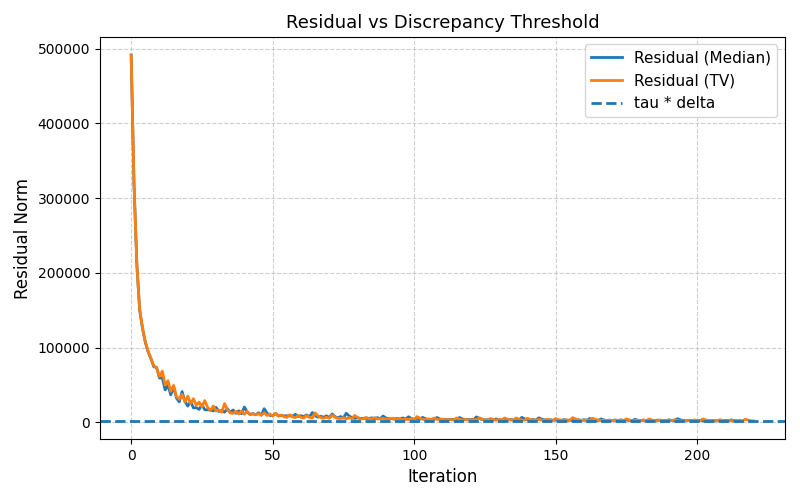}
\hfill
\includegraphics[width=0.49\textwidth]{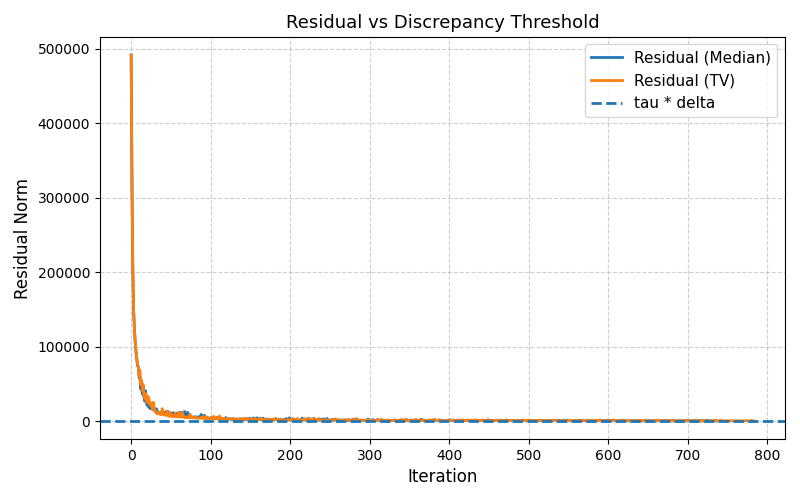}

\vspace{0.3cm}

\includegraphics[width=0.49\textwidth]{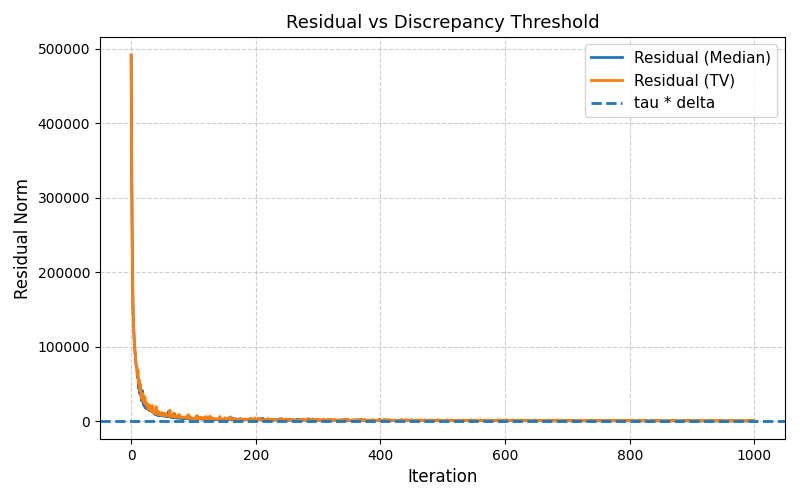}

\caption{Residual curves for phase retrieval CT across different noise levels. From top-left to bottom-right, the panels correspond to $\delta_{\mathrm{rel}} = 0.01, 0.005, 0.003, 0.001,$ and $0.0005$, respectively. }
\label{fig:PR_residuals}
\end{figure}

\end{document}